\documentclass[12pt, letterpaper]{amsart}

\usepackage{amsthm, amssymb, amsfonts, amscd}
\usepackage{array}
\usepackage[frame,cmtip,arrow,matrix,line,graph,curve]{xy}
\usepackage{graphpap, color, enumitem}
\usepackage[pdftex]{graphicx}
\usepackage[pdftex,colorlinks,citecolor=blue]{hyperref}
\usepackage{tikz}
\usepackage{tikz-cd}
\usepackage{verbatim}
\usepackage{mathrsfs, mathtools}
\usepackage{xcolor}
\usepackage{cleveref}
\usepackage{setspace}
\usepackage{diagbox}
\usepackage[style=alphabetic, url=false, doi=false, isbn=false]{biblatex}
\AtBeginBibliography{\small}
\AtBeginDocument{%
  \raggedbottom
  \setlength{\abovedisplayskip}{5pt}%
  \setlength{\belowdisplayskip}{5pt}%
  \setlength{\abovedisplayshortskip}{3pt}%
  \setlength{\belowdisplayshortskip}{3pt}%
}

\allowdisplaybreaks[4]
\makeatletter
\@namedef{subjclassname@2020}{%
\textup{2020} Mathematics Subject Classification}
\makeatother

\newtheorem{theorem}{Theorem}[section]
\newtheorem{theoremletter}{Theorem}

\newtheorem{proposition}[theorem]{Proposition}
\newtheorem{corollary}[theorem]{Corollary}
\newtheorem{lemma}[theorem]{Lemma}

\newtheorem{conjecture}[theorem]{Conjecture}

\theoremstyle{definition}
\newtheorem{definition}[theorem]{Definition}
\newtheorem{example}[theorem]{Example}
\newtheorem{examples}[theorem]{Examples}

\newtheorem{remark}[theorem]{Remark}

\newtheorem{algorithm}[theorem]{Algorithm}
\newtheorem*{acknowledgement}{Acknowledgement}

\renewcommand{\AA}{\mathbb{A} }

\newcommand{\GG}{\mathbb{G} }

\newcommand{\PP}{\mathbb{P} }

\newcommand{\ZZ}{\mathbb{Z} }

\newcommand{\cC}{\mathcal{C} }
\newcommand{\cE}{\mathcal{E} }
\newcommand{\cF}{\mathcal{F} }
\newcommand{\cG}{\mathcal{G} }
\newcommand{\cH}{\mathcal{H} }

\newcommand{\cK}{\mathcal{K} }
\newcommand{\cL}{\mathcal{L} }
\newcommand{\cM}{\mathcal{M} }
\newcommand{\cN}{\mathcal{N} }
\newcommand{\cO}{\mathcal{O} }

\newcommand{\cQ}{\mathcal{Q} }
\newcommand{\cR}{\mathcal{R} }
\newcommand{\cS}{\mathcal{S} }
\newcommand{\cT}{\mathcal{T} }
\newcommand{\cU}{\mathcal{U} }

\newcommand{\ba}{\mathbf{a} }
\newcommand{\bA}{\mathbf{A} }

\newcommand{\bB}{\mathbf{B} }
\newcommand{\bc}{\mathbf{c} }

\newcommand{\bI}{\mathbf{I} }
\newcommand{\bp}{\mathbf{p} }
\newcommand{\bq}{\bsq }
\newcommand{\bm}{\mathbf{m} }
\newcommand{\bn}{\mathbf{n} }

\newcommand{\bsq}{\mathbf{q} }

\newcommand{\bT}{\mathbf{T} }

\newcommand{\bU}{\mathbf{U} }
\newcommand{\bV}{\mathbf{V} }
\newcommand{\bW}{\mathbf{W} }

\newcommand{\bX}{\mathbf{X} }

\newcommand{\rD}{\mathrm{D} }
\newcommand{\rH}{\mathrm{H} }
\newcommand{\rM}{\mathrm{M} }

\newcommand{\rU}{\mathrm{U} }

\newcommand{\hX}{{X_H} }

\newcommand{\tcS}{\widetilde{\cS} }

\newcommand{\Om}{\Omega}
\newcommand{\ga}{\gamma}
\newcommand{\lbd}{\lambda}

\newcommand{\pHom}{\mathrm{ParHom}}

\newcommand{\cEnd}{\mathcal{E}nd}
\newcommand{\cpHom}{\mathcal{P}ar\cHom}

\newcommand{\rpar}{\mathrm{par}}
\newcommand{\parmu}{\mu_{par}}
\newcommand{\pardeg}{\deg_{par}}
\def\codim{\mathrm{codim}}
\def\sym{\mathrm{Sym}}

\def\rHom{\mathrm{Hom} }
\def\cHom{\mathcal{H}om }
\def\rEnd{\mathrm{End} }

\def\OGr{\mathrm{OGr}}
\def\Ext{\mathrm{Ext} }
\def\pExt{\mathrm{ParExt} }
\def\Sym{\mathrm{Sym} }
\def\git{/\!/ }
\def\Pic{\mathrm{Pic} }

\def\im{\mathrm{im}\, }
\def\spec{\mathrm{Spec}\;}

\def\rank{\mathrm{rank}\;}
\def\depth{\mathrm{depth}\;}
\def\type{\mathrm{type}\;}
\def\nullity{\mathrm{nullity}\;}
\def\Cl{\mathrm{Cl}}
\def\ch{\mathrm{ch}}
\def\td{\mathrm{td}}
\newcommand{\Coh}{\mathsf{Coh}}

\makeatletter
\renewcommand{\part}{%
  \clearpage
  \thispagestyle{empty}
  \@startsection{part}{0}
  {\z@}{0pt}{0pt}
  {\centering\normalfont\Large\bfseries}}
\makeatother

\begin{document}
\title{Ulrich bundles on intersections of two quadrics} \date{\today}


\author{Jiwan Jung} \address{Jiwan Jung, Department of Mathematics, POSTECH, 77, Cheongam-ro, Nam-gu, Pohang-si, Gyeongsangbuk-do, 37673, Korea} \email{jiwanjung11@gmail.com}

\author{Kyoung-Seog Lee} \address{Kyoung-Seog Lee, Department of Mathematics, POSTECH, 77, Cheongam-ro, Nam-gu, Pohang-si, Gyeongsangbuk-do, 37673, Korea} \email{kyoungseog@postech.ac.kr}

\author{Han-Bom Moon} \address{Han-Bom Moon, Department of Mathematics, Fordham University, New York, NY 10023} \email{hmoon8@fordham.edu}

\begin{abstract}
We construct Ulrich bundles on smooth intersections of two quadrics. We determine all possible ranks and prove the existence of indecomposable Ulrich bundles of every allowable rank. To our knowledge, this gives the first such construction for a family of nonhomogeneous varieties of arbitrarily large dimension. We also study the moduli spaces of these bundles and relate them to moduli spaces of vector bundles on curves. For intersections of two even-dimensional quadrics, our results prove a conjecture of Eisenbud and Schreyer.
\end{abstract}

\maketitle

\section{Introduction}\label{sec:intro}

The study of Ulrich bundles has attracted considerable attention over the past two decades. Eisenbud and Schreyer introduced them in connection with the computation of Chow forms \cite{ES03}, and subsequent work has revealed connections with several other areas of algebraic geometry.

For example, let \(\mathrm{BS}(X)\) denote the Boij--S\"oderberg cone of cohomology tables of vector bundles on an \(n\)-dimensional variety \(X\). Although this countable-dimensional cone is generally difficult to determine, the existence of an Ulrich bundle implies \(\mathrm{BS}(X) = \mathrm{BS}(\PP^n)\) \cite{ES11}, and the latter cone was characterized in \cite{ES09}.

Let \(X\) be a smooth projective variety equipped with a very ample line bundle \(\cO_X(1)\). A vector bundle \(\cE\) on \(X\) is \textbf{Ulrich} if
\[\rH^\bullet(X, \cE(-j)) = 0\]
for all \(1 \le j \le \dim X\). Eisenbud and Schreyer conjectured that every smooth polarized projective variety admits an Ulrich bundle \cite{ES03}. The strength of the required vanishings makes explicit constructions difficult, particularly in dimensions greater than three and outside the homogeneous setting. Beyond existence, two natural questions concern the possible ranks--especially the minimum rank, or \textbf{Ulrich complexity} \(uc(X)\)--and the geometry of the corresponding moduli spaces. See \cite{CFK23, CFK24, CMRPL21} and references therein. Very recently, Anghel announced a counterexample to the above Eisenbud-Schreyer conjecture for a certain surface of general type \cite{Ang26}, hence there is an example with \(uc(X) = \infty\).

\subsection{Intersection of two quadrics}

The main object of this paper is an intersection of two quadrics. This is the first family of nonhomogeneous varieties of unbounded dimension on which Ulrich bundles are studied systematically. In \cite{ES25}, Eisenbud and Schreyer considered the intersection \(X^{2g-1}\) of two even-dimensional quadrics \(Q_0^{2g} \cap Q_{\infty}^{2g}\). Using matrix factorizations, they showed that \(uc(X^{2g-1}) = 2^{g-1}\) and that every Ulrich bundle has rank \(r2^{g-2}\) for some \(r \in \ZZ_{> 0}\). Furthermore, they made the following explicit conjecture.

\begin{conjecture}[\protect{\cite[Conjecture~1.2]{ES25}}]\label{conj:ESquadrics}
Let \(X^{2g-1} = Q_0^{2g} \cap Q_{\infty}^{2g} \subset \PP^{2g+1}\) be the smooth intersection of two quadrics for some \(g \ge 1\). Then there exists an indecomposable Ulrich bundle of rank \(r2^{g-2}\) if and only if \(r \ge 2\) and \(rg \equiv 0 \pmod{2}\).
\end{conjecture}

The aim of this paper is to extend their work to both even- and odd-dimensional quadrics and to higher-rank Ulrich bundles. Theorems \ref{thm:mainthmeven} and \ref{thm:mainthmodd} establish the existence of Ulrich bundles and describe their moduli spaces.

We begin with intersections of even-dimensional quadrics.

\begin{theoremletter}\label{thm:mainthmeven}
Let \(g \ge 2\) and let \(X = X^{2g-1} := Q_{0}^{2g} \cap Q_{\infty}^{2g} \subset \PP^{2g+1} = \PP V^{2g+2}\) be a smooth intersection of two even-dimensional quadrics.
\begin{enumerate}
  \item There is an indecomposable Ulrich bundle \(\cE\) on \(X\) if and only if \(\rank \cE = r2^{g-2}\) for some \(r \in \ZZ_{> 0}\) such that \(r \ge 2\) and \(rg \equiv 0 \pmod 2\). In particular, \(uc(X) = 2^{g-1}\).
  \item Let \(\cU_{r2^{g-2}}(X)\) be the moduli stack of rank \(r2^{g-2}\) semistable Ulrich bundles on \(X\). For any \(r\) satisfying the conditions in (1), there is an open embedding
  \[\cU_{r2^{g-2}}(X) \hookrightarrow \cM_C(r, r(g-2)/2),\]
  where the target is the moduli stack of rank \(r\), degree \(r(g-2)/2\) semistable bundles over a hyperelliptic curve \(C\) of genus \(g\).
\end{enumerate}
\end{theoremletter}

In particular, part (1) of Theorem~\ref{thm:mainthmeven} resolves Conjecture~\ref{conj:ESquadrics}. As an immediate corollary, we obtain:

\begin{corollary}\label{cor:evenintersection}
If \(g \ge 2\), the good moduli space \(\rU_{r2^{g-2}}(X)\) is an irreducible, normal, uniruled variety of dimension \(r^2(g-1)+1\).
\end{corollary}

When \(g = 1\), \(X\) is an elliptic curve and the theory of Ulrich bundles on \(X\) is classical. We still have (1) and a variation of (2) of Theorem~\ref{thm:mainthmeven} on the level of good moduli spaces. However, Corollary~\ref{cor:evenintersection} encounters an exception. For instance, if \(g = 1\) and \(r = 2\), \(\rM_C(2, -1) \cong C\) is an elliptic curve, which is not uniruled. See Remark~\ref{rmk:genus1}.

We also obtain the parallel result for odd-dimensional quadrics, completing the existence problem for Ulrich bundles on intersections of two quadrics.

\begin{theoremletter}\label{thm:mainthmodd}
Let \(g \ge 2\) and let \(X = X^{2g-2} := Q_{0}^{2g-1} \cap Q_{\infty}^{2g-1} \subset \PP^{2g} = \PP V^{2g+1}\) be a smooth intersection of two odd-dimensional quadrics.
\begin{enumerate}
  \item There is an indecomposable Ulrich bundle \(\cE\) on \(X\) if and only if \(\rank \cE = r2^{g-3}\) for some \(r\in\ZZ_{\ge2}\) and \(r(g+1) \equiv 0 \pmod{2}\). In particular, \(uc(X) = 2^{g-2}\).
  \item Let \(\cU_{r2^{g-3}}(X)\) be the moduli stack of rank \(r2^{g-3}\) Ulrich bundles on \(X\). It has multiple connected components. For any \(r\) satisfying the conditions in (1), there is an open embedding
  \[\cU_{r2^{g-3}}(X) \hookrightarrow \cM_{\PP_{2g+1}^1}(r, r(g-3)/4),\]
  where \(\cM_{\PP_{2g+1}^1}(r, r(g-3)/4)\) denotes the moduli stack of rank \(r\), degree \(r(g-3)/4\) semistable bundles over \(\PP_{2g+1}^1\).
\end{enumerate}
\end{theoremletter}

The moduli stack \(\cM_{\PP_{2g+1}^1}(r,d)\) in general has many connected components. Each component is isomorphic to \(\cM_{(\PP^1, \bp)}(r, d', \bm)\), the moduli stack of parabolic semistable vector bundles with a certain degree \(d'\) and multiplicity \(\bm\).

The geometry of moduli spaces of parabolic bundles over \(\PP^1\) has been studied extensively; see, for example, \cite{MY21}. Therefore, we obtain:

\begin{corollary}\label{cor:oddintersection}
Let \(\rU_{r2^{g-3}}(X)\) be the good moduli space of \(\cU_{r2^{g-3}}(X)\). A connected component \(U\) of \(\rU_{r2^{g-3}}(X)\) in Theorem~\ref{thm:mainthmodd} is a normal rational variety.
\end{corollary}

In \S\ref{sec:existence}, we construct explicit parabolic data defining several connected components. In rank-two, this produces at least \(\binom{2g+1}{g+1}\) components of the same dimension. Higher ranks exhibit components of different dimensions; see Remark~\ref{rmk:dimension}.

Our proof does not rule out the possibility of the existence of extra components. However, based on some numerical calculations in low-rank, low-genus cases, we suspect that the only connected components of \(\rU_{2\cdot2^{g-3}}(X)\) are degree-zero twists of those constructed in the proof of Theorem~\ref{thm:mainthmodd}. In general, we do not know how many connected components of \(\rU_{r2^{g-3}}(X)\) exist.

\subsection{Known results}

These results are already known in several low-dimensional cases by different methods. For \(g = 1\), \(X^1 = Q_0^2 \cap Q_\infty^2 \subset \PP^3\) is an elliptic curve, and classical Brill--Noether theory provides Ulrich bundles. See Remark~\ref{rmk:genus1}. The intersection \(X^2 = Q_0^3 \cap Q_\infty^3 \subset \PP^4\) is a del Pezzo surface of degree four, and Ulrich bundles on it have been studied extensively in \cite{CKM13, MRPL14}. The next case, \(X^3 = Q_0^4 \cap Q_\infty^4 \subset \PP^5\), is a Fano threefold of index two. Theorem~\ref{thm:mainthmeven} in this case was obtained in \cite{CKL21}.

\subsection{Outline of proof}

The proof of Theorems \ref{thm:mainthmeven} and \ref{thm:mainthmodd} consists of several reductions.

Homological projective duality provides an explicit semiorthogonal decomposition of \(\rD^b(X)\) \cite{Kuz07}. Depending on the parity of \(\dim X\), the Kuznetsov component \(\cC\) is either a hyperelliptic curve \(C\) or a root stack \(\PP_{2g+1}^1\) \cite{Kuz08}; see Theorem~\ref{thm:SOD}. \S\ref{sec:spinorbundle} describes the relative spinor bundle \(\cS\) on \(\cC \times X\), which is the Fourier--Mukai kernel of the embedding \(\rD^b(\cC) \to \rD^b(X)\). A standard cohomological computation then reduces the construction of an Ulrich bundle on \(X\) to finding a vector bundle \(\cG\) on \(\cC\) that is cohomologically orthogonal to an explicit bundle \(\cR\), the so-called \textbf{Raynaud bundle}. This reduction is carried out in \S\S\ref{sec:Raynaud} -- \ref{sec:reduction}.

The reduction in \S\ref{sec:reduction} is useful for two reasons. Vector bundles on a one-dimensional stack are easier to study, and the correspondence preserves stability and Jordan--H\"older filtrations. In particular, an extension argument on \(\cC\) constructs higher-rank Ulrich bundles from lower-rank ones. It is therefore enough to construct orthogonal bundles of ranks two and three. Moreover, the case of \(\PP_{2g+1}^1\) implies the statement for \(C\), as shown in \S\ref{sec:existenceUlrich}.

By the work of Borne and Biswas-Dhillon in \cite{Bor07, BD11}, the category of vector bundles on a root stack \(\PP_{2g+1}^1\) is equivalent to the category of parabolic bundles on its coarse moduli space \(\PP^1\). Thus, finding a vector bundle orthogonal to \(\cR\) over \(\PP_{2g+1}^1\) is equivalent to constructing a parabolic bundle orthogonal to a certain parabolic bundle \(\cR\) over \((\PP^1, \bp)\). This parabolic bundle has a concrete description in terms of the relative Clifford algebra, and the parabolic bundle structure on \(\cR\) can be described explicitly in terms of Clifford multiplication. This is achieved in \S\ref{sec:reductionparabolic}.

The condition that a parabolic bundle \(\cG\) over \(\PP^1\) be orthogonal to \(\cR\) gives a concrete set of linear equations on the vector space \(\rHom(\cG, \cR)\). The question is therefore one of finite-dimensional linear algebra. \S\ref{sec:orthogonality} describes this linear algebra problem. Proving the invertibility of the resulting square matrix requires a nontrivial combinatorial computation. \S\ref{sec:existence} and \S\ref{sec:prf_main_thm} prove the required invertibility through a sequence of matrix reductions and a change of coordinates.

% \subsection*{Notation and conventions}
Throughout this paper, unless otherwise specified, \(g \ge 2\). The base field \(k\) is an algebraically closed field with \(\operatorname{char} k=0\).

\acknowledgement

JWJ was partially supported by the National Research Foundation of Korea(NRF) grant funded by the Korea government(MSIT) (RS-2024-00394189) and the POSCO Science Fellowship of POSCO TJ Park Foundation and the National Research Foundation of Korea (NRF) grant funded by the Korea government (MSIT) (No. RS-2026-25508638). KSL was partially supported by Basic Science Research Institute Fund, whose NRF grant number is RS -2021-NR060139, the National Research Foundation of Korea(NRF) grant funded by the Korea government(MSIT) (RS-2024-00394189) and the POSCO Science Fellowship of POSCO TJ Park Foundation. He also thanks Fordham University for hospitality during his visit. HBM is partially supported by AMS-Simons Research Enhancement Grant for PUI Faculty. He also thanks POSTECH and KIAS for their hospitality during his visit.

\medskip

\paragraph*{\textbf{Use of generative AI}} 
The authors conceived and developed all the mathematical arguments and wrote the proofs and the manuscript. They are solely responsible for the final text. After the mathematical arguments had been completed, ChatGPT 5.6 was used to assist with the preparation of the paper, including grammatical corrections, checking the logical consistency of the arguments, conducting literature searches, and constructing the explicit supporting examples in \cite{JLM26}.

\tableofcontents

%%%%%%%%%%%%%%%%%%%%%%%%%%%%%%%%%%%%%%%%%%%%%%%%
\section{Preliminaries}\label{sec:preliminary}

\subsection{Ulrich bundles}

We recall several standard characterizations and basic properties of Ulrich bundles.

\begin{definition}[{\cite[Proposition~2.1]{ES03}, \cite[Chapter~3]{CMRPL21}}]\label{def:Ulrich}
Let \(X \subset \PP V\) be an \(n\)-dimensional integral projective variety with a fixed embedding given by a very ample line bundle \(\cO_X(1)\). Let \(d\) be the degree of \(X\) and \(c = \codim \; X\).

A coherent sheaf \(\cE\) of rank \(r\) on \(X\) is called \textbf{Ulrich} if it satisfies any, hence all, of the following equivalent conditions:
\begin{enumerate}[label=\textup{(\roman*)}]
  \item \(\rH^\bullet(X,\cE(-j))=0\) for all \(1\le j\le n\).
  \item \(\rH^i(X,\cE(t))=0\) for all \(0<i<n\) and \(t\in\ZZ\), and \(\rH^0(X,\cE(-1))=0,\ h^0(X,\cE)=rd\).
  \item If \(\imath:X\hookrightarrow\PP V\) is the given embedding, then \(\imath_*\cE\) has a linear minimal locally free resolution on \(\PP V\):
  \[0\to \cO_{\PP V}(-c)^{\oplus b_{c}} \to\cdots\to \cO_{\PP V}(-1)^{\oplus b_1} \to \cO_{\PP V}^{\oplus b_0} \to\imath_*\cE\to0.\]
  \item For a general finite linear projection \(\pi:X\to\PP^n\), one has \(\pi_*\cE\cong\cO_{\PP^n}^{\oplus rd}\).
  \item The graded module
  \[\rH^0_*(X,\cE):=\bigoplus_{m\in\ZZ}\rH^0(X,\cE(m))\]
  is generated in degree \(0\) by \(rd\) elements and has a linear minimal free resolution over the homogeneous coordinate ring of \(\PP V\).
\end{enumerate}
\end{definition}

When \(X\) is smooth, the standard numerical consequences of the Ulrich condition give
\begin{equation}\label{eqn:Ulrichnumeric}
b_j=rd\binom{c}{j}, \quad\chi(X,\cE(t))=rd\binom{t+n}{n}, \quad \deg \cE = \frac{r((n+1)H + K_X)H^{n-1}}{2},
\end{equation}
where \(H\) is a hyperplane class. See \cite[Corollary~2.2]{ES03} and \cite[Proposition~3.2.5, Corollary~3.2.10]{CMRPL21}.

\begin{remark}
By \cite[\S2.1]{CMRPL21}, every positive-rank Ulrich sheaf on a smooth variety \(X\) is locally free. Thus, for a smooth \(X\), we use Ulrich sheaf and Ulrich bundle interchangeably.
\end{remark}

\begin{remark}
Ulrich bundles are semistable with respect to \(\cO_X(1)\), and the Jordan--H\"older factors of a strictly semistable Ulrich bundle are again Ulrich; see \cite[Proposition~3.3.14]{CMRPL21}.
\end{remark}

\begin{examples}
\begin{enumerate}[label=\textup{(\arabic*)}]
  \item On \(\PP^n\), the only Ulrich bundles are \(\cO_{\PP^n}^{\oplus r}\).
  \item On a rational normal curve \(\nu_d:\mathbb{P}^1 \cong C\hookrightarrow\mathbb{P}^d\), the line bundle on \(C\) associated to \(\cO_{\mathbb{P}^1}(d-1)\) is Ulrich \cite[\S4.1.1]{Cos17}.
  \item On a smooth quadric, the classical spinor bundles are Ulrich: there is one spinor bundle in odd dimension and two half-spinor bundles in even dimension; see \cite{Ott88}.
  \item Stable Ulrich bundles are known in several classes of Fano threefolds: in every rank \(r \ge 2\) on cubic threefolds \cite{LMS15}; in every rank \(r\ge2\) on del Pezzo threefolds \cite{CFK23}; and in every even rank on prime Fano threefolds of index one \cite{CFK24}.
  \item For a smooth intersection of two \(4\)-dimensional quadrics, Cho--Kim--Lee constructed Ulrich bundles of every rank \(r\ge2\) and described the moduli space of stable Ulrich bundles \cite{CKL21}.
  \item Suppose that \(\cE\) is an Ulrich bundle on \(X \subset \PP V\). Then for any smooth hyperplane section \(X_H := X \cap H \subset H\), \(\cE|_{X_H}\) is an Ulrich bundle on \(X_H\).
\end{enumerate}
\end{examples}

\begin{example}\label{ex:intquadrics}
Let \(X\subset\PP^{n+2}\) be a smooth intersection of two quadrics. Then \(d=4\). \textit{If} there is a rank \(r\) Ulrich bundle \(\cE\) on \(X\), it has a resolution
\[0\to \cO_{\PP^{n+2}}(-2)^{\oplus 4r} \to \cO_{\PP^{n+2}}(-1)^{\oplus 8r} \to \cO_{\PP^{n+2}}^{\oplus 4r} \to \imath_*\cE\to 0.\]
In particular, \(h^0(X,\cE)=4r\). Since \(\omega_X\cong\cO_X(1-n)\), the degree formula in \eqref{eqn:Ulrichnumeric} gives \(\mu(\cE)=4\), or equivalently, \(\mu(\cE^*(1))=0\).
\end{example}

\subsection{Intersection of quadrics and common isotropic subspaces}\label{ssec:isotropicspace}

Let \(V\) be a \(k\)-vector space of dimension \(n\). Let \(q_0\) and \(q_\infty\) be two nondegenerate quadratic forms on \(V\). They define two smooth quadrics \(Q_0\) and \(Q_\infty\) in \(\PP V \cong \PP^{n-1}\). Assume that they intersect transversely, so that \(X := Q_0 \cap Q_\infty\) is an \(n-3\)-dimensional smooth projective variety.

Fix a basis \(\{e_1,\ldots,e_{n}\}\) of \(V\), and let \(\{x_1,\ldots,x_{n}\}\) be the dual basis of \(V^*\). The homogeneous coordinate ring of \(X\) is \(P_X:=k[V^*]/(q_0,q_\infty)\), where \(q_0,q_\infty\in k[V^*]_2\) are linearly independent. After a change of basis, we may assume
\begin{equation}\label{eqn:quadforms}
q_0=\sum_{i=1}^{n}x_i^2,\qquad q_\infty=\sum_{i=1}^{n}\lambda_i x_i^2,
\end{equation}
where the scalars \(\lambda_i\in k^*\) are pairwise distinct; this entails no loss of generality for a smooth intersection of two quadrics \cite[Proposition~2.1]{Rei73}.

Fix a maximal isotropic subspace \(U\subset V\), that is, a maximal-dimensional subspace such that \(\PP U\subset X\). If \(\dim V\) is even (resp. odd), \(\dim U = (n-2)/2\) (resp. \((n-1)/2\)). Set
\begin{equation}\label{def:PU}
P_U:=k[V^*]/(U^\perp),\qquad U^\perp:=\ker(V^*\xrightarrow{\mathrm{res}}U^*).
\end{equation}
For later use, write \(\overline{(\cdot)}:V\to V/U\) for the quotient map, and let \(\sigma_i \in \rEnd(V)\) denote the reflection of the \(i\)-th coordinate \(\sigma_i(e_j) = (-1)^{\delta_{ij}}e_j\). When \(\dim V\) is odd, these reflections act transitively on the set of maximal isotropic subspaces in \(X\) \cite[Theorem~3.8]{Rei73}.

\subsection{Clifford representation and spinor bundles on a quadric}\label{ssec:spinor}

Let \(q\) be a quadratic form on \(V\). The \textbf{Clifford algebra} of \((V, q)\) is \(\Cl_q := \bigoplus_{d \ge 0}V^{\otimes d}/\langle v \otimes v - q(v)\cdot1\mid v \in V\rangle\). This is a \(\ZZ/2\ZZ\)-graded algebra. We denote its even (resp. odd) part by \(\Cl_{q, 0}\) (resp. \(\Cl_{q, 1}\)).

If the quadratic form \(q\) is nondegenerate, \(\Cl_q\) and \(\Cl_{q,0}\) are semisimple. Suppose that \(\dim V = 2g + 2\). Then \(\Cl_q\) is isomorphic to a matrix algebra of size \(2^{g+1}\), and \(\Cl_{q,0}\) is a product of two matrix algebras of size \(2^g\). If \(\dim V = 2g+1\), then \(\Cl_q\) is a product of two matrix algebras of size \(2^g\), and \(\Cl_{q,0}\) is a matrix algebra of size \(2^g\).

We need the following one-parameter family of Clifford algebras. Define the \(\ZZ\)-graded Clifford algebra \(\Cl\) associated with the pencil \(sq_0+tq_\infty\) by
\begin{equation}\label{eqn:defCl}
\begin{split}
  \Cl:=&\left(k[s,t]\otimes_k\bigoplus_{d\ge0}V^{\otimes d}\right)/\langle v\otimes v-sq_0(v)-tq_\infty(v)\ |\ v\in V\rangle,\quad\deg s=\deg t=2\\
  =& \; k \oplus V \oplus \left(\wedge^2 V \oplus \langle s, t\rangle\right) \oplus\left(\wedge^3 V \oplus V \otimes \langle s, t\rangle\right) \\
  &\quad\quad \oplus \left(\wedge^4 V \oplus \wedge^2 V \otimes \langle s, t\rangle \oplus \langle s^2, st, t^2\rangle\right) \oplus \dots.
\end{split}
\end{equation}
By Koszul duality \cite[Theorem~3.1]{ES25}, there are graded algebra isomorphisms
\[\Cl\cong\Ext^\bullet_{P_X}(k,k),\qquad P_X\cong\Ext^\bullet_{\Cl}(k,k).\]

We review a construction of the spinor sheaves on a possibly singular quadric, given by Addington \cite{Add11}. Let \(q\) be a quadratic form on \(V\) and let \(Q\) be the associated quadric in \(\PP V\).

Let \(W \subset V\) be an isotropic subspace with \(\codim\; W > 1\), and fix a basis \(\{w_1, w_2, \cdots, w_m\}\) of \(W\). Let \(I\) be the right \(\Cl_q\)-ideal generated by \(w_1 w_2 \cdots w_m\). The ideal \(I\) is independent of the choice of basis of \(W\). Since \(\Cl_q\) is \(\ZZ/2\ZZ\)-graded, write \(I = I_0 \oplus I_1\). Then \(\dim I_0 = \dim I_1\) as \(k\)-vector spaces. Assume further that \(q\) is nondegenerate and \(W\) is a maximal isotropic subspace. If \(\dim V = 2g+2\), then \(\dim W = g+1\), and \(I\) is the unique nontrivial irreducible representation of \(\Cl_q\), of dimension \(2^{g+1}\). Similarly, if \(\dim V = 2g+1\), then \(\dim W = g\), and \(I_0 \cong I_1\) is the unique irreducible representation of \(\Cl_{q,0}\).

Over \(\PP V\), the tautological map \(\cO_{\PP V}(-1) \to \cO_{\PP V}\otimes V\) induces two maps
\begin{equation}\label{eqn:tautologicalmaps}
I_0 \otimes \cO_{\PP V}(-1) \stackrel{\varphi}{\to} I_1 \otimes \cO_{\PP V}, \qquad I_1 \otimes \cO_{\PP V}(-1) \stackrel{\psi}{\to} I_0 \otimes \cO_{\PP V}.
\end{equation}
Indeed, the cokernels of \(\varphi, \psi\) are supported on \(Q\) and \(\varphi, \psi\) are injective because they give a matrix factorization of \(q\). Thus, we obtain exact sequences
\begin{equation}\label{eqn:spinorresolution}
0 \to I_0 \otimes \cO_{\PP V}(-1) \stackrel{\varphi}{\to} I_1 \otimes \cO_{\PP V} \to \tcS_q \to 0, \quad 0 \to I_1 \otimes \cO_{\PP V}(-1) \stackrel{\psi}{\to} I_0 \otimes \cO_{\PP V} \to \widetilde{\cT}_q \to 0.
\end{equation}
It is known that the supports of \(\tcS_q\) and \(\widetilde{\cT}_q\) are \(Q\), and they are rank-\(2^{\codim\; W - 2}\) vector bundles on \(Q \setminus (Q^{\mathrm{sing}} \cap \PP W)\) \cite[Proposition~2.1]{Add11}. The sheaves \(\tcS_q\) and \(\widetilde{\cT}_q\) are called \textbf{spinor sheaves}.

\begin{remark}\label{rmk:Ulrichquadric}
The existence of a locally free resolution in \eqref{eqn:spinorresolution} implies that \(\tcS_q\) and \(\widetilde{\cT}_q\) are Ulrich sheaves on \(Q\). See (iii) of Definition \ref{def:Ulrich}.
\end{remark}

We are primarily interested in the case where \(W\) is maximal. Suppose \(q\) is nondegenerate. If \(\dim V = 2g+2\), then \(\tcS_q\) and \(\widetilde{\cT}_q\) are two classical \textbf{spinor bundles}. If \(\dim V = 2g+1\), then \(\tcS_q \cong \widetilde{\cT}_q\) is the unique classical spinor bundle. Note that if \(q\) has corank one, then there is a unique singular point \(p\) on \(Q\), and \(\tcS_q\) and \(\widetilde{\cT}_q\) are vector bundles away from \(p\).

\subsection{Root stack}\label{ssec:rootstack}

Let \(Y\) be a scheme. The category of pairs \((\cL, s)\), where \(\cL\) is an invertible sheaf on \(Y\) and \(s\) is a section of \(\cL\), is equivalent to the category of morphisms \(Y \to [\AA^{1}/\GG_{m}]\). Such a morphism is represented by a diagram
\[\xymatrix{P \ar[r]^{\rho} \ar[d] & \AA^{1}\\ Y}\]
where \(\rho\) is a \(\GG_{m}\)-equivariant morphism. Then we obtain a \(\GG_{m}\)-invariant morphism
\[(\mathrm{id}, \rho) : P \to P \times \AA^{1} \to P \times^{\GG_{m}}\AA^{1},\]
which gives a morphism \(s : Y \to P \times^{\GG_{m}}\AA^{1}\) corresponding to a section of \(\cL\). The pair \((\cL, s)\) is called a \textbf{generalized effective Cartier divisor}.

The morphism \(\AA^{1} \to \AA^{1}\) given by \(t \mapsto t^{r}\) induces a morphism of stacks
\[\theta_{r} : [\AA^{1}/\GG_{m}] \to [\AA^{1}/\GG_{m}].\]

\begin{definition}\label{def:rootstack}
The \textbf{root stack} \(Y_{(\cL, s, r)}\) is the fiber product
\[\xymatrix{Y_{(\cL, s, r)} \ar[r] \ar[d] & [\AA^{1}/\GG_{m}] \ar[d]^{\theta_{r}}\\ Y \ar[r]^{(\cL, s)} & [\AA^{1}/\GG_{m}]}.\]
\end{definition}

The following example provides an explicit local description of a root stack as a quotient stack.

\begin{example}\label{ex:localrootstack}
Suppose \(Y = \spec R\), \(\cL = \cO_{Y}\), and \(s = f \in \Gamma(Y, \cL) = R\). Then
\[Y_{(\cL, s, r)} = [\spec(R[t]/(t^{r}-f))/\mu_{r}].\]
\end{example}

More generally, we obtain the following global result.

\begin{proposition}[\protect{\cite[Corollary~3.8]{Bor07}}]\label{prop:rootstackasaquotientstack}
Let \(\pi : Z \to Y\) be a uniform cyclic covering of degree \(r\). Let \(D\) (resp. \(E\)) be the branch divisor (resp. ramification divisor), and let \(s\) (resp. \(t\)) be the canonical section of \(\cL = \cO_{Y}(D)\) (resp. \(\cN = \cO_{Z}(E)\)). We have an isomorphism \(\cN^{\otimes r} \cong \pi^{*}\cL\). Then the morphism \(Z \to Y_{(\cL, s, r)}\) induces an isomorphism of \(Y\)-stacks
\[[Z/\mu_{r}] \cong Y_{(\cL, s, r)}.\]
\end{proposition}

For our purpose, the following concrete example is important.

\begin{example}\label{ex:P1r}
For distinct \(\lambda_1, \lambda_2, \dots, \lambda_m \in k\), let \(p_{\ell} := [-\lambda_\ell : 1] \in \PP^1\) and \(D := p_1 + \dots + p_m\). Denote by \(\PP^{1}_{m}:=\PP^{1}_{(\cL_{m},s,2)}\) the root stack associated with the divisor \(D\), where \(\cL_{m} = \cO(D) \cong \cO(m)\) and \(s \in \rH^{0}(\PP^1, \cL_m)\) is the canonical section satisfying \(\mathrm{div} \;s = D\).

When \(m\) is even, Proposition~\ref{prop:rootstackasaquotientstack} applies. Let \(C_{g}\) be a hyperelliptic curve of genus \(g\) with \(m = 2g+2\) ramification points \(q_{i}=\pi^{-1}([-\lambda_{i}:1])\), where \(\pi : C_{g} \to \PP^{1}\) is the quotient map by the hyperelliptic involution \(\iota\). The proposition gives
\[[C_{g}/\langle \iota\rangle ] \cong \PP^{1}_{2g+2}.\]
\end{example}

\begin{lemma}[\protect{\cite[Lemma~3.10]{Bor07}}]
Suppose that \(D\) and \(D'\) are two disjoint divisors, \(\cL = \cO_{Y}(D)\), \(\cL' = \cO_{Y}(D')\), and \(s\) and \(s'\) are the canonical sections of \(\cL\) and \(\cL'\), respectively. Then there is a canonical isomorphism
\[Y_{(\cL, s, r)}\times_{Y}Y_{(\cL', s', r)} \cong Y_{(\cL \otimes \cL', s \otimes s', r)}.\]
\end{lemma}

In particular, we obtain a sequence of stacks
\[\dots\rightarrow\PP^{1}_{r+1}\rightarrow\PP^{1}_{r}\rightarrow\PP^{1}_{r-1}\rightarrow\dots\rightarrow\PP^{1}_{0}=\PP^{1}.\]
Each morphism can be understood as a relative coarse moduli space map.

\subsection{Sheafifications of the graded Clifford algebra}\label{ssec:sheafification}

The \(\ZZ\)-graded Clifford algebra \(\Cl\) from \eqref{eqn:defCl} can be sheafified as an \(\cO_{\PP^1}\)-module
\[
\cC\ell := \cO_{\PP^1} \oplus V \otimes \cO_{\PP^1} \oplus \left(\wedge^2 V \otimes \cO_{\PP^1} \oplus \cO_{\PP^1}(1)\right) \oplus \left(\wedge^3 V \otimes \cO_{\PP^1} \oplus V \otimes \cO_{\PP^1}(1) \right) \oplus \dots,
\]
which is a locally free sheaf of Clifford algebras. This algebra is \(\mathfrak{B}_\sigma\) in \cite[\S3.1]{Kuz08}. By a version of Serre correspondence, a quotient of the category of graded \(\cC\ell\)-modules is equivalent to the category of \(\cC\ell_0\)-modules \cite[Proposition~3.7]{Kuz08}, where
\[\cC\ell_0 := \cO_{\PP^1} \oplus \wedge^2 V \otimes \cO_{\PP^1}(-1) \oplus \wedge^4 V \otimes \cO_{\PP^1}(-2) \oplus \dots.\]

Depending on the parity of \(V\), we may associate the following two variants of the sheafified Clifford algebra. Recall that we used \(q_0\) and \(q_\infty\) of \eqref{eqn:quadforms} in the definition of \(\Cl\) in \eqref{eqn:defCl}. In particular, the associated pencil has \(\dim V\) distinct points parametrizing corank-one singular quadrics.

Suppose that \(\dim V = 2g+2\). Let \(\pi : C \to \PP^1\) be the associated hyperelliptic curve of genus \(g\), ramified over the \(2g+2\) points parametrizing singular quadrics. There is a sheaf of even Clifford algebras \(\widetilde{\cC\ell}_0\) on \(C\) such that \(\pi_*\widetilde{\cC\ell}_0 = \cC\ell_0\), and \(\widetilde{\cC\ell}_0\) is also a sheaf of Azumaya algebras \cite[Proposition~3.13]{Kuz08}. Furthermore, the category of \(\cC\ell_0\)-modules \(\mathsf{Coh}(\PP^1, \cC\ell_0)\) and that of \(\widetilde{\cC\ell}_0\)-modules \( \mathsf{Coh}(C, \widetilde{\cC\ell}_0)\) are equivalent \cite[Corollary~3.13]{Kuz08}.

If \(\dim V = 2g + 1\), the root stack \(\PP_{2g+1}^1\) plays the role of \(C\). For the coarse moduli map \(\pi : \PP_{2g+1}^1 \to \PP^1\), there is a sheaf of Clifford algebras \(\widetilde{\cC\ell}_0\) on \(\PP_{2g+1}^1\) such that \(\pi_*\widetilde{\cC\ell}_0 = \cC\ell_0\). It has the same structural properties: the locally free sheaf \(\widetilde{\cC\ell}_0\) is a sheaf of Azumaya algebras \cite[Proposition~3.15]{Kuz08} over \(\PP_{2g+1}^1\), and \(\mathsf{Coh}(\PP^1, \cC\ell_0) \cong \mathsf{Coh}(\PP_{2g+1}^1, \widetilde{\cC\ell}_0)\) \cite[Corollary~3.16]{Kuz08}.

For notational simplicity, we use the following notation.

\begin{definition}\label{def:curlyC}
If \(\dim V = 2g+2\), or equivalently if \(\dim X = 2g-1\), we denote the hyperelliptic curve \(C\) of genus \(g\) above by \(\cC\). If \(\dim V = 2g+1\), or equivalently if \(\dim X = 2g-2\), \(\cC\) denotes the stack \(\PP_{2g+1}^1\).
\end{definition}

\subsection{Parabolic bundles}\label{ssec:parabolicbundle}

In this section, we recall the relation between vector bundles on a root stack and parabolic bundles on its coarse moduli space. For notational simplicity, we introduce a specialized version of the definition.

\begin{definition}\label{def:parabolicbundle}
Let \(C\) be a smooth projective curve and \(\bp = (p_1, p_2, \dots, p_n)\) be \(n\) distinct points of \(C\). A rank \(r\) \textbf{parabolic bundle} \(\cE := (\underline{\cE}, \{V_i\})\) consists of the following data:
\begin{enumerate}
  \item \(\underline{\cE}\) is a rank \(r\) vector bundle on \(C\);
  \item for each \(1\le i \le n\), \(V_i \subset \underline{\cE}|_{p_i}\) is a subspace.
\end{enumerate}
We set \(m_i := \dim V_i\). The sequence \(\bm = (m_1, m_2, \dots, m_n)\) is the combinatorial type of \(\cE\).
\end{definition}

The \textbf{parabolic degree} of \(\cE\) is defined as
\[\pardeg (\cE) := \deg \underline{\cE} + \frac{1}{2}\sum_{i=1}^n m_i.\]
Then we define the \textbf{parabolic slope} \(\parmu (\cE) := \pardeg (\cE)/r\).

\begin{remark}
More generally, one may choose parabolic weights \(\ba = (a_1, a_2, \dots, a_n)\) with \(0 < a_i < 1\). In this paper, \(a_i = 1/2\) for all \(i\).
\end{remark}

For any subbundle \(\underline{\cF} \subset \underline{\cE}\), we obtain an induced parabolic subbundle \(\cF = (\underline{\cF}, \{W_i\})\) by setting \(W_i := \underline{\cF}|_{p_i} \cap V_i\). We say a parabolic bundle \(\cE\) is \textbf{(semi-)stable} if for any proper parabolic subbundle \(\cF\),
\[\parmu(\cF) \;(\le) < \parmu(\cE).\]

Let \(\cC\) be a one-dimensional smooth root stack with \(n\) stacky points \(\bq = (q_1, q_2, \dots, q_n)\). Suppose that each \(q_i\) is a \(\ZZ/2\ZZ\)-point. Let \(\pi : \cC \to C\) be the coarse moduli map and \(\pi(q_i)=p_i\).

The category of vector bundles on \(\cC\) is equivalent to the category of parabolic bundles on \(C\) \cite[Theorem~3.13]{Bor07}. Set \(\cN := \cO_{\cC}(\sum q_i)\). Then \(\cN\) is a line bundle on \(\cC\) such that \(\cN^{\otimes 2} \cong \pi^*\cO_{C}(\sum p_i)\). For any vector bundle \(\widetilde{\cE}\) on \(\cC\), the corresponding parabolic bundle \(\cE = (\underline{\cE}, \{V_i\})\) on \(C\) is obtained as follows. First, set
\[\underline{\cE} := \pi_*\widetilde{\cE}, \quad \underline{\cE}' := \pi_*(\widetilde{\cE} \otimes \cN^{-1}).\]
Then \(\underline{\cE}'\) is a non-saturated locally free subsheaf of \(\underline{\cE}\). Set
\[V_i := \im (\underline{\cE}'|_{p_i} \to \underline{\cE}|_{p_i}).\]
\begin{remark}
More generally, the abelian category of coherent sheaves on \(\cC\) is equivalent to the category of parabolic sheaves on \(C\) \cite[Theorem~6.1]{BV12}. In particular, this exact equivalence induces an equivalence of bounded derived categories and identifies the corresponding cohomology groups.
\end{remark}

One may also define the degree of a bundle \(\widetilde{\cE}\) on \(\cC\) by setting \(\deg \pi^*\cL := \deg \cL\) and extending linearly; this degree may be rational. Then \(\deg \widetilde{\cE} = \pardeg \cE\) and \(\mu(\widetilde{\cE}) = \parmu(\cE)\) \cite[Theorem~4.3]{Bor07}. Moreover, the notions of (semi-)stability of \(\widetilde{\cE}\) and \(\cE\) coincide \cite[Remark~5.9]{Bor07}.

Here we discuss a particularly important example of this correspondence. Let \(\PP_{2g+1}^1\) be the one-dimensional root stack in Example \ref{ex:P1r}. Denote by \(\pi:\PP_{2g+1}^1\to\PP^1\) the coarse moduli map, and let \(q_1, q_2, \dots, q_{2g+1}\) be stacky points with \(\pi(q_i) = p_i\). Recall that \(\cN = \cO_{\PP_{2g+1}^1}(\sum q_i)\). Then \(\cN^{\otimes2}\cong \pi^*\cO_{\PP^1}(\sum p_i) \cong \pi^*\cO_{\PP^1}(2g+1)\), and the tautological section \(y\) of \(\cN\) satisfies
\[y^2=f(s,t):=\prod_{i=1}^{2g+1}(s+\lbd_i t).\]
Set
\[\cO_{\PP_{2g+1}^1}(1):=\cN\otimes \pi^*\cO_{\PP^1}(-g).\]
Then \(\cO_{\PP_{2g+1}^1}(1)^{\otimes2}\cong \pi^*\cO_{\PP^1}(1)\), and the associated section ring is
\[R_{\PP_{2g+1}^1}:=\bigoplus_{n\ge0}\rH^0\!\left(\PP_{2g+1}^1,\cO_{\PP_{2g+1}^1}(n)\right)\cong k[s,t,y]/(y^2-f(s,t)),\]
where \(\deg s=\deg t=2\), and \(\deg y=2g+1\). It is a normal domain, with \(f(s,t)\) square-free. It is also finite and free, hence finite flat, over \(k[s,t]\); moreover, the extension is integral.

On the other hand, recall that we have an associated \(\ZZ\)-graded Clifford algebra \(\Cl\) constructed in \S\ref{ssec:spinor}. The pseudoscalar \(e_1\dots e_{2g+1} \in \Cl\) lies in the center of \(\Cl\) and satisfies
\[(e_1\dots e_{2g+1})^2 = (-1)^g f(s,t).\]
Fix the identification \(y=\sqrt{-1}^{\,g}e_1\dots e_{2g+1}\). Then \(\Cl\) is naturally an \(R_{\PP_{2g+1}^1}\)-algebra.

For finitely generated graded \(R_{\PP_{2g+1}^1}\)-modules, we use the correspondence of \cite[Lemma~3.2]{HO20} between reflexive graded modules and vector bundles on the associated orbifold curve \(\PP_{2g+1}^1\). Thus, we need to check the reflexivity.

\begin{lemma}\label{lem:reflexive_R_module}
Let \(M\) be a finitely generated graded right \(\Cl\)-module which is free over \(k[s,t]\). Then \(M\) is reflexive as a graded \(R_{\PP_{2g+1}^1}\)-module.
\end{lemma}

\begin{proof}
It suffices to prove that \(M\) is torsion-free and satisfies \((S_2)\) over \(R_{\PP_{2g+1}^1}\) \cite[Tag~0AVT]{SP}.

Let \(0\ne a+by\in R_{\PP_{2g+1}^1}\), with \(a,b\in k[s,t]\). Its conjugate \(a-by\) is nonzero, and the norm \((a+by)(a-by) = a^2 - b^2f(s,t)\) is also nonzero. If \(m(a+by)=0\), then
\[0=m(a+by)(a-by)=m(a^2-b^2f(s,t)).\]
Since \(R_{\PP_{2g+1}^1}\) is a domain and \(M\) is free over \(k[s,t]\), this gives \(m=0\). Hence \(M\) is torsion-free over \(R_{\PP_{2g+1}^1}\).

Let \(\mathfrak p\subset R_{\PP_{2g+1}^1}\) be a prime ideal, and set \(\mathfrak q:=\mathfrak p\cap k[s,t]\). Since \(M_{\mathfrak q}\) is free over \(k[s,t]_{\mathfrak q}\), a regular system of parameters of \(k[s,t]_{\mathfrak q}\) remains \(M_{\mathfrak p}\)-regular after localization. Hence
\[
\depth_{R_{\PP_{2g+1}^1,\mathfrak p}}M_{\mathfrak p}\ge\depth_{k[s,t]_{\mathfrak q}}M_{\mathfrak q}=\dim k[s,t]_{\mathfrak q}=\dim R_{\PP_{2g+1}^1,\mathfrak p}.
\]
Thus \(M_{\mathfrak p}\) is maximal Cohen--Macaulay over \(R_{\PP_{2g+1}^1,\mathfrak p}\), hence \(M\) satisfies \((S_2)\). Therefore \(M\) is reflexive over \(R_{\PP_{2g+1}^1}\). Since the bidual map is homogeneous, \(M\) is reflexive as a graded \(R_{\PP_{2g+1}^1}\)-module.
\end{proof}

Let \(M\) be as in Lemma~\ref{lem:reflexive_R_module}, and let \(\widetilde{\cM}\) be the vector bundle on \(\PP_{2g+1}^1\) associated with \(M\), as in \cite[Lemma~3.2]{HO20}. Then \(\widetilde{\cM}\) gives a parabolic bundle \(\cM := (\underline{\cM},\{U_\ell\}_{\ell=1}^{2g+1})\) on \(\PP^1\), where
\[
\underline{\cM}=\pi_*\widetilde{\cM},\qquad U_\ell:=\im\Bigl((\pi_*y)|_{p_\ell}:(\pi_*\widetilde{\cM}(-\sum q_i))|_{p_\ell}\to(\pi_*\widetilde{\cM})|_{p_\ell}\Bigr)\subset \underline{\cM}|_{p_\ell}.
\]

As graded \(k[s,t]\)-modules, \(M\) and \(M(-\sum q_i)\) sheafify to \(\widetilde{M}^{\mathrm{ev}}\) and \(\widetilde{M}^{\mathrm{odd}}\) on \(\PP^1\), since \(\deg s=\deg t=2\). Indeed,
\[
\underline{\cM}=\widetilde{M}^{\mathrm{ev}},\quad(\pi_*\widetilde{\cM}(-\sum q_i))|_{p_\ell}\cong M^{\mathrm{odd}}|_{p_\ell},\quad\underline{\cM}|_{p_\ell}\cong M^{\mathrm{ev}}|_{p_\ell},
\]
where we use the same notation \((\cdot)|_p\) for graded \(k[s,t]\)-modules, meaning localization at the corresponding homogeneous prime and tensoring with its residue field.

The action of \(y\) on \(M\) is a nonzero scalar multiple of \(e_\ell\prod_{i\ne\ell}e_i\). Locally, the factor \(\prod_{i\ne\ell}e_i\) is invertible on \(M|_{p_\ell}\), since \(e_i^2|_{p_\ell}=(s+\lbd_i t)|_{p_\ell}\neq 0\) for each \(i\ne\ell\). Therefore, we obtain the following simpler description of its parabolic structure.

\begin{proposition}\label{prop:Cliffordmodule}
Let \(\widetilde{\cM}\) be a vector bundle on \(\PP_{2g+1}^1\), induced by a finitely generated graded right \(\Cl\)-module \(M\). Then the corresponding parabolic bundle is \(\cM = (\underline{\cM}, \{U_\ell\})\), where
\[U_\ell=\im\bigl((-) \cdot e_\ell :M^{\mathrm{odd}}|_{p_\ell}\to M^{\mathrm{ev}}|_{p_\ell}\bigr)\subset \underline{\cM}|_{p_\ell}.\]
\end{proposition}

%%%%%%%%%%%%%%%%%%%%%%%%%%%%%%%%%%
\section{Relative spinor bundle}\label{sec:spinorbundle}

In this section, we introduce relative spinor bundles that will play a key role in the construction of Ulrich bundles.

Later, we need to study an intersection of two quadrics together with a hyperplane section. Thus, we will use the following notation from now on. Let \(V = V^{2g+2}\) be a \(k\)-vector space of dimension \(2g+2\), with a fixed basis \(\{e_1, e_2, \dots, e_{2g+2}\}\) and its dual basis \(\{x_1, x_2, \dots, x_{2g+2}\}\). We fix two nondegenerate quadratic forms
\begin{equation}\label{eqn:quadraticodd}
q_0 = \sum_{i=1}^{2g+2}x_i^2, \qquad q_\infty = \sum_{i=1}^{2g+2}\lambda_i x_i^2
\end{equation}
with \(2g+2\) distinct scalars \(\lambda_i \in k^*\). Let \(Q_0 := V(q_0)\), \(Q_\infty := V(q_\infty)\) be two quadrics. Finally, let \(X = X^{2g-1} := Q_0 \cap Q_\infty \subset \PP V\).

Consider the restrictions of the above objects to a hyperplane \(H := V(x_{2g+2})\), and denote each restriction by the subscript \(H\). Thus \(V_H := \mathrm{Span}\{e_1, e_2, \dots, e_{2g+1}\}\),
\[q_{0,H}:= \sum_{i=1}^{2g+1}x_i^2, \quad q_{\infty,H}:= \sum_{i=1}^{2g+1}\lambda_i x_i^2, \quad Q_{0,H} := V(q_{0,H}), \quad Q_{\infty,H} := V(q_{\infty,H}),\]
and finally, \(X_H := X^{2g-2}:=Q_{0, H} \cap Q_{\infty, H} \subset\PP V_H\). Then \(\hX\) is a smooth intersection of two odd-dimensional quadrics. Denote the inclusion by \(j : X_H \hookrightarrow X\).

\begin{remark}
We will use the notation \(X_H\) when we have to use both \(X\) and \(X_H\) in the same statement or proof. If there is no chance of confusion, we will use \(X\) for both even and odd-dimensional intersections.
\end{remark}

\begin{remark}
One may observe that in the construction below, \(\cS\) depends on a choice of a common isotropic subspace \(U\), hence \(\cS_U\) is a more appropriate notation. But replacing \(U\) by another isotropic subspace \(U'\) gives a bundle \(\cS_{U'}\) isomorphic to \(\cS_U \otimes p_{\cC}^*\cL\) for some \(\cL \in \Pic^0(\cC)\), where \(p_{\cC} : \cC \times X \to \cC\). This does not affect the result we will prove and use later, because tensoring by \(\cL \in \Pic(\cC)\) is an autoequivalence of \(\rD^b(\cC)\). Thus, we will suppress the subscript \(U\) and write \(\cS\), and restore the notation \(\cS_U\) only when the choice of \(U\) matters.
\end{remark}

We review a construction of the relative spinor bundle \(\cS\) for an intersection of two even-dimensional quadrics \(X\), given by \cite{BO95}. Let \(C\to\PP^1\) be the double cover ramified over \(p_i=[-\lambda_i:1]\), \(1\le i\le 2g+2\). Then \(C\) is a hyperelliptic curve of genus \(g\). Let \(\tau\) be its hyperelliptic involution. Let \(\OGr_k \to \PP^1\) be the relative space of isotropic subspaces of vector-space dimension \(k\) in \(V\). Here \(\PP^1\) parametrizes the pencil of quadrics \(sq_0 + t q_\infty\). Then \(\OGr_1 \to \PP^1\) is the universal quadric, and \(\OGr_{g+1} \to \PP^1\) is the space of maximal isotropic subspaces. Since \(\OGr_1\) is the universal quadric, we also use another notation \(\cQ_{\PP^1} := \OGr_1\). Over an unramified point, the fiber of \(\OGr_{g+1} \to \PP^1\) has two connected components, each an orthogonal Grassmannian \(\OGr(g+1, V)\). Over a branch point of \(C \to \PP^1\), the corresponding quadric has corank one and the reduced fiber of \(\OGr_{g+1} \to \PP^1\) is smooth and irreducible. The Stein factorization gives a morphism \(h_{g+1} : \OGr_{g+1} \to C\) with connected fibers. Denote the fiber product \(\cQ_{\PP^1} \times_{\PP^1} C\) by \(\cQ_C\).

Let \(\OGr_{1, g+1} \subset \cQ_C \times_C \OGr_{g+1}\) be the relative incidence variety, with projections
\begin{equation}\label{eqn:projections}
\xymatrix{ & \OGr_{1,g+1}\ar[ld]_{f_1} \ar[rd]^{f_{g+1}}\\ \cQ_C \ar[rd]_{h_1} && \OGr_{g+1} \ar[ld]^{h_{g+1}}\\ &C}.
\end{equation}

We take a relative hyperplane section \(\cQ_{C,H}:=\cQ_C\cap(C\times\PP V_H)\). Then we have the following inclusion morphisms:
\begin{equation}\label{eqn:inclusions}
\xymatrix{
  C \times X_H \ar[d]_{\imath_{X_H}} \ar[r]^{\mathrm{id}_C \times j} & C \times X \ar[d]^{\imath_X} \\
  \cQ_{C,H} \ar[r]^k \ar[d]_{\imath_{\cQ_H}} & \cQ_C \ar[d]^{\imath_{\cQ}}\\
  C \times \PP V_H \ar[r]^{\mathrm{id}_C \times \ell} & C \times \PP V.
}
\end{equation}

For any maximal common isotropic subspace \(U \subset V\) of dimension \(g\),
\[D_U:=\{(c,W)\in\OGr_{g+1}:W\cap U\ne0\}\]
is a relative Schubert divisor on \(\OGr_{g+1}\), and fiberwise it defines the ample generator of \(\Pic(\OGr(g+1, V))\).

\begin{definition}\label{def:spinorbundleeven}
The \textbf{universal spinor sheaf} associated with \(U\) is the coherent sheaf \(\tcS\) on the universal quadric \(\cQ_C\) obtained by
\[f_{1*}f_{g+1}^*\cO _{\OGr_{g+1}}(D_U).\]
The \textbf{relative spinor bundle} is \(\cS := \imath_X^{*} \tcS\) over \(C \times X\).
\end{definition}

\begin{remark}
The sheaf \(\tcS\) is not locally free at the singular points of singular quadrics as discussed in \S\ref{ssec:spinor}. However, the smooth intersection of quadrics \(C \times X\) does not meet any of these points. Thus, \(\cS\) is locally free.
\end{remark}

\begin{remark}\label{rem:alternativedefinitionspinor}
Several constructions of a spinor bundle on a fixed quadric appear in the literature; see, for instance, \cite{Ott88} and \cite{Add11}. The construction above is a relative version of the one in \cite{Add11}. Note also that for each fiber, our spinor bundle is the dual of the spinor bundle in \cite{Ott88}.
\end{remark}

To obtain a \(\tau\)-linearized spinor bundle on \(C \times X_H\), choose \(U\) contained in \(V_H \subset V\). There are finitely many such choices \cite[Theorem~3.8]{Rei73}.

Let \(\sigma_{2g+2}: V \to V\) be the reflection
\[\sigma_{2g+2}(e_i)=e_i\quad(1\le i\le2g+1),\qquad\sigma_{2g+2}(e_{2g+2})=-e_{2g+2}.\]
Then \(\sigma_{2g+2}\) induces an involution on each quadric \(V(s q_0 + t q_\infty)\) and hence an action on \(\OGr_k\). It acts nontrivially on \(C\) as the hyperelliptic involution \(\tau\), while fixing \(\PP V_H \subset \PP V\) pointwise. Thus, on \(C\times\PP V_H\), the \(\sigma_{2g+2}\)-action is \(\tau\times\operatorname{id}_{\PP V_H}\).

\begin{proposition} \label{prop:restricted_spinor_linearized}
Let \(k : \cQ_{C,H} \hookrightarrow \cQ_C\) be the inclusion. We set \(\tcS' := k^*\tcS\). Then \(\tcS'\) is \(\tau\)-linearized. Moreover, if \(q \in C\) is the ramification point over \(p_{2g+2}\), we may choose the linearization so that the \(\tau\)-action on \(\widetilde{\cS}'|_{q\times Q_{p_{2g+2},H}}\) is trivial, where \(Q_{p_{2g+2},H}:=Q_{p_{2g+2}}\cap\PP V_H.\)
\end{proposition}

\begin{proof}
Since \(U\subset V_H\), the involution \(\sigma_{2g+2}\) fixes \(U\). Hence, for every maximal isotropic subspace \(W \in \OGr_{g+1}\) in a fiber, \(W\cap U\ne 0\) if and only if \(\sigma_{2g+2}(W)\cap U\ne 0\). Hence \(D_U\) is invariant under the induced \(\sigma_{2g+2}\)-action. Therefore, \(\cO_{\OGr_{g+1}}(D_U)\), and consequently the spinor bundle \(\tcS\), inherit a \(\sigma_{2g+2}\)-linearization. Since \(\sigma_{2g+2}\) fixes \(\PP V_H\) pointwise, \(\tcS'\) is \(\sigma_{2g+2}\)-linearized, and this action coincides with the \(\tau\)-action.

It remains to check the last assertion. At \(p_{2g+2}\), the fiber of \(\cQ_C \to \PP^1\) is defined by
\[\sum_{i=1}^{2g+1}(\lambda_i-\lambda_{2g+2})x_i^2=0,\]
so it is a cone in \(\PP V \) with vertex \([e_{2g+2}]\) over a smooth quadric in \(\PP V_H\). The fiber of \(\tcS\) over \(q \times \PP V \cong \PP V\) is the pullback of the spinor bundle in \(\PP V_H\) \cite[p.24]{BO95}, and the spinor bundle is simple; the \(\tau\)-action is given by a character of \(\tau\). Twisting the linearization by the sign character if necessary makes the action trivial.
\end{proof}

Let \(\rho : C \to \PP_{2g+1}^1\) be the composition
\[C \to [C/\tau] \cong \PP_{2g+2}^1 \to \PP_{2g+1}^1,\]
where the first map is the quotient map, and the second map is the relative coarse moduli space map forgetting the stacky structure at \(p_{2g+2}\).

Recall that
\[\cQ_{\PP^1,H}:=\cQ_{\PP^1}\cap(\PP^1\times\PP V_H)\subset\PP^1\times\PP V_H\]
is the universal quadric of the restricted pencil on \(V_H\). Denote its base changes to \(\PP_{2g+2}^1\) and \(\PP_{2g+1}^1\) by \(\cQ_{\PP_{2g+2}^1}\) and \(\cQ_{\PP_{2g+1}^1}\), respectively. Then we have morphisms
\[\cQ_{C,H}\to\cQ_{\PP_{2g+2}^1}\to\cQ_{\PP_{2g+1}^1}\to\cQ_{\PP^1,H}.\]

Since the sheaf \(\tcS'\) admits a \(\tau\)-linearization, it descends to a sheaf (denoted by the same symbol) on \(\cQ_{\PP_{2g+2}^1}\). Since the stabilizer over \(p_{2g+2}\) acts trivially, it descends further to \(\cQ_{\PP_{2g+1}^1}\). Therefore, we obtain the following corollary.

\begin{corollary} \label{cor:descended_spinor_kernel}
There exists a coherent sheaf \(\tcS\) on \(\cQ_{\PP_{2g+1}^1}\) such that its pullback to \(\cQ_{C,H}\) is \(\tcS'\).
\end{corollary}

\begin{definition}\label{def:spinorbundleodd}
We define the \textbf{relative spinor bundle} \(\cS\) over \(\PP_{2g+1}^1 \times X_H\) by the restriction of \(\tcS\) to \(\PP_{2g+1}^1 \times X_H\hookrightarrow \cQ_{\PP_{2g+1}^1}\). Even though \(\tcS\) is not locally free in general, the image of \(\PP_{2g+1}^1 \times X_H \hookrightarrow \cQ_{\PP_{2g+1}^1}\) avoids the non-locally free locus, hence \(\cS\) is a vector bundle.
\end{definition}

In summary, we have a relative spinor bundle \(\cS\) over \(\cC \times X\) for every \(\cC\).

%%%%%%%%%%%%%%%%%%%%%%%%%%%%%%%%%%%%%%%%%%
\section{Raynaud bundle}\label{sec:Raynaud}

Denote the two projections \(\cC \times X \to \cC\) and \(\cC \times X \to X\) by \(p_{\cC}\) and \(p_X\), respectively.

\begin{definition}\label{def:Raynaud}
Let \(\cS\) be a relative spinor bundle on \(\cC \times X\). The \textbf{Raynaud bundle} is \(\cR := p_{\cC*}\cS\).
\end{definition}

\begin{lemma}\label{lem:Rvb}
For the relative spinor bundle \(\cS\), \(\cR\) is a vector bundle on \(\cC\).
\end{lemma}

\begin{proof}
We give the proof of the case \(\cC = C\). The other case is identical.

Since \(\cS\) is flat over \(C\), it is enough to show that \(R^ip_{\cC*}\cS = 0\) for all \(i>0\). If we denote the spinor sheaf over \(Q_c\) by \(\widetilde{\cS}_c\), we have an exact sequence
\begin{equation}\label{eqn:spinor_hyperplane_exact}
0 \to \widetilde{\cS}_c(-2) \to \widetilde{\cS}_c \to \cS_c \to 0.
\end{equation}
Because \(\widetilde{\cS}_c\) is an Ulrich sheaf on \(Q_c\) (Remark~\ref{rmk:Ulrichquadric}), \(\rH^{i}(Q_c, \widetilde{\cS}_c)=\rH^{i}(Q_c, \widetilde{\cS}_c(-2)) = 0\) for all \(i > 0\) if \(\dim Q_c \geq 2\). The long exact sequence in cohomology then gives \(\rH^{i}(X, \cS_c)=0\) and hence \(R^ip_{\cC*}\cS = 0\) for all \(i \geq 1\).
\end{proof}

As we mentioned earlier, the relative spinor bundle \(\cS\), and hence \(\cR\), depend on a choice of maximal isotropic subspace \(U\). We suppress \(U\) from the notation whenever no confusion can arise.

In this section, we describe \(\cR\) in terms of Clifford algebras, and show that \(\cR\) is indeed isomorphic to a vector bundle \(\cF\) constructed in \cite[Proposition~4.5]{ES25}.

\subsection{Raynaud bundle as a Clifford module}

In this section, we provide a description of \(\cR\) as a \(\widetilde{\cC\ell}_0\)-module.

Recall that \(\cQ_{\cC} \to \cC\) is the family of relative quadrics described in \S\ref{sec:spinorbundle}. Let
\[\imath_{\cQ} :\cQ_{\cC} \hookrightarrow \cC \times\PP V,\qquad \imath_X: \cC\times X\hookrightarrow\cQ_{\cC}\]
be the two natural inclusions; we also have the two projections \(\rho_{\cC} : \cC \times \PP V \to \cC\) and \(\rho_{\PP V} : \cC \times \PP V \to \PP V\).

We first construct a relative version of the exact sequence in \eqref{eqn:spinorresolution}.

\begin{proposition}\label{prop:relativeMF}
The Raynaud bundle \(\cR\) is locally free of rank \(2^g\) and carries a right \(\widetilde{\cC\ell}_0\)-module structure. There is a locally free right \(\widetilde{\cC\ell}_0\)-module \(\cR'\) and an exact sequence
\begin{equation}\label{eqn:relative_spinor_ambient_resolution}
0\longrightarrow \cR'\boxtimes\cO_{\PP V}(-1) \longrightarrow \cR\boxtimes\cO_{\PP V} \longrightarrow (\imath_{\cQ})_*\widetilde{\cS} \longrightarrow0.
\end{equation}
Its restriction to every geometric fiber is the second sequence in \eqref{eqn:spinorresolution}.
\end{proposition}

\begin{proof}
The relative Cartier divisor \(\cC\times X\subset\cQ_{\cC}\) has the ideal sheaf \(\cO_{\cQ_{\cC}}(-2)\otimes p_{\cC}^*\cL\) for some \(\cL\in\Pic(\cC)\). Tensoring its restriction sequence with \(\widetilde{\cS}\) gives
\begin{equation}\label{eqn:StildeS}
0\longrightarrow \widetilde{\cS}(-2)\otimes p_{\cC}^*\cL \longrightarrow \widetilde{\cS} \longrightarrow (\imath_X)_*\cS \longrightarrow0.
\end{equation}
The fiberwise resolution \eqref{eqn:spinorresolution} gives \(R\Gamma(Q_c,\widetilde{\cS}_c(-2))=0\). Hence the derived direct image of the first term vanishes and
\[R\rho_{\cC*}(\imath_{\cQ})_*\widetilde{\cS}\cong Rp_{\cC*}\cS=\cR.\]
The same resolution gives \(H^0(Q_c,\widetilde{\cS}_c)\cong I_{0,c}\) and \(H^i(Q_c,\widetilde{\cS}_c)=0\) for \(i>0\). Grauert's theorem implies that \(\cR\) is locally free of rank \(2^g\), with fiber \(I_{0,c}\).

The relative half-spin construction equips \(\cR\) with its natural right \(\widetilde{\cC\ell}_0\)-action. Let \(\widetilde{\cC\ell}_1\) be the odd Clifford bimodule and set
\[\cR':=\cR\otimes_{\widetilde{\cC\ell}_0} \bigl(\widetilde{\cC\ell}_1\otimes\pi^*\cO_{\PP^1}(-1)\bigr).\]
Then \(\cR'_c\cong I_{1,c}\). Relative odd Clifford multiplication gives the first arrow in \eqref{eqn:relative_spinor_ambient_resolution}; its restriction to every geometric fiber is
\[I_{1,c}\otimes\cO_{\PP V}(-1) \longrightarrow I_{0,c}\otimes\cO_{\PP V}.\]
The fiberwise cokernel criterion identifies the cokernel with \((\imath_{\cQ})_*\widetilde{\cS}\).
\end{proof}

\begin{corollary}\label{cor:Cliffordendomorphismalgebra}
The right Clifford action induces an isomorphism
\[\widetilde{\cC\ell}_0^{\mathrm{op}} \xrightarrow{\sim}\cEnd(\cR).\]
\end{corollary}

\begin{proof}
The fiber \(\cR_c\cong I_{0,c}\) is an irreducible \(\widetilde{\cC\ell}_{0,c}\)-module of dimension \(2^g\). The Clifford morphism is an isomorphism on every geometric fiber, since both sides are locally free of rank \(2^{2g}\). Hence it is an isomorphism.
\end{proof}

\subsection{Tautological sequence}

We relativize the tautological filtration of \cite[Lemma~6.1]{Kuz08b}. Let \(\widetilde p_{2g+2}\in C\) be the ramification point over \(p_{2g+2}\) fixed in Proposition~\ref{prop:restricted_spinor_linearized}, and set
\begin{equation}\label{eqn:half_pencil_bundle}
\cO_{\cC}(1):=\begin{cases} \cO_C(\widetilde p_{2g+2}),&\dim V=2g+2,\\ \cO_{\PP_{2g+1}^1}(1),&\dim V=2g+1. \end{cases}\qquad\cO_{\cC}(2)\cong\pi^*\cO_{\PP^1}(1).
\end{equation}
The definitions in \S\ref{ssec:parabolicbundle} also give
\begin{equation}\label{eqn:half_pencil_pullback}
\rho^*\cO_{\PP_{2g+1}^1}(1)\cong\cO_C(\widetilde p_{2g+2}).
\end{equation}
Define the normalized tautological line bundle on \(\cQ_{\cC}\subset\cC\times\PP V\) by
\begin{equation}\label{eqn:normalized_tautological}
\cU:=\imath_{\cQ}^*\bigl(\cO_{\cC}(-1)\boxtimes\cO_{\PP V}(-1)\bigr),
\end{equation}
and use the same notation for its restriction to \(\cC\times X\).

\begin{remark}
The factor \(\cO_{\cC}(-1)\) in \eqref{eqn:normalized_tautological} is essential. On \(\cC\times\PP V\), the universal quadratic form \(q\) is a section of
\[\pi^*\cO_{\PP^1}(1)\boxtimes\cO_{\PP V}(2)\cong \cO_{\cC}(2)\boxtimes\cO_{\PP V}(2).\]
Its restriction to \(\cQ_{\cC}\) is therefore valued in
\[\cU^{-2}\cong \imath_{\cQ}^*\bigl(\pi^*\cO_{\PP^1}(1)\boxtimes\cO_{\PP V}(2)\bigr).\]
The two normalized Clifford compositions below are consequently \(q\cdot\mathrm{id}\) and arise by restricting a matrix factorization of \(q\). Without the factor \(\cO_{\cC}(-1)\), the residual twist \(\pi^*\cO_{\PP^1}(1)\) would remain.
\end{remark}

\begin{proposition}\label{prop:filtrationeven}
Suppose that \(X\) is a \((2g-1)\)-dimensional intersection of quadrics. There is an exact sequence
\begin{equation}\label{eqn:Raynaudfiltrationeven}
0 \to (\tau \times \mathrm{id}_X)^*\cS \otimes \cU \to p_C^*\cR \to \cS \to 0
\end{equation}
over \(C \times X\), where \(\cU\) is the tautological bundle of the universal quadric restricted to \(C \times X\).
\end{proposition}

\begin{proof}
Set \(\cR_1:=\cR'\otimes\cO_C(1)\), where \(\cR'\) is the module in Proposition~\ref{prop:relativeMF}. Clifford multiplication gives
\[\alpha : \cR_1 \otimes \cU \longrightarrow \cR, \qquad \beta : \cR \otimes \cU \longrightarrow \cR_1.\]
Their compositions are
\[\alpha \circ (\beta \otimes \mathrm{id}_\cU)=q\cdot\mathrm{id},\qquad \beta \circ (\alpha \otimes \mathrm{id}_\cU)=q\cdot\mathrm{id}.\]
They vanish on \(\cQ_C\), and the resulting complex
\[\cdots \longrightarrow\cR \otimes \cU^{\otimes 2}\xrightarrow{\beta \otimes \mathrm{id}_{\cU}}\cR_1 \otimes\cU\xrightarrow{\alpha}\cR\longrightarrow 0\]
is exact except at the rightmost term. Since \(\tcS=\operatorname{coker}\alpha\), setting \(\widetilde{\cT}:=\operatorname{coker}\beta\) gives
\[0 \to \widetilde{\cT} \otimes \cU \to \cR \to \tcS \to 0.\]
Set \(\cT:=\imath_X^*\widetilde{\cT}\). Both cokernels are locally free along \(C\times X\), so pullback along \(\imath_X\) gives
\[0\longrightarrow \cT \otimes \cU \longrightarrow p_C^*\cR \longrightarrow \cS \longrightarrow 0.\]
Clifford multiplication by \(e_{2g+2}\), divided by the canonical section of \(\cO_C(\widetilde p_{2g+2})\), gives
\[\cT\cong(\mathrm{id}_C\times\sigma_{2g+2})^*\cS\cong(\tau\times\mathrm{id}_X)^*\cS,\]
where the second isomorphism follows from the invariance of \(D_U\) under \(\sigma_{2g+2}\).
\end{proof}

\begin{proposition}\label{prop:filtrationodd}
Suppose that \(X\) is a \((2g-2)\)-dimensional intersection of quadrics. There is an exact sequence
\begin{equation}\label{eqn:Raynaudfiltrationodd}
0 \to \cS \otimes \cU \to p_{\PP_{2g+1}^1}^*\cR \to \cS \to 0
\end{equation}
over \(\PP_{2g+1}^1 \times X\), where \(\cU\) is the tautological bundle of the universal quadric restricted to \(\PP_{2g+1}^1 \times X\).
\end{proposition}

\begin{proof}
Use the notation \(X_H\subset X\) and \(\rho:C\to\PP_{2g+1}^1\) from \S\ref{sec:spinorbundle}, and set \(\cS':=(\mathrm{id}_C\times j)^*\cS\). Restricting \eqref{eqn:Raynaudfiltrationeven} to \(C\times X_H\) gives
\[0 \longrightarrow (\tau \times \mathrm{id}_{X_H})^*\cS' \otimes \cU|_{C \times X_H} \longrightarrow p_C^*\cR \longrightarrow \cS' \longrightarrow 0.\]
Proposition~\ref{prop:restricted_spinor_linearized} and \eqref{eqn:half_pencil_pullback} identify
\[(\tau\times\mathrm{id}_{X_H})^*\cS'\cong\cS' \cong(\rho\times\mathrm{id}_{X_H})^*\cS_H, \qquad \cU|_{C\times X_H}\cong(\rho\times\mathrm{id}_{X_H})^*\cU_H.\]
It remains to identify \(\cR\cong\rho^*\cR_H\).

The hyperplane restriction sequence is
\[0 \longrightarrow \cS(-1) \longrightarrow \cS \longrightarrow (\mathrm{id}_C \times j)_*\cS' \longrightarrow 0.\]
Tensoring \eqref{eqn:StildeS} with \(\cO_{\PP V}(-1)\) and applying
\(R\rho_{\cC*}(\imath_{\cQ})_*\), the fiberwise Ulrich vanishings give
\[R\rho_{\cC*}(\imath_{\cQ})_*\widetilde{\cS}(-1)=0,\qquad R\rho_{\cC*}(\imath_{\cQ})_*\bigl(\widetilde{\cS}(-3)\otimes\pi^*\cO_{\PP^1}(1)\bigr)=0.\]
Hence \(Rp_{C*}\cS(-1)=0\), and flat base change yields
\[\cR\cong Rp_{C*}\cS' \cong Rp_{C*}(\rho\times\mathrm{id}_{X_H})^*\cS_H \cong\rho^*Rp_{\PP_{2g+1}^1*}\cS_H \cong\rho^*\cR_H.\]
The restricted sequence is therefore the pullback of \eqref{eqn:Raynaudfiltrationodd}; descent along \(\rho\) proves the proposition.
\end{proof}

%%%%%%%%%%%%%%%%%%%%%%%%%%%%%%%%%%%%%%%%%
\section{Semiorthogonal decomposition and spinor bundle}\label{sec:SOD}

As a concrete example of his theory of homological projective duality, Kuznetsov described the following semiorthogonal decompositions of the intersection of two quadrics.

\begin{theorem}[\protect{\cite[Corollary~5.7]{Kuz08}}]\label{thm:SOD}
The bounded derived category of the intersection of two quadrics \(X\) has the following semiorthogonal decomposition
\begin{equation}\label{eqn:SOD}
\rD^b(X) = \langle \rD^b(\cC), \cO(1), \cO(2), \dots, \cO(\dim X - 1)\rangle.
\end{equation}
\end{theorem}

When \(\dim X\) is odd, Theorem~\ref{thm:SOD} was obtained earlier by Bondal and Orlov in \cite[Theorem~2.7]{BO95}. Furthermore, the embedding \(\rD^b(C) \hookrightarrow \rD^b(X)\) was described as a Fourier--Mukai transform whose kernel is the relative spinor bundle \(\cS \in \rD^b(C \times X)\). We extend the result to the case where \(\dim X\) is even.

Let \(\iota:\cQ_{\PP^1} \hookrightarrow \PP^1\times\PP V\) be the inclusion, and let \(\delta:=\sum_i x_i\otimes e_i\in V^*\otimes V\) be the element corresponding to \(\mathrm{id}_V\). Right multiplication by \(\delta\) gives a morphism of left \(\cC\ell_0\)-modules
\[\delta'_{0,0}: \cC\ell_1\otimes\cO_{\PP^1}(-1)\boxtimes\cO_{\PP V}(-1) \longrightarrow \cC\ell_0\boxtimes\cO_{\PP V},\]
where
\[\cC\ell_1 := V \otimes \cO_{\PP^1} \oplus \wedge^3 V \otimes \cO_{\PP^1}(-1) \oplus \wedge^5 V \otimes \cO_{\PP^1}(-2) \oplus \dots\]
is a \((\cC\ell_0, \cC\ell_0)\)-bimodule. Since \((\delta'_{0,0})^2\) is multiplication by the universal quadratic form, \cite[Lemma~4.5]{Kuz08} gives a left \(\cC\ell_0\)-module \(\cE'_{0,0}\) on \(\cQ_{\PP^1}\) and an exact sequence
\begin{equation}\label{eqn:E00prime}
0\to \cC\ell_1\otimes\cO_{\PP^1}(-1)\boxtimes\cO_{\PP V}(-1) \xrightarrow{\delta'_{0,0}} \cC\ell_0\boxtimes\cO_{\PP V} \to \iota_*\cE'_{0,0} \to 0.
\end{equation}
Set \(\cK:=(\PP^1\times X\hookrightarrow\cQ_{\PP^1})^*\cE'_{0,0}\). The long Clifford resolution \cite[(22)]{Kuz08} restricts to
\begin{align}\label{eqn:Kuznetsov_kernel_resolution}
\cdots\to\bigl(\cC\ell_0\otimes\cO_{\PP^1}\bigr)(-1)\boxtimes\cO_X(-2) \to\bigl(\cC\ell_1\otimes\cO_{\PP^1}(-1)\bigr)&\boxtimes\cO_X(-1) \\
&\to \cC\ell_0\boxtimes\cO_X \to \cK \to 0.\nonumber
\end{align}
Let \(\rD^b(\PP^1, \cC\ell_0)\) be the bounded derived category of right \(\cC\ell_0\)-modules over \(\PP^1\). We define a functor \(\Phi_{\cK} : \rD^b(\PP^1, \cC\ell_0) \to \rD^b(X)\) as
\[\Phi_{\cK}(F) := Rp_{X*} \bigl( p_{\PP^1}^*F \otimes_{p_{\PP^1}^*\cC\ell_0}^{\mathbf L} \cK \bigr),\]
where \(p_{\PP^1}:\PP^1\times X\to\PP^1\) and \(p_{X} :\PP^1\times X\to X\) are the two projections.

Here and below, pullback along \(\pi\times\mathrm{id}_X\) is taken in the category of Clifford modules; in particular, it includes extension of scalars along \(p_{\cC}^*\pi^*\cC\ell_0\to p_{\cC}^*\widetilde{\cC\ell}_0\).

\begin{proposition}\label{prop:MoritaKuznetsov}
The functor
\[\mathsf M_{\cR}:=\pi_*(-\otimes_{\cO_{\cC}}\cR):\rD^b(\cC)\xrightarrow{\sim}\rD^b(\PP^1,\cC\ell_0)\]
is a Morita equivalence. Moreover, there is an isomorphism
\begin{equation}\label{eqn:Kuznetsov_spinor_comparison}
p_{\cC}^*\cR\otimes_{p_{\cC}^*\widetilde{\cC\ell}_0}(\pi\times\mathrm{id}_X)^*\cK\cong\cS
\end{equation}
on \(\cC\times X\). Consequently, the diagram
\[
\begin{tikzcd}[column sep=6.2em,row sep=4.2em]\rD^b(\cC)\arrow[rr,"\mathsf M_{\cR}","\sim"']\arrow[dr,"\Phi_{\cS}"']&&\rD^b(\PP^1,\cC\ell_0)\arrow[dl,"\Phi_{\cK}"]\\ &\rD^b(X)&\end{tikzcd}
\]
commutes up to a natural isomorphism.
\end{proposition}

\begin{proof}
Corollary~\ref{cor:Cliffordendomorphismalgebra} identifies \(\widetilde{\cC\ell}_0^{\mathrm{op}}\) with \(\cEnd(\cR)\), so tensoring with \(\cR\) gives a Morita equivalence from \(\rD^b(\cC)\) to \(\rD^b(\cC,\widetilde{\cC\ell}_0)\). The exact equivalence of \S\ref{ssec:sheafification} is induced by \(\pi_*\), and their composition is \(\mathsf M_{\cR}\).

Pulling \eqref{eqn:Kuznetsov_kernel_resolution} back in the category of Clifford modules gives the exact sequence
\begin{align}\label{eqn:Kuznetsov_kernel_resolution_pullback}
\cdots\to\bigl(\widetilde{\cC\ell}_0\otimes\pi^*\cO_{\PP^1}(-1)\bigr)\boxtimes\cO_X(-2)\to\bigl(\widetilde{\cC\ell}_1&\otimes\pi^*\cO_{\PP^1}(-1)\bigr)\boxtimes\cO_X(-1)\\
&\to\widetilde{\cC\ell}_0\boxtimes\cO_X\to(\pi\times\mathrm{id}_X)^*\cK\to0.\nonumber
\end{align}
Since \(\cR\) is locally projective as a right \(\widetilde{\cC\ell}_0\)-module, tensoring \eqref{eqn:Kuznetsov_kernel_resolution_pullback} with \(p_{\cC}^*\cR\) preserves exactness and gives
\begin{align}\label{eqn:Morita_Kuznetsov_resolution}
\cdots\to\bigl(\cR\otimes\pi^*\cO_{\PP^1}(-1)\bigr)\boxtimes\cO_X(-2)&\to\cR'\boxtimes\cO_X(-1)\\
&\to\cR\boxtimes\cO_X\to p_{\cC}^*\cR\otimes_{p_{\cC}^*\widetilde{\cC\ell}_0}(\pi\times\mathrm{id}_X)^*\cK\to0.\nonumber
\end{align}
The Clifford multiplication \(\cR'\boxtimes\cO_X(-1)\to\cR\boxtimes\cO_X\) in \eqref{eqn:Morita_Kuznetsov_resolution} is the restriction of the multiplication in \eqref{eqn:relative_spinor_ambient_resolution}. Its cokernel is \(\cS\), so exactness proves \eqref{eqn:Kuznetsov_spinor_comparison}.

For \(F\in\rD^b(\cC)\), flat base change and the projection formula give
\begin{align*}
(\Phi_{\cK}\circ\mathsf M_{\cR})(F) &=Rp_{X*}\bigl(p_{\PP^1}^*\pi_*(F\otimes\cR)\otimes_{p_{\PP^1}^*\cC\ell_0}^{\mathbf L}\cK\bigr)\\
&\cong R\bigl(p_X\circ(\pi\times\mathrm{id}_X)\bigr)_*\bigl(p_{\cC}^*F\otimes^{\mathbf L}\bigl(p_{\cC}^*\cR\otimes_{p_{\cC}^*\widetilde{\cC\ell}_0}(\pi\times\mathrm{id}_X)^*\cK\bigr)\bigr)\\
&\cong R\bigl(p_X\circ(\pi\times\mathrm{id}_X)\bigr)_*\bigl(p_{\cC}^*F\otimes^{\mathbf L}\cS\bigr)=\Phi_{\cS}(F).
\end{align*}
Hence \(\Phi_{\cK}\circ\mathsf M_{\cR}\cong\Phi_{\cS}\).
\end{proof}

It remains to identify \(\Phi_{\cK}\) with Kuznetsov's Clifford embedding. Apply the linear-section theorem \cite[Theorem~6.3, \S\S6.1, 6.4]{Kuz07} to the double Veronese HPD of \cite[Theorem~5.4]{Kuz08} and the pencil \(L:=\langle q_0,q_\infty\rangle\subset S^2V^*\). Here \(r\) denotes the vector-space dimension of the chosen space of quadrics. The \(r=1\) family restricted to \(\PP L\) is \(\OGr_1\), with kernel \(\cE'_{0,0}\) by \cite[Lemma~4.5, Remark~4.8]{Kuz08}, while the fiber of the \(r=2\) family over \([L]\) is \(X=Q_0\cap Q_\infty\). On this fiber, \(\zeta_2\) restricts to the inclusion \(\PP L\times X\hookrightarrow\OGr_1\); hence the induced kernel is
\[(\PP L\times X\hookrightarrow\OGr_1)^*\cE'_{0,0}=\cK.\]
The smoothness of \(X\) makes \(L\) admissible, so faithful base change identifies \(\Phi_{\cK}\) with the fully faithful functor realizing the Clifford component in \cite[Theorem~5.5]{Kuz08}. Since \(\mathsf M_{\cR}\) is an equivalence, Proposition~\ref{prop:MoritaKuznetsov} gives \(\Phi_{\cS}(\rD^b(\cC))=\Phi_{\cK}(\rD^b(\PP^1,\cC\ell_0))\).

\begin{corollary}\label{cor:FMfullyfaithful}
The Fourier--Mukai transform \(\Phi_{\cS}:\rD^b(\cC)\to\rD^b(X)\) is fully faithful, and
\[\rD^b(X)=\langle\Phi_{\cS}(\rD^b(\cC)),\cO_X(1),\cO_X(2),\dots,\cO_X(\dim X-1)\rangle.\]
\end{corollary}

\begin{remark}
Using the Bondal-Orlov criterion for Deligne--Mumford stack \cite{LP21}, we may prove the full faithfulness of \(\Phi_{\cS} : \rD^b(\PP_{2g+1}^1) \to \rD^b(X)\) directly. We leave the computational details to the interested reader.
\end{remark}

%%%%%%%%%%%%%%%%%%%%%%%%%%%%%%%%%%%%%%%%%%%%%%%%%%%
\section{Reduction to curves}\label{sec:reduction}

The semiorthogonal decomposition of \(\rD^b(X)\) in Theorem~\ref{thm:SOD}, together with its Lefschetz structure, reduces the construction of Ulrich bundles on \(X\) to the construction of a vector bundle on the Kuznetsov component. We describe this reduction. We retain the notational convention of \S\ref{sec:Raynaud}.

Suppose \(\cE\) is an Ulrich bundle on \(X\). By definition, \(\rH^*(X, \cE(-j)) = 0\) for \(1 \le j \le d = \dim X\). We may rephrase this vanishing condition as
\begin{equation}\label{eqn:Ulrichcondition}
\rHom_{\rD^b(X)}(\cO(j), \cE(-1)[i]) = 0, \quad i \in \ZZ, \; 0 \le j \le d - 1.
\end{equation}
This implies \(\rHom_{\rD^b(X)}(\cO, \cE(-1)[i]) = 0\) and \(\cE(-1) \in \langle \cO(1), \cO(2), \dots, \cO(d - 1)\rangle^\perp\). This orthogonal complement is \(\Phi_{\cS}(\rD^b(\cC))\). Thus, there is \(\cF^\bullet \in \rD^b(\cC)\) such that \(\cE(-1) = \Phi_{\cS}(\cF^\bullet)\).

\begin{proposition}\label{prop:FMofUlrichisvb}
Let \(\Phi_{\cS}^! : \rD^b(X) \to \rD^b(\cC)\) be the right adjoint functor of \(\Phi_{\cS}\). For any Ulrich bundle \(\cE\) on \(X\), the object \(\cF^\bullet = \Phi_{\cS}^!(\cE(-1))\in \rD^b(\cC)\) is a vector bundle on \(\cC\).
\end{proposition}

\begin{proof}
We have
\[\cF^\bullet = \Phi_{\cS}^!(\cE(-1)) = Rp_{\cC*}(p_X^*(\cE(-1)) \otimes \cS^*) \otimes \omega_\cC[1].\]
Thus, \(\cF^\bullet\) is a complex of coherent sheaves whose \(i\)-th cohomology is given fiberwise by
\[\rH^{i+1}(X, \cE(-1) \otimes \cS_{c}^*) \cong \Ext_X^{i+1}(\cE^*(1), \cS_{c}^*)\]
for every \(c \in \cC\). These groups vanish for \(i < -1\).

Suppose \(d = 2g - 1\) is odd, so \(\cC = C\). For a general \(c \in C\), there is an exact sequence
\begin{equation}\label{eqn:Ottavianieven}
0 \to \cS_{c}^* \to \cO_{X}^{\oplus 2^g} \to \cS_{\tau c}^*(1) \to 0
\end{equation}
obtained from \cite[Theorem~2.8]{Ott88} by intersecting with a complementary quadric \(Q_c'\). Here \(\tau : C \to C\) is the hyperelliptic involution. For a special \(c \in C\) with \(\tau c = c\), the corresponding quadric is singular. But the spinor bundle is the pullback of one on a lower-dimensional odd quadric, so we have a similar exact sequence, again from \cite[Theorem~2.8]{Ott88}. Tensoring with \(\cE(-1)\) yields an inclusion
\[0 \to \rH^0(X, \cS_c^* \otimes \cE(-1)) \to \rH^0(X, \cE(-1))^{\oplus 2^g}.\]
Since \(\cE\) is Ulrich, the latter is zero. Thus, \(\Ext_X^0(\cE^*(1), \cS_c^*) \cong \rH^0(X, \cS_c^* \otimes \cE(-1)) = 0\).

On the other hand, exchanging \(c\) and \(\tau c\), tensoring \eqref{eqn:Ottavianieven} with \(\cE(-2)\), and taking the long exact sequence gives \(\rH^{i+1}(X, \cE(-1) \otimes \cS_{c}^*) \cong \rH^{i+2} (X, \cE(-2) \otimes \cS_{\tau c}^*)\). Applying the same computation with \(\cE(-j)\) for \(j \le 2g-1\) gives
\[\rH^{i+1}(X, \cE(-1) \otimes \cS_{c}^*) \cong \rH^{i+3}(X, \cE(-3) \otimes \cS_{c}^*) \cong \dots \cong \rH^{i+2g-1}(X, \cE(-2g+1) \otimes \cS_{c}^*).\]
The last cohomology is zero for \(i > 0\). Hence, the only potentially nonzero cohomology sheaf of \(\cF^\bullet\) occurs in degree zero and \(\cF^\bullet = \cF\in \Coh(\cC)\). Since \(\cS^* \otimes p_X^*\cE(-1)\) is flat, \(\dim \rH^1(X, \cS_{c}^* \otimes \cE(-1)) = -\chi(X, \cS_{c}^* \otimes \cE(-1))\) is a constant. Therefore, \(\cF\) is a vector bundle.

Suppose \(d = 2g - 2\), hence \(\cC = \PP_{2g+1}^1\). The proof is nearly identical to the previous case. For any non-stacky point, we have an exact sequence
\begin{equation}\label{eqn:Ottavianiodd}
0 \to \cS_c^* \to \cO_X^{\oplus 2^g} \to \cS_c^*(1) \to 0
\end{equation}
from \cite[Theorem~2.8]{Ott88}. Tensoring \(\cE(-j)\) for \(1 \le j \le 2g-2\), we obtain
\[0 \to \cE(-j) \otimes \cS_c^* \to \cE(-j)^{\oplus 2^g} \to \cE(-j+1) \otimes \cS_c^* \to 0.\]
Since \(\cE\) is Ulrich, \(\rH^i(X, \cE(-j)) = 0\) for all \(i\). This implies \(\Ext^0_X(\cE^*(1), \cS_c^*) = \rH^0(X, \cS_c^* \otimes \cE(-1)) = 0\) and
\[\rH^{i+1}(X, \cE(-j+1)\otimes \cS_c^*) \cong \rH^{i+2}(X, \cE(-j) \otimes \cS_c^*).\]
Iterating the same argument, we obtain
\[\rH^{i+1}(X, \cE(-1) \otimes \cS_c^*) \cong \rH^{i+2g-2}(X, \cE(-2g+2)\otimes \cS_c^*) = 0\]
for \(i > 0\). For a stacky point \(c\), the bundle \(\cS_c^*\) splits into a direct sum of two spinor bundles from a lower-dimensional even quadric, but we have the same exact sequence in \eqref{eqn:Ottavianieven}. Hence the same vanishing argument works. Thus, \(\cF^\bullet = \cF\) for some vector bundle \(\cF\).
\end{proof}

A partial converse (Theorem~\ref{thm:existenceUlrich}) is also true. This is a key ingredient in the construction of Ulrich bundles on \(X\) and \(\hX\).

We say \(A^\bullet \in \rD^b(Y)\) is \textbf{cohomologically (left) orthogonal} to \(B^\bullet \in \rD^b(Y)\) if \(\rHom_{\rD^b(Y)}(A^\bullet, B^\bullet[i]) = 0\) for all \(i\).

\begin{theorem}\label{thm:existenceUlrich}
Let \(\cF\) be a vector bundle on \(\cC\) such that \(\cF^*\) is cohomologically orthogonal to \(\cR\). Then \(\Phi_{\cS}(\cF)(1)\) is an Ulrich bundle on \(X\).
\end{theorem}

\begin{proof}
For any vector bundle \(\cF\) on \(\cC\),
\begin{equation}\label{eqn:duality}
\begin{split}
R\rHom_{\rD^b(\cC)}(\cF^*, \cR) &\cong R\Gamma(\cC, \cF \otimes p_{\cC*}\cS) \cong R\Gamma(\cC \times X, p_{\cC}^*\cF \otimes \cS)\\
&\cong R\Gamma(X, \Phi_{\cS}(\cF)) \cong R\rHom_{\rD^b(X)}(\cO_X, \Phi_{\cS}(\cF)).
\end{split}
\end{equation}

From the cohomological orthogonality, \(\rHom_{\rD^b(X)}(\cO, \Phi_{\cS}(\cF)[i]) = \rHom_{\rD^b(\cC)}(\cF^*, \cR[i]) = 0\). Corollary~\ref{cor:FMfullyfaithful} gives
\[\rH^i(X, \cE^\bullet (-j)) = \rHom_{\rD^b(X)}(\cO_X(j-1), \Phi_{\cS}(\cF)[i]) = 0\]
for \(2 \le j \le d\). Thus, \(\cE^\bullet (-1) = \Phi_{\cS}(\cF)\) is an Ulrich-like object.

Proposition~\ref{prop:relativeMF} gives a surjection \(\cR\twoheadrightarrow\cS_x\) for every \(x\in X\). Since \(\cC\) has cohomological dimension one, tensoring with \(\cF\) and taking cohomology gives
\[0=\rH^1(\cC,\cF\otimes\cR)\longrightarrow\rH^1(\cC,\cF\otimes\cS_x)\longrightarrow0.\]
The bundle \(p_{\cC}^*\cF\otimes\cS\) is flat over \(X\), so cohomology and base change imply that \(\Phi_{\cS}(\cF)\) is a vector bundle concentrated in degree zero.
\end{proof}

\begin{proposition}\label{prop:Fissemistable}
Let \(\cE\) be an Ulrich bundle on \(X\) and let \(\cF = \Phi_{\cS}^!(\cE(-1))\). Then \(\cF\) (hence \(\cF^*\)) is semistable.
\end{proposition}

\begin{proof}
The computation in \eqref{eqn:duality} shows that \(\rH^i(\cC, \cF \otimes \cR) = \rHom_{\rD^b(\cC)}(\cF^*, \cR[i]) = 0\) for all \(i\). By \cite[Exercise~2.8]{Pop13}, \(\cF\) and \(\cR\) are semistable if \(\cC = C\). When \(\cC = \PP_{2g+1}^1\), this follows from \cite[Theorem~6.1]{BD11}.
\end{proof}

\begin{remark}
Note that \cite[Theorem~6.1]{BD11} is stated in terms of parabolic bundles, but we may use the equivalence between the category of bundles on \(\PP_{2g+1}^1\) and the category of parabolic bundles on its coarse moduli space, as described in \S\ref{ssec:parabolicbundle}. Indeed, its proof is based on the cohomology computation on the stack.
\end{remark}

\begin{proposition}\label{prop:Sequivalence}
Suppose that there is an exact sequence
\begin{equation}\label{eqn:sesUlrich}
0 \to \cE' \to \cE \to \cE'' \to 0
\end{equation}
of Ulrich bundles on \(X\). Then the functor \(\Phi_{\cS}^!(- \otimes \cO(-1))\) induces an exact sequence of bundles on \(\cC\)
\begin{equation}\label{eqn:sesorthogonal}
0 \to \cF' \to \cF \to \cF'' \to 0
\end{equation}
whose duals are cohomologically orthogonal to \(\cR\). Conversely, if there is a short exact sequence of bundles whose duals are orthogonal to \(\cR\) as in \eqref{eqn:sesorthogonal}, the functor \(\Phi_{\cS}(-)\otimes \cO(1)\) induces an exact sequence \eqref{eqn:sesUlrich} of Ulrich bundles on \(X\).
\end{proposition}

\begin{proof}
Suppose that there is an exact sequence \eqref{eqn:sesUlrich}. Applying the functor \(\Phi_{\cS}^!(- \otimes \cO(-1))\), we obtain a distinguished triangle \(\cF' \to \cF \to \cF''\to \cF'[1]\). From Proposition~\ref{prop:FMofUlrichisvb}, this is a short exact sequence of vector bundles \(0 \to \cF' \to \cF \to \cF'' \to 0\) whose duals are orthogonal to \(\cR\). The converse is similar. Applying \(\Phi_{\cS}(-)\otimes \cO(1)\), we have a distinguished triangle \(\cE' \to \cE \to \cE'' \to \cE'[1]\). This is a short exact sequence of Ulrich bundles, as each object is concentrated in degree zero.
\end{proof}

\begin{corollary}\label{cor:stability}
Let \(\cE\) be an Ulrich bundle on \(X\). The bundle \(\cE\) is stable (resp. strictly semistable) if and only if \(\cF = \Phi_{\cS}^!(\cE(-1))\) is.
\end{corollary}

\begin{proof}
An Ulrich bundle is semistable \cite[Proposition~3.3.14]{CMRPL21}. Suppose that \(\cE\) is not stable. Then there is a destabilizing sequence \(0 \to \cE' \to \cE \to \cE'' \to 0\). Then both \(\cE'\) and \(\cE''\) are Ulrich \cite[Proposition~3.3.1, Proposition~3.3.16]{CMRPL21}. By Proposition~\ref{prop:Sequivalence}, we obtain an exact sequence \eqref{eqn:sesorthogonal}. Propositions \ref{prop:Ulrichrankeven} and \ref{prop:Ulrichrankodd} show that \(\mu(\cF') = \mu(\cF) = \mu(\cF'')\). Therefore, \(\cF\) is strictly semistable.

Conversely, suppose that \(\cF\) is strictly semistable and we have a destabilizing sequence \(0 \to \cF' \to \cF \to \cF'' \to 0\). Tensoring with \(\cR\), we have
\begin{equation}\label{eqn:sesR}
0 \to \cF' \otimes \cR \to \cF \otimes \cR \to \cF'' \otimes \cR \to 0.
\end{equation}
Then \(\mu(\cF') = \mu(\cF) = \mu(\cF'')\). For \(\cF\), the orthogonality yields \(\chi(\cC, \cF \otimes \cR) = 0\). Then Riemann--Roch gives \(\chi(\cC, \cF' \otimes \cR) = \chi(\cC, \cF'' \otimes \cR) = 0\), too. Chasing the long exact sequence from \eqref{eqn:sesR}, we obtain \(\rH^*(\cC, \cF' \otimes \cR) = \rH^*(\cC, \cF'' \otimes \cR) = 0\). Thus, we obtain the orthogonality of \({\cF'}^*\) and \({\cF''}^*\). Now applying Proposition~\ref{prop:Sequivalence}, we conclude \(\cE\) is strictly semistable.
\end{proof}

\begin{corollary}\label{cor:Sequivalence}
\begin{enumerate}
  \item Let \(\cE\) be an Ulrich bundle on \(X\). Then the functor \(\Phi_{\cS}^!(-\otimes \cO(-1))\) preserves Jordan--H\"older filtration.
  \item Let \(\cF\) be a semistable bundle on \(\cC\) such that \(\cF^*\) is cohomologically orthogonal to \(\cR\). Then the functor \(\Phi_{\cS}(-)\otimes \cO(1)\) preserves Jordan--H\"older filtration.
\end{enumerate}
\end{corollary}

\begin{proof}
Suppose that \(\cE\) admits a nontrivial Jordan--H\"older filtration \(0= \cE_0 \subset \cE_1 \subset \cE_2 \subset \cdots \subset \cE_k = \cE\). Repeated application of \cite[Proposition~3.3.1, Proposition~3.3.16]{CMRPL21} shows that all graded pieces \(\cE_i/\cE_{i-1}\) are stable Ulrich bundles. Moreover, because any extension of cohomologically orthogonal bundles is also orthogonal, we obtain a filtration
\[0= \cF_0 \subset \cF_1 \subset \cF_2 \subset \cdots \subset \cF_k = \cF\]
of bundles such that \(\cF_i/\cF_{i-1}\) are all orthogonal. Moreover, Propositions \ref{prop:Ulrichrankeven} and \ref{prop:Ulrichrankodd} imply that \(\mu(\cF_i/\cF_{i-1}) = \mu(\cF)\). Since each \(\cF_i/\cF_{i-1}\) is stable, the filtration is a Jordan--H\"older filtration of \(\cF\). This proves Item (1).

The proof of Item (2) is identical.
\end{proof}

%%%%%%%%%%%%%%%%%%%%%%%%%%%%%%%%%%%%%%%%%%%%%%%
\section{Conditional existence of Ulrich bundles}\label{sec:existenceUlrich}

The purpose of this section is twofold. First, we characterize the rank of Ulrich bundles on \(X\) in Theorems \ref{thm:mainthmeven} and \ref{thm:mainthmodd}. Second, we show that the existence of \textit{an} Ulrich bundle for a few low-rank cases for intersections of odd-dimensional quadrics implies both Theorems \ref{thm:mainthmeven} and \ref{thm:mainthmodd}.

\begin{definition}\label{def:modulispace}
\begin{enumerate}
  \item Let \(\cU_r(X)\) (resp. \(\rU_r(X)\)) be the moduli stack (resp. good moduli space) of rank \(r\) Ulrich bundles on \(X\). Note that any Ulrich bundle is semistable.
  \item Let \(\cM_\cC(r, d)\) (resp. \(\rM_\cC(r, d)\)) be the moduli stack (resp. good moduli space) of rank \(r\), degree \(d\) semistable vector bundles on \(\cC\).
\end{enumerate}
\end{definition}

Our goal is to prove that \(\cU_r(X)\) is nonempty for the relevant ranks and to study its geometry.

\subsection{Even-dimensional quadrics}\label{ssec:evenRaynaud}

Suppose \(g\ge2\). We treat the exceptional \(g = 1\) case in Remark~\ref{rmk:genus1}. Then \(\dim X=2g-1\), \(X\) is an intersection of two even-dimensional quadrics, and \(\cC=C\) is a hyperelliptic curve of genus \(g\).

Let \(p_C \in \rH^2(C, k)\) be the point class on \(C\) and \(h\in \rH^2(X, k)\) be the hyperplane class. Since \(g \ge 2\), \(\rH^1(X, k) = 0\), and
\[\rH^2(C \times X, k) \cong \rH^0(C, k) \otimes \rH^2(X, k) \oplus \rH^2(C, k) \otimes \rH^0(X, k) \cong k p_C \oplus k h.\]

\begin{lemma}\label{lem:chofShat}
We have
\[\ch_0(\cS) = 2^{g-1},\quad\ch_1(\cS) = (g+1)2^{g-2}p_C + 2^{g-2}h.\]
\end{lemma}

\begin{proof}
We have \(c_1(\cU) = -p_C - h\). The Chern character of \(\cR_U \cong \cF_U\) was computed in \cite[Proposition~4.5]{ES25} as \(2^g + g2^{g-1}p_C\).

Since \(\tau\) acts trivially on \(\rH^2(C, k)\) and \(\rH^0(C, k)\), \(\ch_1((\tau \times \mathrm{id}_X)^*\cS) = \ch_1(\cS)\). By Proposition~\ref{prop:filtrationeven}, we have
\[\ch(\cS) + \ch(\cU)(\tau \times \mathrm{id}_X)^*\ch(\cS) = p_{\cC}^*\ch(\cR) = 2^g + g2^{g-1}p_C.\]
Calculating low-degree terms, we obtain the result.
\end{proof}

\begin{proposition}\label{prop:Ulrichrankeven}
Let \(\cF\) be a rank-\(r\) degree-\(d\) vector bundle on \(C\) such that \(\Phi_{\cS}(\cF) = \cE(-1)\) for some Ulrich bundle \(\cE\) on \(X\). Then \(d = r(g-2)/2\) and \(\rank \cE = r2^{g-2}\).
\end{proposition}

\begin{proof}
From the orthogonality, we have \(\rH^*(C, \cF \otimes \cR) = 0\). By Riemann--Roch,
\[d2^g + rg2^{g-1} + r2^g(1-g) = \chi(C, \cF \otimes \cR) = 0.\]
This yields \(d = r(g-2)/2\).

To apply Grothendieck-Riemann--Roch, we compute
\[\ch(p_{\cC}^*\cF)\ch(\cS)\td(\cT_C) = (r+dp_C)\cdot\ch(\cS)\cdot (1+(1-g)p_C).\]
Expanding it, we obtain
\[r2^{g-1} + 2^{g-1}\left(d+r+r(1-g)/2\right)p_C + r2^{g-2}h + \dots.\]
Taking the push-forward, we have
\[\ch(\cE(-1)) = p_{X*}(\ch(p_{\cC}^*\cF)\ch(\cS)\td(\cT_C)) = 2^{g-1}((d+r)+\frac{r}{2}(1-g)) + \dots.\]
Therefore,
\[\rank \cE = \rank \cE(-1) = 2^{g-1}\left(d+r+r(1-g)/2\right) = r2^{g-2}. \qedhere\]
\end{proof}

Summing up the previous results, we obtain the following \emph{conditional} result.

\begin{theorem}\label{thm:maintheoremevenconditional}
Suppose that there is a rank \(r\), degree \(r(g-2)/2\) vector bundle \(\cF\) on \(C\) such that \(\cF^*\) is cohomologically orthogonal to \(\cR\). Then Theorem~\ref{thm:mainthmeven} is true for this value of \(r\).
\end{theorem}

\begin{proof}
Such bundles are semistable by Theorem~\ref{thm:existenceUlrich} and Proposition~\ref{prop:Fissemistable}, and their nonempty open locus meets the dense stable locus. Corollary~\ref{cor:stability} and Proposition~\ref{prop:Ulrichrankeven} therefore yield a stable Ulrich bundle of rank \(r2^{g-2}\).

Proposition~\ref{prop:FMofUlrichisvb} and Theorem~\ref{thm:existenceUlrich}, together with full faithfulness of \(\Phi_{\cS}\) and base change, give an equivalence of stacks
\[\Phi_{\cS}(-)(1)\colon\cM_C(r,r(g-2)/2)^{\perp\cR}\xrightarrow{\sim}\cU_{r2^{g-2}}(X)\]
with inverse \(\Phi_{\cS}^!(-\otimes\cO_X(-1))\). By Corollary~\ref{cor:Sequivalence} and closure under extensions, the orthogonal locus is saturated, so the induced morphism of good moduli spaces is an open embedding.
\end{proof}

The following observation reduces the problem to low-rank cases.

\begin{proposition}\label{prop:lowrankisenougheven}
\begin{enumerate}
  \item Suppose that \(g\) is even. If there are vector bundles \(\cF_2\) and \(\cF_3\) of ranks two and three whose duals are cohomologically orthogonal to \(\cR\), then for every \(r \ge 2\), there is a stable bundle \(\cF_r\) whose dual is cohomologically orthogonal to \(\cR\).
  \item Suppose that \(g\) is odd. If we have a rank-two vector bundle \(\cF_2\) whose dual is cohomologically orthogonal to \(\cR\), then for any even \(r \ge 2\), we have a rank-\(r\) stable bundle \(\cF_r\) whose dual is cohomologically orthogonal to \(\cR\).
\end{enumerate}
In particular, Theorem~\ref{thm:mainthmeven} is true for any \(r \ge 2\) with \(rg \equiv 0 \;(\mathrm{mod}\; 2)\).
\end{proposition}

\begin{proof}
We first consider even \(g\). By Theorem~\ref{thm:maintheoremevenconditional}, we may assume that \(\cF_2\) and \(\cF_3\) are stable. Any integer \(r \ge 2\) can be written as \(r = 2a + 3b\) for some nonnegative integers \(a, b\). Then \(\cF_2^{\oplus a}\oplus \cF_3^{\oplus b}\) is a semistable orthogonal bundle. Note that \(\cM_C(r, r(g-2)/2)\) is irreducible and the stable locus is nonempty open. Since the orthogonality is also an open condition, we may choose a rank \(r\) stable, hence indecomposable orthogonal bundle \(\cF_r\). This proves (1).

When \(g\) is odd, for any even \(r\), we may choose \(\cF_2^{\oplus r/2}\) and apply the same argument.
\end{proof}

\begin{remark}\label{rmk:genus1}
Suppose \(g = 1\), so \(X\) is an elliptic curve of degree 4 in \(\PP^3\). On a curve \(X\), a vector bundle \(\cE\) is Ulrich if and only if \(\rH^0(\cE(-1)) = \rH^1(\cE(-1)) = 0\). Here \(\cE(-1)\) is a degree zero bundle. Atiyah showed the existence of indecomposable bundles \(\cF\) of degree zero on \(X\) of arbitrary rank \(r\) and classified them \cite[Theorem~5]{Ati57}. Moreover, the proof of the theorem also shows that a general degree zero indecomposable bundle is Brill-Noether general, hence satisfies \(\rH^0(\cF) = \rH^1(\cF) = 0\). Thus, taking \(\cE := \cF(1)\), we obtain Ulrich bundles of arbitrary rank. By definition, \(\rU_{r2^{g-2}}(X) = \rU_{r/2}(X)\) is a subspace of \(\rM_X(r/2, 2r) \cong \sym^{r/2}(X) \cong \rM_X(r, r/2)\) \cite[Theorem~1]{Tu93}.

Thus, the statements of Theorem~\ref{thm:mainthmeven} for \(g = 1\) hold, on the level of good moduli spaces. However, the proof does not rely on the construction of Fourier-Mukai transform. Moreover, Corollary~\ref{cor:evenintersection} does not hold. When \(r = 2\), \(\rM_X(2, -1) \cong X\) is not uniruled, and for \(r \ge 4\), \(\rU_{r2^{g-2}}(X)\subset \sym^{r/2}(X)\) is of dimension \(r/2\).
\end{remark}

\subsection{Odd-dimensional quadrics}

The argument for odd-dimensional quadrics requires some extra work, as not all of the components of \(\cM_{\PP_{2g+1}^1}(r, d)\) have a stable bundle. In this section, \(\dim X = 2g - 2\) and \(\cC = \PP_{2g+1}^1\).

\begin{lemma}\label{lem:chofS}
Let \(\xi:=c_1(\cO_{\cC}(1))\). Then for the relative spinor bundle \(\cS\) on \(X\), we have
\[\ch_0(\cS) = 2^{g-1}, \quad \ch_1(\cS) = (g+1)2^{g-2}\xi + 2^{g-2}h.\]
\end{lemma}

\begin{proof}
Let \(\cS'\) be the restriction to \(C\times X\) of the even-dimensional relative spinor bundle. The bundle \(\cS\) is the descent of \(\cS'\); equivalently,
\[\cS'\cong(\rho\times\operatorname{id}_X)^*\cS.\]
Hence Lemma~\ref{lem:chofShat} gives
\[\ch_0(\cS')=2^{g-1},\qquad\ch_1(\cS')=(g+1)2^{g-2}p_C+2^{g-2}h.\]
Since \(\rho^*\xi=p_C\), these are precisely the pullbacks of the asserted formulas for \(\cS\).
\end{proof}

\begin{proposition}\label{prop:Ulrichrankodd}
Let \(\cF\) be a rank-\(r\) degree-\(d\) vector bundle on \(\PP_{2g+1}^1\) such that \(\Phi_{\cS}(\cF) = \cE(-1)\) for some Ulrich bundle \(\cE\) on \(X\). Then \(d = r(g-3)/4\) and \(\rank \cE = r2^{g-3}\).
\end{proposition}

\begin{proof}
Let \(q_i\) denote the \(i\)-th stacky point. Since \(\omega_{\PP_{2g+1}^1} = \pi^*\omega_{\PP^1} \otimes \cO(\sum q_i)\), we have \(\td(\cT_{\PP_{2g+1}^1}) = 1 + \frac{3-2g}{2}\xi\). Riemann--Roch yields
\[\chi(\cC, \cF \otimes \cR) = 2^{g-2}(4d + r(3-g)).\]
Orthogonality therefore implies \(d = r(g-3)/4\).

Let \(\cS_i:=\cS|_{\{q_i\}\times X}\), and let \(\cS_i^\pm\) denote its \(\pm1\)-eigenbundles under \(\tau\). Write \(\dim\cF_{q_i}^+=a_i\) and \(\dim\cF_{q_i}^-=b_i\). We know that \(\rank \cS_i^+ = \rank \cS_i^-\).

The stacky Grothendieck--Riemann--Roch theorem yields
\[\ch(\Phi_{\cS}(\cF))=p_{X*}(\ch(p_{\cC}^*\cF)\ch(\cS)\td(\cT_{\PP_{2g+1}^1}))+\frac14\sum_{i=1}^{2g+1}(a_i - b_i)(\ch(\cS_i^+)-\ch(\cS_i^-)).\]

The latter stacky term does not contribute to \(\ch_0\) once we have \(\rank \cS_i^+ = \rank \cS_i^-\), hence we may ignore it.

We have
\[\ch(p_{\cC}^*\cF)\ch(\cS)\td(\cT_{\PP_{2g+1}^1})=r2^{g-1} + 2^{g-2}(4d+r(4-g))\xi + \dots.\]
Taking the pushforward and noting that \(p_{X*}\xi = 1/2\), we have
\[\ch(\cE(-1)) = p_{X*}(\ch(p_{\cC}^*\cF)\ch(\cS)\td(\cT_{\PP_{2g+1}^1})) = 2^{g-3}(4d+r(4-g)) + \dots.\]
From \(d = r(g-3)/4\), we have \(\rank \cE = 2^{g-3}(4d + r(4-g)) = r2^{g-3}\).
\end{proof}

In the following sections, we construct explicit low-rank vector bundles that are cohomologically orthogonal to \(\cR\). To state the result more explicitly, we use the equivalence between coherent sheaves on \(\cC\) and parabolic sheaves on \((\PP^1, \bp)\). Here \(\bp = (p_1, p_2, \dots, p_{2g+1})\) is the set of the images of the stacky points under \(\cC \to \PP^1\); see \S\ref{ssec:parabolicbundle}. By abuse of notation, we will use the same notation for the identified objects in these two categories.

\begin{theorem}\label{thm:initialconst}
\begin{enumerate}
  \item Let \(g \ge 3\) be odd. Let \(\cG_2\) be a parabolic bundle on \((\PP^1, \bp)\) with underlying bundle \(\underline{\cG}_2 = \cO((1-g)/2)^{\oplus 2}\), bit vector \(\bm=(m_i)\) with \(g+1\) nonzero entries, and with a general choice of flags.
  \item Let \(g \ge 3\) be odd. Let \(\cG_3\) be a parabolic bundle on \((\PP^1, \bp)\) with underlying bundle \(\underline{\cG}_3 = \cO((1-g)/2)^{\oplus 3}\), bit vector \(\bn=(n_i)\) with \(3(g+1)/2\) nonzero entries, and a general choice of flags.
  \item Let \(g \ge 2\) be even. Let \(\cG_2^{\mathrm{ev}}\) be a parabolic bundle on \((\PP^1, \bp)\) with underlying bundle \(\underline{\cG}_2^{\mathrm{ev}} = \cO((2-g)/2)\oplus \cO(-g/2)\), bit vector \(\bm=(m_i)\) with \(g+1\) nonzero entries, and a general choice of flags.
\end{enumerate}
Then \(\cG_2\), \(\cG_3\), and \(\cG_2^{\mathrm{ev}}\) are stable and cohomologically orthogonal to \(\cR\).
\end{theorem}

\begin{proof}
The cohomological orthogonality follows from the computations in \S\S\ref{sec:existence} and \ref{sec:prf_main_thm}. Proposition~\ref{prop:Fissemistable} then gives the semistability of \(\cG_2\), \(\cG_3\), and \(\cG_2^{\mathrm{ev}}\). Tensoring \(\cG_2\) and \(\cG_3\) by \(\cO_{\PP^1}((g-1)/2)\) preserves parabolic stability and makes their underlying bundles trivial.

Let \(\cH\subset\cO_{\PP^1}^{\oplus r}\) be a saturated subbundle of rank \(s\) and degree \(-d\), and set \(m(\cH):=\sum_i\dim(\underline{\cH}|_{p_i}\cap V_i)\). The parameter dimensions are \(r(d+1)-1\) for \(s=1\) and \(3d+2\) for \(s=2\), the latter via \(\cO_{\PP^1}^{\oplus3}\twoheadrightarrow\cO_{\PP^1}(d)\), while each incidence condition has codimension \(r-s\). Since only finitely many \(d\) can destabilize, general flags give
\begin{align*}
(r,s)=(2,1):\quad &m(\cH)\le2d+1, &\parmu(\cH)&\le\frac12<\frac{g+1}{4},\\
(r,s)=(3,1):\quad &2m(\cH)\le3d+2, &\parmu(\cH)&\le\frac12-\frac d4<\frac{g+1}{4},\\
(r,s)=(3,2):\quad &m(\cH)\le3d+2. &&
\end{align*}
Here \((g+1)/4\) is the parabolic slope of both twisted bundles.

A destabilizing rank-two subbundle \(\cH\subset\cG_3\) would satisfy \(m(\cH)\ge2d+g+1\). Hence
\[m(\cH)\le3d+2\quad\Longrightarrow\quad d\ge g-1,\qquad m(\cH)\le\frac{3(g+1)}2\quad\Longrightarrow\quad d\le\frac{g+1}{4}<g-1,\]
a contradiction. Therefore, \(\cG_2\) and \(\cG_3\) are stable.

Finally, if \(\cG_2^{\mathrm{ev}}\) were not stable, its semistability would give a line subbundle \(\cL\) of degree \(e\) with \(\parmu(\cL)=\parmu(\cG_2^{\mathrm{ev}})\). Hence
\[4e+2m(\cL)=3-g,\]
which is impossible since \(g\) is even. Therefore, \(\cG_2^{\mathrm{ev}}\) is stable.
\end{proof}

It remains to construct indecomposable bundles orthogonal to \(\cR\) in every admissible rank.

\begin{theorem}\label{thm:maintheoremoddconditional}
\begin{enumerate}
  \item Suppose that \(g \ge 3\) is odd. For each rank \(r \ge 2\), there is an indecomposable vector bundle \(\cF_r\) whose dual is orthogonal to \(\cR\).
  \item Suppose that \(g \ge 2\) is even. For each even \(r \ge 2\), there is an indecomposable vector bundle \(\cF_r\) whose dual is orthogonal to \(\cR\).
\end{enumerate}
\end{theorem}

\begin{proof}
Taking the dual bundle preserves indecomposability. Thus, it is sufficient to construct an indecomposable bundle orthogonal to \(\cR\).

\textbf{Case 1.} \(g\) is odd.

We choose \(\bm\) and \(\bn\) so that \(n_i \ge m_i\) for all \(i\). From the definition of parabolic morphisms, for two parabolic bundles \(\cG := (\underline{\cG}, \{V_i\})\) and \(\cH := (\underline{\cH}, \{W_i\})\), there is a short exact sequence
\begin{align}\label{eq:par_ex_seq}
0\to \cpHom(\cG,\cH)\to\cHom\bigl(\underline{\cG},\underline{\cH}\bigr)\to\bigoplus_{i=1}^{2g+1}\rHom(V_i,\underline{\cH}|_{p_i}/W_i)\otimes k(p_i)\to 0,
\end{align}
which induces a long exact sequence of cohomology groups:
\begin{equation}\label{eqn:par_long_seq}
\begin{split}
  0\to \pHom(\cG,\cH)\to&\rHom\bigl(\underline{\cG},\underline{\cH}\bigr)\to\bigoplus_{i=1}^{2g+1}\rHom(V_i,\underline{\cH}|_{p_i}/W_i)\otimes k(p_i)\\
  &\to \pExt^1(\cG, \cH) \to \Ext^1(\underline{\cG}, \underline{\cH}) \to 0.
\end{split}
\end{equation}
Applying this long exact sequence to stable bundles \(\cG_2\) and \(\cG_3\), we obtain
\[\dim \pExt^1(\cG_2, \cG_2) = g - 2, \quad \dim \pExt^1(\cG_3, \cG_3) = 3g - 5,\]
and
\[\dim \pExt^1(\cG_2, \cG_3) = \dim \pExt^1(\cG_3, \cG_2) = 2g - 4.\]
Since \(g \ge 3\), all extension groups are nontrivial.

For each \(t \in \ZZ_{>0}\), construct \(\cG_{2,t}\) as follows. Set \(\cG_{2,1}:=\cG_2\), and suppose that \(\cG_{2,t-1}\) has been constructed. Then choose \(\cG_{2,t}\) to be a nontrivial extension
\begin{equation}\label{eqn:iteration}
0 \to \cG_2 \to \cG_{2,t} \to \cG_{2,t-1} \to 0.
\end{equation}
There is a nontrivial extension because \(\pExt^1(\cG_{2,t-1},\cG_2)\to\pExt^1(\cG_2,\cG_2)\) is surjective and one may lift a nontrivial extension.

Applying \(\pHom(\cG_2,-)\) to \eqref{eqn:iteration}, we inductively obtain \(\pHom(\cG_2,\cG_{2,t})\cong k\). Since every simple factor of \(\cG_{2,t}\) is isomorphic to \(\cG_2\), this shows that \(\cG_{2,t}\) has a unique simple subobject and hence is indecomposable.

If \(r\) is even, take \(\cG':=\cG_{2,r/2}\). If \(r\) is odd, the same argument applies by replacing the last factor by \(\cG_3\), using \(\pExt^1(\cG_3,\cG_2)\ne0\).

\textbf{Case 2.} \(g\) is even and \(\ge 4\).

In this case, we may apply the same argument. For \(\cG_2^{\mathrm{ev}}\) in Theorem~\ref{thm:initialconst}, using \eqref{eqn:par_long_seq}, we have \(\dim \pExt^1(\cG_2^{\mathrm{ev}}, \cG_2^{\mathrm{ev}}) = g-2\). Therefore, if \(g \ge 4\), the same iterative construction with \(\cG_2^{\mathrm{ev}}\) in place of \(\cG_2\) gives an indecomposable bundle of every even rank \(r\), orthogonal to \(\cR\).

\textbf{Case 3.} \(g = 2\).

In this case, one may choose \(\cG_{2, I}^{\mathrm{ev}} \in \cM_{(\PP^1, \bp)}(2, -1, \bm_I)\) and \(\cG_{2, J}^{\mathrm{ev}} \in \cM_{(\PP^1, \bp)}(2, -1, \bm_J)\), where \(\bm_I = (1, 1, 1, 0, 0)\) and \(\bm_J = (0, 0, 1, 1, 1)\). Both are stable. Using \eqref{eqn:par_long_seq}, we have \(\pExt^1(\cG_{2, I}^{\mathrm{ev}}, \cG_{2, J}^{\mathrm{ev}})\cong\pExt^1(\cG_{2, J}^{\mathrm{ev}}, \cG_{2, I}^{\mathrm{ev}})\cong k\). Starting with either bundle and taking nontrivial extensions successively, alternating \(\cG_{2,I}^{\mathrm{ev}}\) and \(\cG_{2,J}^{\mathrm{ev}}\), the same proof gives an indecomposable bundle of every even rank \(r\), orthogonal to \(\cR\).
\end{proof}

Using these bundles, we may construct higher-rank indecomposable Ulrich bundles. The proof is a minor adjustment of that of Theorem~\ref{thm:maintheoremevenconditional}.

\begin{proposition}\label{prop:maintheoremoddconditional}
Suppose that there is an indecomposable rank \(r\) parabolic bundle \(\cG\) whose dual is orthogonal to \(\cR\). Then there is an indecomposable Ulrich bundle of rank \(r2^{g-3}\), and there is an open embedding
\[\cU_{r2^{g-3}}(X)\hookrightarrow\cM_{\cC}(r,r(g-3)/4).\]
In particular, Theorem~\ref{thm:mainthmodd} is true for any \(r\ge2\) with \(r(g+1)\equiv0\pmod2\).
\end{proposition}

\begin{proof}
As in the proof of Theorem~\ref{thm:maintheoremevenconditional}, the Fourier--Mukai functors identify \(\cU_{r2^{g-3}}(X)\) with \(\cM_{\cC}(r,r(g-3)/4)^{\perp\cR}\). Since cohomological orthogonality is open, this is an open substack of \(\cM_{\cC}(r,r(g-3)/4)\), giving the asserted open embedding. Moreover, \(\Phi_{\cS}(\cG)(1)\) is indecomposable because \(\Phi_{\cS}\) is fully faithful.

It remains to prove the existence of multiple components. Choosing \(\bm\) amounts to choosing \(g+1\) of the \(2g+1\) parabolic points. Distinct choices give bundles \(\cG_2^*\) in distinct connected components of \(\cM_{\cC}(2,(g-3)/2)\). Hence \(\cU_{2\cdot2^{g-3}}(X)\) has at least \(\binom{2g+1}{g+1}\) connected components. The same construction gives multiple components in every higher admissible rank.
\end{proof}

There is one further reduction: the odd-dimensional case implies the even-dimensional case.

\begin{proposition}\label{prop:oddimplieseven}
The existence part of Theorem~\ref{thm:mainthmodd} with parameter \(g\) implies the existence part of Theorem~\ref{thm:mainthmeven} for parameter \(g -1\).
\end{proposition}

\begin{proof}
Suppose first that \(g\) is odd. By the assumed existence part of Theorem~\ref{thm:mainthmodd}, there are Ulrich bundles \(\cE_2\) and \(\cE_3\) on \(X\) of ranks \(2\cdot 2^{g-3}\) and \(3\cdot 2^{g-3}\), respectively. If \(g\) is even, take only an Ulrich bundle \(\cE_2\) of rank \(2\cdot 2^{g-3}\). Let \(Y:=X\cap H\) be a smooth hyperplane section. Then \(Y\) is the intersection of two even-dimensional quadrics with parameter \(g-1\), and the restrictions \(\cE_i|_Y\) are Ulrich. Set
\[\cF_i:=\Phi_{\cS}^!(\cE_i|_Y(-1)).\]
By Proposition~\ref{prop:FMofUlrichisvb}, the \(\cF_i\) are vector bundles on the hyperelliptic curve of genus \(g-1\). Equation \eqref{eqn:duality} and the Ulrich vanishings show that the bundles \(\cF_i^*\) are cohomologically orthogonal to \(\cR\), while Proposition~\ref{prop:Ulrichrankeven} gives \(\rank\cF_i=i\). Since \(g-1\) is even when \(g\) is odd and odd when \(g\) is even, Proposition~\ref{prop:lowrankisenougheven} gives the existence part of Theorem~\ref{thm:mainthmeven} for parameter \(g-1\).
\end{proof}

Therefore, we may focus on the odd-dimensional quadrics.

\begin{remark}\label{rmk:dimension}
The moduli space \(\rM_{\cC}(r, d)\) has multiple connected components, and they are not equidimensional. This can be seen using parabolic bundles. Each component of \(\rM_{\cC}(r, d)\) is identified with the moduli space \(\rM_{(\PP^1, \bp)}(r, d', \bm)\) of parabolic bundles over \(\PP^1\) with parabolic points \(\bp = (p_1, p_2, \dots, p_{2g+1})\), and the multiplicity \(\bm = (m_1, m_2, \dots, m_{2g+1})\) with
\[d = d' + \frac{1}{2}\sum m_i.\]
If the stable locus is nonempty, this component has dimension
\[
1 - r^2 + \sum_{i=1}^n m_i (r-m_i).
\]

In \S\ref{sec:existence}, when \(g\) is odd, we construct a component \(U\) admitting an injective morphism into \(\rM_{(\PP^1, \bp)}(2, (1-g), \bm)\) where
\[m_i = \begin{cases} 1, & 1 \le i \le g+1,\\ 0, & g+2 \le i \le 2g+1. \end{cases}\]
Hence \(\dim \rM_{(\PP^1, \bp)}(2, (1-g), \bm) = g - 2\).

More generally, for any \((g+1)\)-subset \(I \subset \{1, 2, \dots, 2g+1\}\), define \(\bm_I\) by
\[\bm_{I,i} = \begin{cases} 1, & i \in I,\\ 0, & i \notin I. \end{cases}\]
Then the nonempty orthogonality locus in \(\rM_{(\PP^1, \bp)}(2, 1-g, \bm_I)\) identifies with another connected component of \(\rU_{2\cdot 2^{g-3}}(X)\). Thus, \(\rU_{2\cdot 2^{g-3}}(X)\) has at least \(\binom{2g+1}{g+1}\) connected components. The components obtained in this way all have dimension \(g-2\). We do not know if, up to degree-zero twists, there are any extra components of \(\rU_{2\cdot 2^{g-3}}(X)\) or whether \(\rU_{2\cdot 2^{g-3}}(X)\) is equidimensional. Our numerical search for small values of \(g\) suggests that there are no additional components up to degree-zero twists, but at this point, we do not have a proof.

In higher rank, equidimensionality fails; let \(g\ge3\) be odd, and set \(I=\{1,\dots,g+1\}\) and \(J=\{g+1,\dots,2g+1\}\). Choose stable orthogonal rank-two bundles \(\cF_2,\cF_2'\) of types \(\bm_I,\bm_J\), respectively. Since \(\dim\pExt^1(\cF_2,\cF_2)=g-2\), the local deformation quotient at \(\cF_2^{\oplus2}\) is \(\operatorname{Mat}_2(k)^{g-2}\git\operatorname{PGL}_2\) under simultaneous conjugation. This gives
\[\dim\rM_{(\PP^1,\bp)}(4,2-2g,2\bm_I)=\begin{cases} 2,&g=3,\\4g-11,&g\ge5. \end{cases}\]
At \(\cF_2\oplus\cF_2'\), both mixed extension spaces have dimension \(2g-3\) by \eqref{eqn:par_long_seq}, and the effective stabilizer \(\mathbb G_m\) acts on them with opposite weights. Hence
\[\dim\rM_{(\PP^1,\bp)}(4,2-2g,\bm_I+\bm_J)=2(g-2)+2(2g-3)-1=6g-11.\]
Openness of orthogonality gives components of \(\rU_{4\cdot2^{g-3}}(X)\) of these distinct dimensions.
\end{remark}

%%%%%%%%%%%%%%%%%%%%%%%%%%%%%%%%%%%%%%%%%%%%%%%%%%%%
\section{Reduction to parabolic bundle}\label{sec:reductionparabolic}

Eisenbud and Schreyer realized the Raynaud bundle for \(\dim V=2g+2\) by sheafifying the even part of an algebraic Clifford module \cite[\S4]{ES25}. The Ext construction and its Clifford action are defined in either parity. We identify the resulting bundle on \(\cC\) with the geometric Raynaud bundle and compute its parabolic flags in odd dimension.

\subsection{Algebraic realization of the Raynaud bundle}\label{ssec:Raynaudalgebraic}

Fix a maximal common isotropic subspace \(U\subset V\). Set \(P_X:=k[V^*]/(q_0,q_\infty)\) and \(P_U:=k[V^*]/(U^\perp)\).

\begin{definition}\label{def:Raynaudmodule}
The algebraic \textbf{Raynaud module} associated with \(U\) is the graded right \(\Cl\)-module
\[F_U:=\Ext_{P_X}^\bullet(P_U,k),\]
with the Yoneda action induced by \(\Ext_{P_X}^\bullet(k,k)\cong\Cl\).
\end{definition}

Let \(\cF_U\) be the coherent right \(\widetilde{\cC\ell}_0\)-module on \(\cC\) corresponding to the sheafification of \(F_U\) under the equivalences in \S\ref{ssec:sheafification}. If \(\dim V=2g+2\), the even part of \(F_U\) sheafifies to a vector bundle on \(C\) by \cite[Proposition~4.5]{ES25}. If \(\dim V=2g+1\), Proposition~\ref{prop:Ray_act} allows us to apply the construction of \S\ref{ssec:parabolicbundle} to \(F_U\), giving a vector bundle on \(\PP_{2g+1}^1\). Hence \(\cF_U\) is a vector bundle in either parity.

Write \(\cR_U\) for the Raynaud bundle of Definition~\ref{def:Raynaud} constructed from the same isotropic subspace \(U\), and retain the half-pencil line bundle \(\cO_{\cC}(1)\) fixed in \eqref{eqn:half_pencil_bundle}.

\begin{proposition}\label{prop:Raynaudidentical}
There is an isomorphism \(\cF_U\cong\cR_U\) of right \(\widetilde{\cC\ell}_0\)-modules. Moreover, there is an isomorphism of graded right \(\Cl\)-modules
\[F_U\cong\bigoplus_{n\in\ZZ}H^0\!\left(\cC, \cR_U\otimes\cO_{\cC}(n) \right).\]
\end{proposition}

\begin{proof}
The Clifford action exhibits \(\cF_U\) as a splitting bundle for \(\widetilde{\cC\ell}_0\), while Corollary~\ref{cor:Cliffordendomorphismalgebra} does the same for \(\cR_U\). Morita uniqueness gives \(\cF_U\cong\cR_U\otimes\cL\) for some \(\cL\in\Pic(\cC)\).

For any irreducible half-spin representation \(M^{\mathrm{ev}}\), extend it to \(M=M^{\mathrm{ev}}\oplus M^{\mathrm{odd}}\) and set
\[\ell_U(M):=\bigcap_{u\in U}\ker\bigl(\mathord{\cdot}u:M^{\mathrm{ev}}\longrightarrow M^{\mathrm{odd}}\bigr).\]
This subspace is one-dimensional. Proposition~\ref{prop:Ray_act} shows that \(1\in\Ext_{P_X}^0(P_U,k)\) lies in \(\ell_U(F_U)\). It therefore defines a nowhere-vanishing section of \(\ell_U(\cF_U)\), so \(\ell_U(\cF_U)\cong\cO_{\cC}\).

The canonical section \(\sigma_U\) of \(\cO_{\OGr_{g+1}}(D_U)\) determines the corresponding line in \(\cR_U\). On each smooth geometric fiber, choose a decomposition \(V=W_1\oplus W_2\) into maximal isotropic subspaces with \(U\subset W_1\), where \(V\) is the even-dimensional ambient quadratic space fixed in \S\ref{sec:spinorbundle}. Under the identification \(\rH^0(\OGr_{g+1},\cO_{\OGr_{g+1}}(D_U))\cong\bigwedge^{\mathrm{ev}}W_2\), the section \(\sigma_U\) corresponds to \(1\in\bigwedge^0W_2\) \cite[\S1]{Muk95}. Since \(W_1\) acts by contraction, this vector is annihilated by \(U\). The resulting nowhere-vanishing section descends together with \(\cR_U\), and hence \(\ell_U(\cR_U)\cong\cO_{\cC}\).

The Morita isomorphism induces \(\ell_U(\cF_U)\cong\ell_U(\cR_U)\otimes\cL\), so \(\cL\cong\cO_{\cC}\). The section-module formula follows from \cite[Proposition~4.5]{ES25} in even dimension and from Lemma~\ref{lem:reflexive_R_module} and \cite[Lemma~3.2]{HO20} in odd dimension.
\end{proof}

For the remainder of this subsection, assume that \(\dim V=2g+1\), so that \(\cC=\PP_{2g+1}^1\). Retain the notation of \S\ref{ssec:isotropicspace}; the vectors \(\overline e_\ell\), the coefficients \(a_{i,\ell}\), and the vectors \(x_i^+\) are defined in \S\ref{ssec:KT_res}.

\begin{theorem}\label{thm:raynaud_parabolic}
The parabolic bundle associated with \(\cF_U\cong\cR_U\) is
\[\cR:= \left( \bigoplus_{d\ge0} \bigwedge^{2d}(V/U)\otimes \cO_{\PP^1}(-d), \{W_\ell\}_{\ell=1}^{2g+1} \right),\]
where
\[
W_\ell= \begin{cases}
  \operatorname{im}(\varepsilon_{\overline e_\ell}), &1\le\ell\le g+1,\\[4pt]
  \operatorname{im}\!\left( \varepsilon_{\overline e_\ell} +\displaystyle\sum_{i=1}^{g+1} a_{i,\ell}(\lambda_\ell-\lambda_i)\iota_{x_i^+} \right), &g+2\le\ell\le2g+1,
\end{cases}
\]
and their annihilators with respect to the pairing \eqref{eq:perf_pair} are
\[
W_\ell^\perp= \begin{cases}
  \operatorname{im}(\iota_{\overline e_\ell}), &1\le\ell\le g+1,\\[4pt]
  \operatorname{im}\!\left( \iota_{\overline e_\ell} +\displaystyle\sum_{i=1}^{g+1} a_{i,\ell}(\lambda_\ell-\lambda_i)\varepsilon_{x_i^+} \right), &g+2\le\ell\le2g+1.
\end{cases}
\]
\end{theorem}

We will compute the Clifford action in either parity; its odd-dimensional specialization proves the theorem.

\subsection{Koszul--Tate resolution and Clifford action}\label{ssec:KT_res}

Recall that \(P_X\) is the homogeneous coordinate ring of \(X = Q_0 \cap Q_\infty\), and \(P_U = k[V^*]/(U^\perp)\) is that of the linear subspace \(\PP U\). We begin by constructing a \(P_X\)-free resolution of \(P_U\) using \cite[Theorem~4]{Tat57}, with notation compatible with \cite[Proposition~4.4]{ES25}.

Let \(G=\langle s, t\rangle\), and fix a tensor
\[\ga\in U^\perp\otimes_k V^*\otimes_k G^*\]
whose image under the natural map
\[G\xrightarrow{\ga}U^\perp\otimes_kV^*\xrightarrow{mult}\Sym^2(V^*)\]
sends \(s\mapsto q_0\) and \(t\mapsto q_\infty\). Such a tensor will be constructed explicitly in \S\ref{ssec:explicit_flags}.

Using the canonical identification \(V/U\cong (U^\perp)^*\), choose an expression
\[\mathbf 1_{U^\perp}=\sum_i u_i\otimes \overline v_i\in U^\perp\otimes_k (V/U)\cong \rEnd(U^\perp).\]
Let \(\iota\) and \(\varepsilon\) denote contraction and left exterior multiplication, respectively.

\begin{proposition}[{\cite[Theorem~4]{Tat57}, \cite[\S2]{Eis80}}]\label{prop:KT_resolution}
A \(P_X\)-free resolution of \(P_U\) can be written as
\[\left(\bigwedge\nolimits^\bullet U^\perp\right)\otimes_k\bigl(P_X\otimes_k \Sym^\bullet(G)\bigr),\qquad\deg s=\deg t=2,\]
with intrinsic differential \(\partial=(\iota_\bullet\otimes\bullet+\varepsilon_\bullet\otimes\iota_\bullet\ga)(\mathbf 1_{U^\perp})\), that is,
\[\partial=\sum_i\Bigl(\iota_{\overline v_i}\otimes u_i+\varepsilon_{u_i}\otimes\ga(\overline v_i)\Bigr).\]
Here \(u_i\in U^\perp\subset V^*\) acts on \(P_X\) by multiplication, and \(\ga(\overline v_i)\in V^*\otimes_k G^*\) acts on \(P_X\otimes_k\Sym^\bullet(G)\) via multiplication on the \(P_X\)-factor and the natural derivation on the \(\Sym^\bullet(G)\)-factor.
\end{proposition}

We now compute the induced right \(\Cl\)-action on \(F_U:=\Ext^\bullet_{P_X}(P_U,k)\) using the above \(P_X\)-free resolution. Recall that the fixed basis \(e_1,\dots,e_{2g+1}\) of \(V\) from \S\ref{ssec:isotropicspace} diagonalizes \(q_0\) and \(q_\infty\), and that \(x_1,\dots,x_{2g+1}\) are the dual basis. Denote the image of \(e_j\) under the projection \(V \to V/U\) by \(\overline e_j\). Write the differential as
\begin{equation}\label{eq:diff_coeff}
\partial=\sum_{i,j}x_j\bigl(\iota_{\overline v_i}\otimes u_i(e_j)+\varepsilon_{u_i}\otimes\ga(\overline v_i,e_j)\bigr) =\sum_j x_j\bigl(\iota_{\overline e_j}+\varepsilon_{\ga(e_j)}\bigr).
\end{equation}
Here the last equality follows from \(\sum_i u_i\otimes \overline v_i=\mathbf 1_{U^\perp},\) which implies
\[\sum_i u_i(e_j)\overline v_i=\overline e_j,\qquad\sum_i\ga(\overline v_i,e_j)u_i=\ga(e_j)\in U^\perp\otimes G^*.\]
In particular, the resolution is minimal and \(\rHom_{P_X}(-,k)\) has the zero differential. Then the degree-\(n\) part of \(F_U\) is canonically identified with
\[\rHom_k\!\left( \bigoplus_d \bigwedge^{n-2d}U^\perp\otimes_k \Sym^d(G),\,k \right) \cong \bigoplus_d \bigwedge^{n-2d}(V/U)\otimes_k \Sym^d(G^*)\]
via the perfect pairing
\begin{equation}\label{eq:perf_pair}
\langle-,-\rangle: \left( \bigoplus_d \bigwedge^{n-2d}(V/U)\otimes_k \Sym^d(G^*) \right) \otimes \left( \bigoplus_d \bigwedge^{n-2d}U^\perp\otimes_k \Sym^d(G) \right) \longrightarrow k.
\end{equation}

Take \(\alpha\) in the degree-\(n\) part of \(F_U\), and let \(u\in\bigoplus_d\bigwedge^{n+1-2d}U^\perp\otimes_k\Sym^d(G).\) Then
\[
\langle\alpha\cdot e_j,u\rangle = \langle\alpha,(\iota_{\overline e_j}+\varepsilon_{\ga(e_j)})u\rangle = \left\langle (\varepsilon_{\overline e_j}+\iota_{\ga(e_j)})\alpha,u\right\rangle,
\]
where the first equality follows from \cite[Lemma~3.5]{ES25} and \eqref{eq:diff_coeff}, and the second from the adjointness of contraction and exterior multiplication, together with the adjointness between the natural derivation action of \(G^*\) on \(\Sym^\bullet(G)\) and multiplication on \(\Sym^\bullet(G^*)\). Hence
\begin{equation}\label{eqn:Cliffordaction}
\alpha\cdot e_j=(\varepsilon_{\overline e_j}+\iota_{\ga(e_j)})\alpha,
\end{equation}
and by linearity, this extends to every \(v\in V\). In summary:

\begin{proposition}\label{prop:Ray_act}
The Raynaud module \(F_U\) is isomorphic, as a graded right \(\Cl\)-module, to
\[F_U\cong \bigoplus_{d\ge0}\bigwedge^d(V/U)\otimes_k k[s,t](-d), \qquad \deg s=\deg t=2,\]
with the Clifford action
\[\alpha\cdot v=\bigl(\varepsilon_{\overline v}+\iota_{\ga(v)}\bigr)\alpha \qquad\text{for all }\alpha\in F_U,\;v\in V.\]
\end{proposition}

Here we also denote by \(s,t\) the basis of \(G^*\) dual to the chosen basis of \(G\).

\begin{remark}
The action depends on the choice of \(\ga\), but only up to conjugation. Let \(\ga'\) be another tensor satisfying the same condition. Since \(\ga\) and \(\ga'\) have the same image in \(\rHom(G,\Sym^2(V^*))\), the tensor
\[\eta:=\ga'-\ga\in U^\perp\otimes_k V^*\otimes_k G^*\]
is skew-symmetric in the two \(V\)-slots; in particular,
\[\eta(\overline v,u)=-\eta(\overline u,v)=0,\qquad\text{for all }u\in U,\;v\in V\]
and indeed \(\eta\) descends to \(\bigwedge^2U^\perp\otimes_k G^*\).

Let \(\iota_\eta\) be the induced contraction operator on \(\bigwedge^\bullet(V/U)\otimes_k \Sym^\bullet(G^*)\). Then
\[\iota_\eta\,\iota_{\ga(v)}=\iota_{\ga(v)}\,\iota_\eta, \qquad \iota_\eta\,\varepsilon_{\overline v} = \varepsilon_{\overline v}\,\iota_\eta+\iota_{\eta(v)}.\]
Since the second iterated commutator vanishes, the Campbell identity gives
\[e^{\iota_\eta}\bigl(\varepsilon_{\overline v}+\iota_{\ga(v)}\bigr)e^{-\iota_\eta} = \varepsilon_{\overline v}+\iota_{\ga(v)+\eta(v)}.\]
Hence the action induced from \(\ga'=\ga+\eta\) is conjugate to that from \(\ga\). In particular, the resulting \(\Cl\)-module structure is well-defined up to isomorphism.
\end{remark}

%%%%%%%%%%%%%%%%%%%%%%%%%%%%%%%%%%%%%%%%%%%%%%%%%%%%
\subsection{An explicit description}\label{sec:exp_des}

Recall that \(\overline e_j\) is the image of \(e_j\) under the projection \(V \to V/U\). We now use \(\overline e_1,\dots,\overline e_{g+1}\) as coordinates on the quotient \(V/U\). We express the tensor defining the Clifford action and compute the resulting parabolic flags of the Raynaud bundle.

We first justify that these vectors form a basis.

\begin{lemma}\label{lem:adapted_matrix}
The vectors \(\overline e_1,\dots,\overline e_{g+1}\) form a basis of \(V/U\). Hence, for each \(g+2\le \ell\le2g+1\), there are unique scalars \(a_{i,\ell}\in k\) such that
\[\overline e_\ell=\sum_{i=1}^{g+1}a_{i,\ell}\overline e_i .\]
Equivalently, the vectors \(e_\ell-\sum_{i=1}^{g+1}a_{i,\ell}e_i\), for \(g+2\le\ell\le2g+1\), form a basis of \(U\).
\end{lemma}

\begin{proof}
Let \(\sigma_i : V \to V\) be the reflection defined by \(\sigma_i (e_j) = (-1)^{\delta_{ij}}e_j\). By \cite[Theorem~3.8]{Rei73},
\[\dim\bigl(U\cap\sigma_{g+2}\cdots\sigma_{2g+1}U\bigr)=0.\]
Therefore \(U\cap\langle e_1,\dots,e_{g+1}\rangle=0\) as a subspace of \(U\cap\sigma_{g+2}\cdots\sigma_{2g+1}U\), and hence \(\overline e_1,\dots,\overline e_{g+1}\) form a basis of \(V/U\).

Then the coefficients \(a_{i,\ell}\) are uniquely determined. For each \(\ell\), the vector \(e_\ell-\sum_i a_{i,\ell}e_i\) maps to zero in \(V/U\), hence lies in \(U\). These vectors project to \(e_{g+2},\ldots,e_{2g+1}\), so they form a basis of \(U\).
\end{proof}

Set \(x_i^\pm:=x_i\pm\sum_{\ell=g+2}^{2g+1}a_{i,\ell}x_\ell\) for \(1\le i\le g+1\). Then \(x_1^+,\dots,x_{g+1}^+\) form a basis of \(U^\perp\), dual to \(\overline e_1,\dots,\overline e_{g+1}\). We use the block notation
\[
\begin{matrix}
  \begin{matrix}
  \bX_f:=\left(\begin{smallmatrix}
  x_1\\
  \vdots\\
  x_{g+1}
  \end{smallmatrix}\right) \\[6pt]
  \bX_b:=\left(\begin{smallmatrix}
  x_{g+2}\\
  \vdots\\
  x_{2g+1}
  \end{smallmatrix}\right)
  \end{matrix} & \begin{matrix}
  \mathbf\Lambda_f:=\left(\begin{smallmatrix}
  \lbd_1\\
  &\ddots\\
  &&\lbd_{g+1}
  \end{smallmatrix}\right) \\[6pt]
  \mathbf\Lambda_b:=\left(\begin{smallmatrix}
  \lbd_{g+2}\\
  &\ddots\\
  &&\lbd_{2g+1}
  \end{smallmatrix}\right)
  \end{matrix} & \bA:= \begin{pmatrix}
  a_{1,g+2}&\cdots&a_{1,2g+1}\\
  \vdots&\ddots&\vdots\\
  a_{g+1,g+2}&\cdots&a_{g+1,2g+1}
  \end{pmatrix}.
\end{matrix}
\]

\begin{lemma}\label{lem:split_quadrics}
We have
\begin{equation}\label{eqn:matrixidentity}
\bA^T\bA=-\bI_g, \qquad \bA^T\mathbf\Lambda_f\bA=-\mathbf\Lambda_b.
\end{equation}
Consequently,
\begin{equation}\label{q_split}
q_0=\sum_{i=1}^{g+1}x_i^+x_i^-, \qquad q_\infty=\sum_{i=1}^{g+1}\lbd_i x_i^+x_i^-.
\end{equation}
\end{lemma}

\begin{proof}
The two matrix identities in \eqref{eqn:matrixidentity} are the conditions that \(U\) is isotropic for \(q_0\) and \(q_\infty\). We compute
\begin{align*}
\sum_{i=1}^{g+1}x_i^+x_i^- &=(\bX_f+\bA\bX_b)^T(\bX_f-\bA\bX_b)\\
&=\bX_f^T\bX_f-\bX_f^T\bA\bX_b +\bX_b^T\bA^T\bX_f-\bX_b^T\bA^T\bA\bX_b\\
&=\bX_f^T\bX_f+\bX_b^T\bX_b=q_0,
\end{align*}
since \(\bX_f^T\bA\bX_b\) is a scalar. Similarly,
\begin{align*}
\sum_{i=1}^{g+1}\lbd_i x_i^+x_i^- &=(\bX_f+\bA\bX_b)^T\mathbf\Lambda_f(\bX_f-\bA\bX_b)\\
&=\bX_f^T\mathbf\Lambda_f\bX_f -\bX_f^T\mathbf\Lambda_f\bA\bX_b +\bX_b^T\bA^T\mathbf\Lambda_f\bX_f -\bX_b^T\bA^T\mathbf\Lambda_f\bA\bX_b\\
&=\bX_f^T\mathbf\Lambda_f\bX_f+\bX_b^T\mathbf\Lambda_b\bX_b =q_\infty.\qedhere
\end{align*}
\end{proof}

For subsets \(I\subset\{1,\dots,g+1\}\) and \(J\subset\{g+2,\dots,2g+1\}\) with \(|I|=|J|\), write
\[a_{I,J}:=\det(a_{i,j})_{i\in I,\ j\in J}.\]

\begin{lemma}\label{lem:A_minors}
Every \(a_{I,J}\) is nonzero.
\end{lemma}

\begin{proof}
Set
\[\sigma:= \Bigl(\prod_{j\in \{g+2,\dots,2g+1\}\setminus J}\sigma_j\Bigr) \Bigl(\prod_{i\in I}\sigma_i\Bigr),\]
where \(\sigma_i\) is the reflection flipping \(e_i\) as before. The number of reflections in \(\sigma\) is \(g-|J|+|I|=g\), so \cite[Theorem~3.8]{Rei73} gives \(U\cap\sigma U=0\).

Suppose \(a_{I,J}=0\). Then there are scalars \(c_j\), not all zero, such that \(\sum_{j\in J}c_j a_{i,j}=0\) for every \(i\in I\). Hence
\[u:=\sum_{j\in J}c_j\left(e_j-\sum_{i=1}^{g+1}a_{i,j}e_i\right)\]
is a nonzero vector of \(U\) fixed by \(\sigma\), which contradicts \(U\cap\sigma U=0\).
\end{proof}

\subsection{An explicit Clifford action and parabolic flags}\label{ssec:explicit_flags}

The decomposition \eqref{q_split} suggests the choice
\begin{equation*}
\ga:=\sum_{i=1}^{g+1}x_i^+\otimes x_i^-\otimes(s+\lbd_i t)\in U^\perp\otimes V^*\otimes G^*.
\end{equation*}

\begin{proposition}\label{prop:explicit_action}
With \(\ga\) as above, the Clifford action on \(F_U\) is given by
\begin{equation}\label{explicit_action_formula}
\alpha\cdot e_\ell= \begin{cases}
  \bigl(\varepsilon_{\overline e_\ell}+(s+\lbd_\ell t)\iota_{x_\ell^+}\bigr)\alpha, & 1\le\ell\le g+1,\\[4pt]
  \left(\varepsilon_{\overline e_\ell}-\sum_{i=1}^{g+1}a_{i,\ell}(s+\lbd_i t)\iota_{x_i^+}\right)\alpha, & g+2\le\ell\le 2g+1.
\end{cases}\qquad(\alpha\in F_U)
\end{equation}
\end{proposition}

\begin{proof}
The formula follows by substituting the identities
\[x_i^-(e_\ell)=\begin{cases} \delta_{i,\ell},&\text{if }1\le\ell\le g+1,\\ -a_{i,\ell},&\text{if }g+2\le\ell\le 2g+1, \end{cases}\]
into \(\alpha\cdot e_\ell=(\varepsilon_{\overline e_\ell}+\iota_{\ga(e_\ell)})\alpha=(\varepsilon_{\overline e_\ell}+\sum_{i=1}^{g+1}x_i^-(e_\ell)(s+\lbd_it)\iota_{x_i^+})\alpha\).
\end{proof}

From now on, we use the standard trivialization over the affine chart \(t\ne0\) of \(\PP^1\), so that the fiber at \(p_\ell\) is obtained by evaluating at \(\lbd=-s/t=\lbd_\ell\). We now apply \S\ref{ssec:rootstack} to the Raynaud module \(F_U\), and describe the resulting parabolic structure explicitly.

\begin{proof}[Proof of Theorem~\ref{thm:raynaud_parabolic}]
By Proposition~\ref{prop:Ray_act}, the even part of \(F_U\) is \(\bigoplus_d\bigwedge^{2d}(V/U)\otimes k[s,t](-2d)\), which sheafifies to the above underlying bundle as \(\deg s=\deg t=2\). Using \eqref{explicit_action_formula}, \(W_\ell\) is the image of the operator obtained by evaluating the action of \(e_\ell\) at \(p_\ell\):
\begin{itemize}
  \item If \(1\le\ell\le g+1\), then \((s+\lbd_\ell t)|_{p_\ell}=0\), hence
  \[W_\ell=\im(\varepsilon_{\overline e_\ell}).\]
  \item If \(g+2\le\ell\le 2g+1\), then \((s+\lbd_i t)|_{p_\ell}=\lbd_i-\lbd_\ell\), hence
  \[W_\ell=\im\!\left(\varepsilon_{\overline e_\ell}+\sum_{i=1}^{g+1}a_{i,\ell}(\lbd_\ell-\lbd_i)\iota_{x_i^+}\right).\]
\end{itemize}
Then \(W_\ell^\perp\) is the kernel of the adjoint operator, and Lemma~\ref{lem:exactness} below identifies this kernel with the image of the same operator.
\end{proof}

\begin{lemma}\label{lem:exactness}
Let \(W\) be a finite-dimensional vector space, and let
\[\Delta=\iota_\beta+\varepsilon_u:\bigwedge W\longrightarrow \bigwedge W\]
for some \(u\in W\) and \(\beta\in W^*\) satisfying \(\beta(u)=0\). If \((u,\beta)\ne(0,0)\), then the complex \((\bigwedge W,\Delta)\) is exact.
\end{lemma}

\begin{proof}
The identity \(\Delta^2=\beta(u)=0\) shows that this is a complex. If \(\beta\ne0\), choose \(v\in W\) with \(\beta(v)=1\). Then
\[\Delta\varepsilon_v+\varepsilon_v\Delta=1.\]
If \(\beta=0\), then \(u\ne0\). Choose \(\alpha\in W^*\) with \(\alpha(u)=1\). Then
\[\Delta\iota_\alpha+\iota_\alpha\Delta=\varepsilon_u\iota_\alpha+\iota_\alpha\varepsilon_u=1.\]
Hence the identity map is null-homotopic in both cases.
\end{proof}

\begin{corollary}\label{cor:par_rk_deg}
One has
\[\rank\cR=2^g,\qquad\dim W_\ell=2^{g-1}\quad \text{for all }1\le\ell\le2g+1,\qquad\pardeg\cR=g2^{g-2}.\]
\end{corollary}

\begin{proof}
By definition,
\[\rank\cR=\sum_{d\ge0}\binom{g+1}{2d}=2^g.\]
Moreover, \(W_\ell\) is the image of the odd-to-even component of an exact odd endomorphism of \(\bigwedge(V/U)\). Since its two parity components are conjugate, they have the same rank; hence
\[\dim W_\ell=\frac12\dim\bigwedge^{\mathrm{ev}}(V/U)=2^{g-1}.\]
The underlying bundle has degree
\[\deg\cR=-\sum_{d\ge0} d\binom{g+1}{2d}=-(g+1)2^{g-2},\]
and each marked point contributes \(\frac12\dim W_\ell=2^{g-2}\). Since there are \(2g+1\) marked points,
\[\pardeg\cR=-(g+1)2^{g-2}+(2g+1)2^{g-2}=g2^{g-2}. \qedhere\]
\end{proof}

%%%%%%%%%%%%%%%%%%%%%%%%%%%%%%%%%%%%%%%%%%%%%%%%%%%%
\section{Orthogonality criterion}\label{sec:orthogonality}

We seek parabolic bundles \(\cG\) orthogonal to \(\cR\) in Theorem~\ref{thm:raynaud_parabolic}. Since \(\PP_{2g+1}^1\) is one-dimensional, it suffices to show that \(\pHom(\cG, \cR) = 0\) and \(\chi_{\rpar}(\cG, \cR) = 0\). The explicit description of parabolic bundles on \(\PP^1\) converts these conditions into a concrete linear algebra problem.

We begin with the numerical restriction imposed by the vanishing of \(\chi_{\rpar}(\cG,\cR)\). The following result is a restatement of Proposition~\ref{prop:Ulrichrankodd} in terms of parabolic bundles. Let \(\bp = (p_1, p_2, \dots, p_{2g+1})\).

\begin{proposition}\label{prop:chi_condition}
Let \( \cG=(\underline{\cG},\{L_\ell\}_{\ell=1}^{2g+1}) \) be a rank-\(r\) parabolic bundle on \((\PP^1, \bp)\) such that \(\chi_{\rpar}(\cG, \cR) = 0\). Then
\[\pardeg \cG = r(3-g)/4.\]
\end{proposition}

\subsection{Balanced parabolic bundles}\label{ssec:balancedparabolic}

The case of odd \(g\) is the essential one. Assume that \(g\) is odd, and write the underlying bundle of \(\cG\) as \(\underline{\cG}\cong \bigoplus_{i=1}^r \cO_{\PP^1}(n_i)\). If \(\cG\) is orthogonal to \(\cR\), then \(\pExt^1(\cG, \cR) = 0\). Equation~\eqref{eqn:par_long_seq} gives \(\Ext^1(\underline{\cG},\underline{\cR})=0\). Since \(\cO_{\PP^1}\!\left(-(g+1)/2\right)\) is a direct summand of \(\underline{\cR}\) of minimal degree, it follows that
\[n_i\le (1-g)/2, \qquad\text{for all }i.\]

Motivated by this necessary condition, we consider the extremal balanced bundle:
\begin{equation}\label{eqn:extbalbdl}
\cG:=\left(\cO_{\PP^1}\!\left((1-g)/2\right)^{\oplus r},\{L_\ell\}_{\ell=1}^{2g+1}\right),\qquad r=1,2,3.
\end{equation}
Choose nonzero vectors \(z_\ell=(z_{\ell,1}\;\dots\;z_{\ell,r})\in k^r\), and set
\[
L_\ell= \begin{cases}
  k\cdot z_\ell\subset k^r\cong \underline{\cG}|_{p_\ell}, & 1\le\ell\le r(g+1)/2,\\[4pt]
  0, & r(g+1)/2<\ell\le 2g+1,
\end{cases}
\]
so that \(\pardeg\cG=r(3-g)/4\). Then \(\cG\) is orthogonal to \(\cR\) if and only if \(\pHom(\cG,\cR)=0\).

Let \(\lbd=-s/t\) on the chart \(t\ne0\). For \(I:=\{i_1<\cdots<i_s\}\subset\{1,\dots,g+1\}\), write
\[\overline e_I:=\overline e_{i_1}\wedge\cdots\wedge\overline e_{i_s},\qquad x_I^+:=x_{i_1}^+\wedge\cdots\wedge x_{i_s}^+.\]
A morphism
\[\psi:\cO_{\PP^1}\!\left((1-g)/2\right)^{\oplus r}\longrightarrow\bigoplus_{d\ge0}\bigwedge^{2d}(V/U)\otimes \cO_{\PP^1}(-d)\]
is determined by coefficients \(c_{\varrho,I,d}\in k\), where \(1\le\varrho\le r\), \(I\subset\{1,\cdots,g+1\}\) with \(2\mid |I|\), and \(0\le d\le (g-1-|I|)/2\), through
\[
\psi=\left(\sum_{\substack{I\subset\{1,\cdots,g+1\}\\|I|\ \mathrm{even}}}\sum_{d=0}^{(g-1-|I|)/2}c_{1,I,d}\lbd^d\overline e_I,\ \cdots,\ \sum_{\substack{I\subset\{1,\cdots,g+1\}\\|I|\ \mathrm{even}}}\sum_{d=0}^{(g-1-|I|)/2}c_{r,I,d}\lbd^d\overline e_I\right).
\]

For \(1\le\ell\le r(g+1)/2\), the image of the generator \(z_\ell\in L_\ell\) is
\[
\psi(z_\ell)=\sum_{\substack{I\subset\{1,\cdots,g+1\}\\|I|\ \mathrm{even}}}\sum_{d=0}^{(g-1-|I|)/2}\bigl(c_{1,I,d}z_{\ell,1}+\cdots+c_{r,I,d}z_{\ell,r}\bigr)\lbd_\ell^d\overline e_I.
\]
Hence \(\psi\) is parabolic precisely when \(\psi(z_\ell)\in W_\ell\), or equivalently, \(\langle\psi(z_\ell), W_\ell^\perp\rangle=0.\)

\subsection{A matrix description}\label{ssec:matrix}

For each \(1\le\ell\le r(g+1)/2\), choose a basis \(B_\ell\) of \(W_\ell^\perp\). The parabolicity of \(\psi\) is equivalent to \(\langle\psi(z_\ell),\beta\rangle=0\) for all \(\beta\in B_\ell\).

We define a matrix \(\bW_r\) as follows. Its rows are indexed by
\[R_r:=\left\{(\ell,\beta): 1\le\ell\le r(g+1)/2,\ \beta\in B_\ell\right\},\]
and its columns by
\[S_r:=\left\{(\varrho,I,d):1\le\varrho\le r,\ I\subset\{1,\cdots,g+1\},\ |I|\ \text{even},\ 0\le d\le (g-1-|I|)/2\right\}.\]
The \(((\ell,\beta),(\varrho,I,d))\)-entry of \(\bW_r\) is defined to be the coefficient of \(c_{\varrho,I,d}\) in the equation \(\langle \beta,\psi(z_\ell)\rangle\). Hence if we set \(\bc := (c_{\varrho,I,d})_{(\varrho,I,d)\in S_r}\),
\begin{equation}\label{lin_cond}
\bW_r \cdot \bc=\mathbf{0}
\end{equation}
is exactly the condition that \(\psi\) be parabolic.

The matrix \(\bW_r\) is square: its number of rows is
\[|R_r|=\sum_{i=1}^{r(g+1)/2}\dim W_i^\perp=r(g+1)2^{g-2},\]
while its number of columns is
\[|S_r|=r\sum_{\substack{I\subset\{1,\cdots,g+1\}\\|I|\ \mathrm{even}}}\left(\frac{g+1-|I|}{2}\right)=r\sum_{m\ge1}m\binom{g+1}{g+1-2m}=r(g+1)2^{g-2}.\]

Order the columns by increasing \(|I|\), and choose each \(B_\ell\) to consist of images of odd degree vectors under the operators defining \(W_\ell^\perp\) in Theorem~\ref{thm:raynaud_parabolic}. These operators shift exterior degree by one, so each row meets only its diagonal block indexed by \(|I|\) and possibly one more block with larger \(|I|\). Thus, \(\bW_{r}\) is block upper bidiagonal. We denote by \(\bW_{r,m}\) the diagonal block corresponding to \(|I| = g + 1 - 2m\).

\begin{example}\label{ex:matrix_rank1_genus3}
Let \(r=1\) and \(g=3\). Write \(c_{ij}:=c_{\{i,j\},0}\). Then
\[\psi=(c_{\emptyset,0}+c_{\emptyset,1}\lbd)+\sum_{1\le i<j\le4}c_{ij}\overline e_{ij}.\]
Take \(z_1=z_2=1\), and choose bases
\[B_1=\{1,x_{23}^+,x_{24}^+,x_{34}^+\}, \qquad B_2=\{1,x_{13}^+,x_{14}^+,x_{34}^+\}.\]
Then the matrix
\[
\setlength{\arraycolsep}{3pt} \bW_1= \left( \begin{array}{@{}cc|cccccc|l@{}}
  c_{\emptyset,0}&c_{\emptyset,1}&c_{12}&c_{13}&c_{14}&c_{23}&c_{24}&c_{34}\\
  \hline 1&\lbd_1&0&0&0&0&0&0&(1,1)\\
  1&\lbd_2&0&0&0&0&0&0&(2,1)\\
  \hline 0&0&0&0&0&1&0&0&(1,x_{23}^+)\\
  0&0&0&0&0&0&1&0&(1,x_{24}^+)\\
  0&0&0&0&0&0&0&1&(1,x_{34}^+)\\
  0&0&0&1&0&0&0&0&(2,x_{13}^+)\\
  0&0&0&0&1&0&0&0&(2,x_{14}^+)\\
  0&0&0&0&0&0&0&1&(2,x_{34}^+)
\end{array}\right).
\]
Here \(\bW_1=\bW_{1,2}\oplus\bW_{1,1}\), with \(\det\bW_{1,2}=\lbd_2-\lbd_1\ne0\) and \(\bW_{1,1}\) is singular.
\end{example}

\begin{example}\label{ex:matrix_rank1_genus3_mixed}
Let \(r=1\) and \(g=3\), and take \(L_1=L_5=k\) and \(L_\ell=0\) otherwise. Write
\[\overline e_5=\sum_{i=1}^4\alpha_{5,i}\overline e_i.\]
The operator defining \(W_5^\perp\) is
\[\Delta_5:=\iota_{\overline e_5} +\sum_{i=1}^4 \alpha_{5,i}(\lbd_5-\lbd_i)\varepsilon_{x_i^+}.\]
Choose bases
\[B_1=\{1,x_{23}^+,x_{24}^+,x_{34}^+\},\qquad B_5=\{\Delta_5(x_1^+),\Delta_5(x_{123}^+),\Delta_5(x_{124}^+),\Delta_5(x_{134}^+)\}.\]
Then the matrix
\[
\setlength{\arraycolsep}{2.5pt} \bW_{1,\{1,5\}}=\left(\begin{array}{@{}cc|cccccc|l@{}}
  c_{\emptyset,0}&c_{\emptyset,1}&c_{12}&c_{13}&c_{14}&c_{23}&c_{24}&c_{34}\\
  \hline 1&\lbd_1&0&0&0&0&0&0&(1,1)\\
  \alpha_{5,1}&\alpha_{5,1}\lbd_5&-\alpha_{5,2}(\lbd_5-\lbd_2)&-\alpha_{5,3}(\lbd_5-\lbd_3)&-\alpha_{5,4}(\lbd_5-\lbd_4)&0&0&0&(5,\Delta_5(x_1^+))\\
  \hline 0&0&0&0&0&1&0&0&(1,x_{23}^+)\\
  0&0&0&0&0&0&1&0&(1,x_{24}^+)\\
  0&0&0&0&0&0&0&1&(1,x_{34}^+)\\
  0&0&\alpha_{5,3}&-\alpha_{5,2}&0&\alpha_{5,1}&0&0&(5,\Delta_5(x_{123}^+))\\
  0&0&\alpha_{5,4}&0&-\alpha_{5,2}&0&\alpha_{5,1}&0&(5,\Delta_5(x_{124}^+))\\
  0&0&0&\alpha_{5,4}&-\alpha_{5,3}&0&0&\alpha_{5,1}&(5,\Delta_5(x_{134}^+))
\end{array}\right).
\]
is block upper bidiagonal. The coefficients \(\alpha_{5,i}\) will be generalized to \(\alpha_{\Lambda,I}\) defined in Proposition~\ref{prop:modified_entry}.
\end{example}

\begin{proposition}\label{prop:matrixcriterion}
With the notation above, each \(\bW_{r,m}\) is a square matrix, and
\[\pHom(\cG,\cR)=0\quad\Longleftrightarrow\quad\det \bW_{r,m}\ne0\text{ for }1\le m\le (g+1)/2.\]
\end{proposition}

\begin{proof}
A parabolic morphism is uniquely determined by the coefficients \(c_{\varrho,I,d}\), and the parabolicity condition is exactly \eqref{lin_cond}. Hence \(\pHom(\cG,\cR)=0\) holds precisely when \eqref{lin_cond} has only the trivial solution. Since \(\bW_r\) is block upper bidiagonal, this is equivalent to the invertibility of every diagonal block \(\bW_{r,m}\) provided they are all square.

For a fixed \(m\), the number of columns of \(\bW_{r,m}\) is
\[r\,m\binom{g+1}{2m},\]
since for each \(I\subset\{1,\cdots,g+1\}\) with \(|I|=g+1-2m\), the index \(d\) ranges from \(0\) to \(m-1\). On the other hand, the number of rows is
\[\frac{r(g+1)}{2}\binom{g}{2m-1}=r\,m\binom{g+1}{2m},\]
because the image of any nonzero contraction map \(\iota_v\) in degree \(g+1-2m\) has dimension \(\binom{g}{g+1-2m}=\binom{g}{2m-1}\). Hence \(\bW_{r,m}\) is square for every \(m\).
\end{proof}

The rest of this paper proves that each block \(\bW_{r,m}\) is invertible for \(r=2,3\) and a generic choice of parabolic flags \(L_{\ell}\), \(1 \le \ell \le r(g+1)/2\). This requires a nontrivial linear-algebra computation.

The same argument shows that it suffices to check the orthogonality with a modified parabolic bundle \(\cR^{aux} = (\underline{\cR}, \{W_\ell'\}_{\ell = 1}^{2g+1})\), whose associated matrix \(\bW'_r\) is block diagonal and whose diagonal blocks are \(\bW_{r,m}\).

\begin{corollary}\label{cor:diagonal_reduction}
For the parabolic bundle \(\cG\) considered above,
\[\pHom(\cG,\cR)=0\iff\pHom(\cG,\cR^{aux})=0,\]
where \(\cR^{aux}=(\underline{\cR},\{W_\ell'\}_{\ell=1}^{2g+1})\) is the auxiliary parabolic bundle with \(W_\ell'=\im(\varepsilon_{\overline e_\ell})\) so that \(W_\ell'^\perp=\im(\iota_{\overline e_\ell})\).
\end{corollary}

\subsection{An explicit description of the modified diagonal blocks}\label{ssec:diagonalblock}

By Proposition~\ref{prop:matrixcriterion}, it suffices to understand the diagonal block \(\bW_{r,m}\). The Hodge star operator gives a change of coordinates that makes the computation more tractable.

Fix dual volume forms \(\omega_q:=\overline e_1\wedge\cdots\wedge\overline e_{g+1}\in \bigwedge^{g+1}(V/U)\) and \(\omega_p:=x_1^+\wedge\cdots\wedge x_{g+1}^+\in \bigwedge^{g+1}U^\perp\), normalized by \(\langle \omega_q,\omega_p\rangle=1\). Define
\[\star_q:\bigwedge^d(V/U)\longrightarrow \bigwedge^{g+1-d}U^\perp,\qquad\star_p:\bigwedge^d U^\perp\longrightarrow \bigwedge^{g+1-d}(V/U)\]
by \(\star_q(\alpha):=\iota_\alpha(\omega_p)\) and \(\star_p(\beta):=\iota_\beta(\omega_q)\). For a subset \(I\subset\{1,\cdots,g+1\}\), let \(\epsilon(I,J)\in\{\pm1\}\) be the sign determined by \(\overline e_I\wedge\overline e_J=\epsilon(I,J)\overline e_{I\sqcup J}\). Then
\begin{equation}\label{eqn:Hodgestar}
\star_q(\overline e_I)=\epsilon(I,I^c)x_{I^c}^+, \quad\star_p(x_I^+)=\epsilon(I,I^c)\overline e_{I^c},\quad \text{and} \quad\star_p\circ\iota_{\overline e_i}=(-1)^{d-1}\varepsilon_{\overline e_i}\circ\star_p.
\end{equation}
Moreover, the pairing \eqref{eq:perf_pair} satisfies \(\langle \alpha,\beta\rangle=\langle \star_p\beta,\star_q\alpha\rangle\).

Since \(W_\ell'=\im(\varepsilon_{\overline e_\ell})\) and \(W_\ell'^\perp=\im(\iota_{\overline e_\ell})\), the identity above yields \(\star_p(W_\ell'^\perp)=W_\ell'\). Hence
\[
\left\langle \psi(z_\ell)^{(m)},\,W_\ell'^\perp\cap \bigwedge^{g+1-2m}U^\perp\right\rangle=0 \qquad\Longleftrightarrow\qquad \left\langle W_\ell'\cap \bigwedge^{2m}(V/U),\,\star_q(\psi(z_\ell)^{(m)})\right\rangle=0,
\]
where
\[\psi(z_\ell)^{(m)}:=\sum_{\substack{I\subset\{1,\cdots,g+1\}\\|I|=g+1-2m}}\sum_{d=0}^{m-1}(c_{1,I,d}z_{\ell,1}+\cdots+c_{r,I,d}z_{\ell,r})\lbd_\ell^d\overline e_I.\]

Via the expansion of \(\star_q(\psi(z_\ell)^{(m)})\), we obtain a new matrix \(\star\bW_{r,m}\) from \(\bW_{r,m}\) by replacing each column index \(I\) with \(I^c\) and each coefficient \(c_{\varrho,I,d}\) with \(\epsilon(I,I^c)c_{\varrho,I,d}\). Since this only permutes the columns and changes their signs, the new matrix has the same rank as \(\bW_{r,m}\).

We next describe the new matrix \(\star\bW_{r,m}\). Set
\[
\Om_{\mathrm{head}}:=\left\{1,\dots, (g+1)/2\right\},\;\Om_{\mathrm{mid}}:=\left\{(g+3)/2, \dots,g+1\right\},\;\Om_{\mathrm{tail}}:=\left\{g+2,\dots,3(g+1)/2\right\},
\]
and
\[
\Om_\ell:= \begin{cases}
  \Om_{\mathrm{good}}:=\Om_{\mathrm{head}}\sqcup\Om_{\mathrm{mid}}=\{1,\dots,g+1\}, & 1\le\ell\le g+1,\\[4pt]
  \Om_{\mathrm{bad}}:=\Om_{\mathrm{head}}\sqcup\Om_{\mathrm{tail}}=\left\{1,\dots,(g+1)/2,\,g+2,\dots,3(g+1)/2\right\}, & g+2\le\ell\le 2g+1.
\end{cases}
\]
The rows of \(\star\bW_{r,m}\) are indexed by
\[R_{r,m}^\star:=\{(\ell,\Lambda):1\le\ell\le r(g+1)/2,\ \ell\in\Lambda\subset\Om_\ell,\ |\Lambda|=2m\},\]
and its columns by
\[S_{r,m}^\star:=\{(\varrho,I,d):1\le\varrho\le r,\ I\subset\Om_{\mathrm{good}},\ |I|=2m,\ 0\le d\le m-1\}.\]
For \(g+2\le\ell\le2g+1\) and \((\ell,\Lambda)\in R_{r,m}^\star\), decompose \(\Lambda=\Lambda_0\sqcup\Lambda_1\), where
\[\Lambda_0:=\Lambda\cap\Om_{\mathrm{head}},\quad\Lambda_1:=\Lambda\cap\Om_{\mathrm{tail}}.\]

\begin{proposition}\label{prop:modified_entry}
Let \((\ell,\Lambda)\in R_{r,m}^\star\) and \((\varrho,I,d)\in S_{r,m}^\star\).

\begin{enumerate}
  \item If \(1\le\ell\le g+1\), then
  \[(\star\bW_{r,m})_{(\ell,\Lambda),(\varrho,I,d)}= \begin{cases} z_{\ell,\varrho}\lbd_\ell^d, &\text{if }\Lambda=I,\\ 0, & \text{otherwise}. \end{cases}\]
  \item If \(g+2\le \ell\le 2g+1\), define
  \[
  \alpha_{\Lambda,I}:= \begin{cases}
    \epsilon(\Lambda_0,I\setminus\Lambda_0)\,a_{I\setminus\Lambda_0,\Lambda_1}, & \text{if }\Lambda_0\subset I,\\[4pt]
    0, & \text{otherwise}.
  \end{cases}
  \]
  Then
  \[(\star\bW_{r,m})_{(\ell,\Lambda),(\varrho,I,d)}=\alpha_{\Lambda,I}\,z_{\ell,\varrho}\lbd_\ell^d.\]
\end{enumerate}
In either case, the coefficient is obtained by writing \(\overline e_\Lambda\) in terms of the basis indexed by \(\Om_{\mathrm{good}}\). In particular,
\[\overline e_\Lambda=\sum_{\substack{I\subset\Om_{\mathrm{good}}\\|I|=2m}}\alpha_{\Lambda,I}\,\overline e_I\qquad\text{if }\Lambda\subset\Om_{\mathrm{bad}}.\]
\end{proposition}

\begin{proof}
First, \(B_{\mathrm{bad}} := \{\overline e_i\}_{i \in \Om_{\mathrm{bad}}}\) is a basis of \(V/U\). Set \(B_{\mathrm{good}} := \{\overline e_1, \overline e_2, \dots, \overline e_{g+1}\}\). Then the change-of-coordinates matrix is
\[P = \begin{pmatrix} I & *\\ 0 & A_{\Om_{\mathrm{mid}}, \Om_{\mathrm{tail}}} \end{pmatrix},\]
and \(\det A_{\Om_{\mathrm{mid}}, \Om_{\mathrm{tail}}} = a_{\Om_{\mathrm{mid}}, \Om_{\mathrm{tail}}} \ne 0\) by Lemma~\ref{lem:A_minors}. Thus, \(\{\overline e_\Lambda\}_{\Lambda \subset \Omega_{\mathrm{bad}}}\) is a basis of the corresponding exterior powers.

With \(c_{\varrho,I,d}\) replaced as above, the statement follows from
\[
\langle\overline e_\Lambda,\star_q(\psi(z_\ell)^{(m)})\rangle=\sum_{\substack{I\subset\Om_{\mathrm{good}}\\|I|=2m}}\sum_{d=0}^{m-1}(c_{1,I,d}z_{\ell,1}+\cdots+c_{r,I,d}z_{\ell,r})\lbd_\ell^d\langle\overline e_\Lambda, x_I^+\rangle
\]
and \(\langle\overline e_I,x_J^+\rangle=\delta_{I,J}\) whenever \(I,J\subset\Om_{\mathrm{good}}\).
\end{proof}

For a fixed \(r\), call an index \(\ell\in\{1,\dots,r(g+1)/2\}\) \textbf{good} if \(\ell\le g+1\), so that \(\Om_\ell = \Om_{\mathrm{good}}\), and \textbf{bad} if \(\ell\ge g+2\), so that \(\Om_\ell = \Om_{\mathrm{bad}}\). Proposition~\ref{prop:modified_entry} shows that the rows of \(\star \bW_{r, m}\) corresponding to good indices are particularly simple.

%%%%%%%%%%%%%%%%%%%%%%%%%%%%%%%%%%%%%%%%%%%%%%%%%%%%
\section{Existence of orthogonal bundles}\label{sec:existence}

In this section, we study the orthogonal parabolic bundles of rank \(1,2,3\), using the block criterion of Proposition~\ref{prop:matrixcriterion} together with the explicit formula in Proposition~\ref{prop:modified_entry}.

In Proposition~\ref{prop:modified_entry} (2), the entry pattern of the block is controlled by how many indices of a subset \(I\) lie in \(\{1,\cdots,(g+1)/2\}\). This motivates the following notation:

\begin{definition}\label{def:type}
Define the \textbf{type} of \(I \subset \{1, 2, \dots, g+1\}\) by
\[\type(I):=\left|I\cap\Om_{\mathrm{head}}\right|.\]
\end{definition}

\subsection{Rank 1}\label{ssec:rank1}

We start with parabolic bundles of the form
\begin{equation}\label{eqn:rank1}
\cG=\left(\cO_{\PP^1}\!\left((1-g)/2\right),\{L_\ell\}_{\ell=1}^{2g+1}\right),\qquad L_\ell= \begin{cases} k, & 1\le \ell\le (g+1)/2,\\[4pt] 0, & (g+1)/2<\ell\le 2g+1, \end{cases}
\end{equation}
which is a special case of \eqref{eqn:extbalbdl}. We may assume \(z_\ell=1\) for all \(\ell\), since \(k\cdot z_\ell=k\) for every nonzero scalar \(z_\ell\).

\begin{example}\label{ex:rk1}
When \(g=5\), the matrix \(\star\bW_{1,2}\) has the following shape:
\[
\left( \begin{array}{cccccc|c}
  1234 & 1235 & \dots & 2356 & 2456 & 3456 & \\
  \hline \begin{matrix} 1 & \lbd_1 \end{matrix} & & & & & & 1\in1234\\
  \begin{matrix} 1 & \lbd_2 \end{matrix} & & & & & & 2\in1234\\
  \begin{matrix} 1 & \lbd_3 \end{matrix} & & & & & & 3\in1234\\
  & \begin{matrix} 1 & \lbd_1 \end{matrix} & & & & & 1\in1235\\
  & \begin{matrix} 1 & \lbd_2 \end{matrix} & & & & & 2\in1235\\
  & \begin{matrix} 1 & \lbd_3 \end{matrix} & & & & & 3\in1235\\
  & & \ddots & & & & \vdots\\
  & & & \begin{matrix} 1 & \lbd_2 \end{matrix} & & & 2\in2356\\
  & & & \begin{matrix} 1 & \lbd_3 \end{matrix} & & & 3\in2356\\
  & & & & \begin{matrix} 1 & \lbd_2 \end{matrix} & & 2\in2456\\
  & & & & & \begin{matrix} 1 & \lbd_3 \end{matrix} & 3\in3456
\end{array} \right).
\]
For brevity, we write a subset \(\{i_1<\dots<i_s\}\) as \(i_1\dots i_s\), suppress the \(\varrho\)-index, group the two columns corresponding to \(d=0,1\) for each fixed \(I\), and denote the row indexed by \((\ell,\Lambda)\) simply by \(\ell\in\Lambda\).
\end{example}

The matrix in the preceding example is visibly singular. Furthermore, it motivates the following block decomposition of the modified diagonal block \(\star\bW_{1,m}\).

\begin{lemma}\label{lem:rank1_block_decomposition}
For each \(m\), the matrix \(\star\bW_{1,m}\) becomes block diagonal:
\[\star\bW_{1,m}=\bigoplus_{\substack{I\subset\Om_{\mathrm{good}}\\|I|=2m}}\bV_{1,I},\]
where \(\bV_{1,I}\) is a \(\type(I)\times m\) matrix, whose rows are indexed by \(\ell\in I\cap\{1,\dots,(g+1)/2\}\), whose columns are indexed by \(d=0,\dots,m-1\), and whose entries are given by
\[(\bV_{1,I})_{\ell,d}=\lbd_\ell^d.\]
\end{lemma}

\begin{proof}
Since \(1\le \ell\le (g+1)/2\), every row of \(\star\bW_{1,m}\) falls into item (1) of Proposition~\ref{prop:modified_entry}. Hence
\[(\star\bW_{1,m})_{(\ell,\Lambda),(1,I,d)}= \begin{cases} \lbd_\ell^d, & \Lambda=I,\\ 0, & \text{otherwise}. \end{cases}\]
In particular, a row indexed by \((\ell,\Lambda)\) interacts only with the columns indexed by the same subset \(\Lambda\). This gives the block diagonal decomposition and the stated formula for the entries of each block.
\end{proof}

\begin{proposition}\label{prop:rank1_nullity}
For \(1\le m\le (g+1)/2\),
\[\nullity(\star\bW_{1,m})=\sum_{p=0}^{m}(m-p)\binom{(g+1)/2}{p}\binom{(g+1)/2}{2m-p}=\frac{m(g+1-2m)}{2(g+1)}\binom{(g+1)/2}{m}^{2}.\]
In particular, the rank \(1\) bundle in \eqref{eqn:rank1} is never orthogonal to \(\cR\) for odd \(g\ge 3\).
\end{proposition}

\begin{proof}
By Lemma~\ref{lem:rank1_block_decomposition},
\[\nullity(\star\bW_{1,m})=\sum_{\substack{I\subset\{1,\cdots,g+1\}\\|I|=2m}}\nullity(\bV_{1,I}).\]
Let \(p=\type(I)\). Then \(\bV_{1,I}\) is a \(p\times m\) Vandermonde matrix
\[
\bV_{1,I}= \begin{pmatrix}
  1 & \lbd_{\ell_1} & \dots & \lbd_{\ell_1}^{m-1}\\
  \vdots & \vdots & & \vdots\\
  1 & \lbd_{\ell_p} & \dots & \lbd_{\ell_p}^{m-1}
\end{pmatrix},
\]
where \(\ell_1,\dots,\ell_p\) are the elements of \(I\cap\{1,\dots,(g+1)/2\}\). Since the scalars \(\lbd_{\ell_i}\) are pairwise distinct, \(\bV_{1,I}\) has rank \(\min(p,m)\), and hence nullity \(\max(m-p,0)\). Therefore
\[\nullity(\star\bW_{1,m})=\sum_{\substack{I\subset\{1,\cdots,g+1\}\\|I|=2m}}\max(m-\type(I),0).\]

Now a subset of type \(p\) is obtained by choosing \(p\) elements from \(\{1,\cdots,(g+1)/2\}\) and \(2m-p\) elements from \(\{(g+3)/2,\cdots,g+1\}\). Hence the number of such subsets is
\[\binom{(g+1)/2}{p}\binom{(g+1)/2}{2m-p}.\]
Since only the range \(p\le m\) contributes to the nullity, this proves the first equality. The second equality comes from the telescoping identity for \(S(p):=\frac{(2m-p)(g+1-2p)}{2(g+1)}\genfrac(){0pt}{}{(g+1)/2}{p}\genfrac(){0pt}{}{(g+1)/2}{2m-p}\):
\[
\begin{split}
  \sum_{p=0}^{m}(m-p)\binom{(g+1)/2}{p}\binom{(g+1)/2}{2m-p}&=\sum_{p=0}^{m}(S(p)-S(p-1))=S(m)\\
  &=\frac{m(g+1-2m)}{2(g+1)}\binom{(g+1)/2}{m}^{2}.\qedhere
\end{split}
\]
\end{proof}

We now extend the previous proposition to arbitrary parabolic line bundles. The next result shows that \(r = 1\) is impossible in Theorem~\ref{thm:mainthmodd}.

\begin{theorem}\label{thm:line_bundle_case}
For \(g\ge 2\), there is no parabolic line bundle orthogonal to \(\cR\). Consequently, \(X^{2g-2}\) admits no Ulrich bundle of rank \(2^{g-3}\).
\end{theorem}

\begin{proof}
Suppose that the parabolic line bundle \(\cL=(\underline{\cL},\{U_\ell\}_{\ell=1}^{2g+1})\) is orthogonal to \(\cR\). By the discussion in \S\ref{ssec:balancedparabolic}, \(g\) must be odd and we may write
\[
\underline{\cL}=\cO_{\PP^1}\left(\frac{1-g}{2}-n\right),\quad U_\ell= \begin{cases} \underline{\cL}|_{p_\ell}, & 1\le\ell\le (g+1)/2+2n,\\[4pt] 0, & (g+1)/2+2n<\ell\le 2g+1, \end{cases}
\]
for \(0\le n<3(g+1)/4\), after relabeling the marked points if necessary.

Choose nonzero \(\eta\in\bigwedge^{(g-3)/2}(V/U)\), and set
\[\xi:=\eta\cdot\prod_{\ell=1}^{(g+1)/2+2n}e_\ell\in(F_U)_{g-1+2n}.\]
The element \(\xi\) is nonzero since each \(e_\ell\) acts injectively on \(F_U\), and it defines a morphism \(\underline{\cL}\to\underline{\cR}\). For \(1\le i\le(g+1)/2+2n\), the Clifford relations give
\[\xi=\pm\left(\eta\cdot\prod_{\substack{1\le\ell\le(g+1)/2+2n\\\ell\ne i}}e_\ell\right)e_i,\]
so \(\xi|_{p_i}\in W_i=\im(\cdot e_i)\). Hence \(\xi\) defines a nonzero parabolic morphism \(\cL\to\cR\), contradicting \(\pHom(\cL,\cR)=0\).
\end{proof}

\subsection{Rank 2}\label{ssec:rank2}

The rank-two case is analogous to the rank-one case.

Consider the rank-two extremal balanced bundle in \eqref{eqn:extbalbdl}. The following lemma can be shown in the same way as Lemma~\ref{lem:rank1_block_decomposition}. Note that all nontrivial parabolic flags are indexed by item (1) of Proposition~\ref{prop:modified_entry}.

\begin{lemma}\label{lem:rank2_block_decomposition}
For each \(m\), the matrix \(\star\bW_{2,m}\) becomes block diagonal:
\[\star\bW_{2,m}=\bigoplus_{\substack{I\subset\Om_{\mathrm{good}}\\|I|=2m}}\bV_{2,I},\]
where \(\bV_{2,I}\) is a \(2m\times 2m\) matrix whose rows are indexed by \(\ell\in I\), whose columns are indexed by \((\varrho,d)\) with \(\varrho\in\{1,2\}\) and \(0\le d\le m-1\), and whose entries are given by
\[(\bV_{2,I})_{\ell,(\varrho,d)}=z_{\ell,\varrho}\lbd_\ell^d.\]
\end{lemma}

\begin{proposition}\label{prop:rank2_generic_full_rank}
For a general choice of nonzero vectors \(z_1,\cdots,z_{g+1}\in k^2\), one has
\[\det \star\bW_{2,m}\neq 0\quad\text{for all }1\le m\le (g+1)/2.\]
Hence, a general rank-two extremal balanced bundle in \eqref{eqn:extbalbdl} is orthogonal to \(\cR\) for odd \(g\).
\end{proposition}

\begin{proof}
It suffices to show that \(\bV_{2,I}\) is generically invertible. Set
\[z_{\ell,1}=1,\qquad z_{\ell,2}=\lbd_\ell^m\qquad\text{for all }\ell\in I.\]
Then \(\bV_{2,I}\) becomes, up to a permutation of columns, the ordinary Vandermonde matrix
\[\bigl(\lbd_\ell^{d'}\bigr)_{\ell\in I,\ 0\le d'\le 2m-1}.\]
Since \(\lbd_\ell\) are pairwise distinct, this matrix is invertible.
\end{proof}

\begin{theorem}\label{thm:rk2_bundle_case}
For \(g\ge2\), there exists a rank \(2\) parabolic bundle orthogonal to \(\cR\). Consequently, \(X\) admits an Ulrich bundle of rank \(2^{g-2}\).
\end{theorem}

\begin{proof}
The case of odd \(g\) is covered by Proposition~\ref{prop:rank2_generic_full_rank}. Suppose that \(g\) is even, and consider the almost balanced bundle
\[
\cG=\left(\cO\left((2-g)/2\right)\oplus\cO\left(-g/2\right),\{L_\ell\}_{\ell=1}^{2g+1}\right),\quad L_\ell= \begin{cases} k\cdot z_\ell\subset k^2\cong \underline{\cG}|_{p_\ell}, & 1\le\ell\le g+1,\\[4pt] 0, & g+2\le\ell\le 2g+1. \end{cases}
\]

Let \(\star\bW^{\mathrm{ev}}_{2,m}\) be the matrix defined using the same formalism as in \S\ref{ssec:explicit_flags}, where now \(m\) is determined by the condition \(|I|=2m+1\), since \(g\) is even. Then \(\star\bW^{\mathrm{ev}}_{2,m}\) also decomposes as
\[
\star\bW^{\mathrm{ev}}_{2,m}=\bigoplus_{\substack{I\subset\Om_{\mathrm{good}}\\|I|=2m+1}}\bV^{\mathrm{ev}}_I,\qquad(\bV^{\mathrm{ev}}_I)_{\ell,(\varrho,d)}=z_{\ell,\varrho}\lbd_\ell^d,
\]
where the rows are indexed by \(\ell\in I\), and the columns are indexed by pairs \((\varrho,d)\) with \(\varrho\in\{1,2\}\) and \(0\le d\le m+\varrho-2\). In particular, as in the proof of Proposition~\ref{prop:rank2_generic_full_rank}, \(\det \bV^{\mathrm{ev}}_I\) is nonzero.
\end{proof}

\subsection{Rank 3}\label{ssec:rank3}

Unlike the cases \(r=1\) and \(r=2\), the matrix \(\star\bW_{3,m}\) is no longer block diagonal, since rows with bad indices in item (2) of Proposition~\ref{prop:modified_entry} occur. We begin with the smallest nontrivial example. Additional computational examples are given in \cite{JLM26}.

\begin{example}\label{ex:rk3_first}
When \(g=3\), the \(18 \times 18\) matrix \(\star\bW_{3,1}\) may be written in block form as in Figure~\ref{fig:W31}.
\begin{figure}[!ht]
\[
\left( \begin{array}{cccccc|c}
  12 & 13 & 14 & 23 & 24 & 34 & \\
  \hline z_1 &&&&&& 1\in 12\\
  &z_1&&&&& 1\in 13\\
  &&z_1&&&& 1\in 14\\
  z_2&&&&&& 2\in 12\\
  &&&z_2&&& 2\in 23\\
  &&&&z_2&& 2\in 24\\
  &z_3&&&&& 3\in 13\\
  &&&z_3&&& 3\in 23\\
  &&&&&z_3& 3\in 34\\
  &&z_4&&&& 4\in 14\\
  &&&&z_4&& 4\in 24\\
  &&&&&z_4& 4\in 34\\
  a_{2,5}z_5 & a_{3,5}z_5 & a_{4,5}z_5 &&&& 5\in 15\\
  -a_{1,5}z_5 &&& a_{3,5}z_5 & a_{4,5}z_5 && 5\in 25\\
  a_{12,56}z_5 & a_{13,56}z_5 & a_{14,56}z_5 & a_{23,56}z_5 & a_{24,56}z_5 & a_{34,56}z_5 & 5\in 56\\
  a_{2,6}z_6 & a_{3,6}z_6 & a_{4,6}z_6 &&&& 6\in 16\\
  -a_{1,6}z_6 &&& a_{3,6}z_6 & a_{4,6}z_6 && 6\in 26\\
  a_{12,56}z_6 & a_{13,56}z_6 & a_{14,56}z_6 & a_{23,56}z_6 & a_{24,56}z_6 & a_{34,56}z_6 & 6\in 56
\end{array} \right).
\]
\caption{\(\star\bW_{3,1}\)}\label{fig:W31}
\end{figure}

We use the same convention as in Example~\ref{ex:rk1}, and, by abuse of notation, \( z_\ell=\bigl(z_{\ell,1}\;\;z_{\ell,2}\;\;z_{\ell,3}\bigr) \) denotes the corresponding row block.
\end{example}

To prove \(\det\star\bW_{3,m} \ne 0\), regard \(\{z_{\ell, i}\}\) as formal variables, since the parabolic flags may be chosen generically. Then \(\det \star\bW_{3,m}\) is a polynomial in these variables, so it suffices to choose a monomial and show that its coefficient is nonzero. The obvious choice of the leading term with respect to a lexicographic order does not work because its coefficient is zero. The monomial must therefore be chosen carefully.

The computation of \(\det \star \bW_{3,m}\) has two stages. Algorithm~\ref{alg:detcom} describes the first, and \S\ref{sec:prf_main_thm} the second. Three examples explain why the reduction is needed and illustrate the underlying idea. The first shows that a naive greedy choice may leave a rank \(1\) residual matrix, the second shows how a corrected choice avoids this defect, and the third explains the same mechanism for larger \(m\).

\begin{example}\label{ex:rk3_forcing}
Consider the matrix in Figure~\ref{fig:W31} and the monomial
\[z := z_{1,1}^3 z_{2,1}^2 z_{2,2} z_{3,1} z_{3,2}^2 z_{4,2}^3,\]
which is the leading term with respect to the lexicographic order.

Because of the factor \(z_{1,1}^3\), the three entries \(z_{1,1}\) in the rows \(1\in12\), \(1\in13\), and \(1\in14\) must be chosen. Remove the corresponding rows and columns. Then the factor \(z_{2,1}^2z_{2,2}\) forces the three entries in the rows \(2\in12\), \(2\in23\), and \(2\in24\). Proceeding in the same way, each remaining good row has a unique entry giving the factors \(z_{3,1}z_{3,2}^2\) and \(z_{4,2}^3\).

After deleting the corresponding rows and columns, the remaining minor is
\[
\left( \begin{array}{cccccc|c}
  12 & 13 & 14 & 23 & 24 & 34 & \\
  \hline a_{2,5}z_{5,3} & a_{3,5}z_{5,3} & a_{4,5}z_{5,3} &&&& 5\in 15\\
  -a_{1,5}z_{5,3} &&& a_{3,5}z_{5,3} & a_{4,5}z_{5,3} && 5\in 25\\
  a_{12,56}z_{5,3} & a_{13,56}z_{5,3} & a_{14,56}z_{5,3} & a_{23,56}z_{5,3} & a_{24,56}z_{5,3} & a_{34,56}z_{5,3} & 5\in 56\\
  a_{2,6}z_{6,3} & a_{3,6}z_{6,3} & a_{4,6}z_{6,3} &&&& 6\in 16\\
  -a_{1,6}z_{6,3} &&& a_{3,6}z_{6,3} & a_{4,6}z_{6,3} && 6\in 26\\
  a_{12,56}z_{6,3} & a_{13,56}z_{6,3} & a_{14,56}z_{6,3} & a_{23,56}z_{6,3} & a_{24,56}z_{6,3} & a_{34,56}z_{6,3} & 6\in 56
\end{array} \right).
\]
However, this is not invertible because the third and last rows are linearly dependent. This is exactly the matrix corresponding to the rank \(1\) parabolic bundle
\[\left(\cO_{\PP^1}(-1),\{L_\ell\}_{\ell=1}^7\right),\qquad L_\ell= \begin{cases} kz_{\ell,3}, & \ell=5,6,\\ 0, & \text{otherwise,} \end{cases}\]
whose matrix is singular by \S\ref{ssec:rank1}. Hence the leading monomial does not produce a nonzero coefficient.
\end{example}

\begin{example}\label{ex:rk3_forcing2}
The next largest possible term is obtained by replacing \(z_{2,2}\) in the monomial \(z\) from Example~\ref{ex:rk3_forcing} by \(z_{2,3}\). Consider the monomial
\[z' := z_{1,1}^3 z_{2,1}^2 z_{2,3} z_{3,1} z_{3,2}^2 z_{4,2}^3.\]
After the same elimination of rows and columns, the remaining minor is
\[
\left( \begin{array}{cccccc|c}
  12' & 13 & 14 & 23 & 24 & 34 & \\
  \hline a_{2,5}z_{5,2} & a_{3,5}z_{5,3} & a_{4,5}z_{5,3} &&&& 5\in 15\\
  -a_{1,5}z_{5,2} &&& a_{3,5}z_{5,3} & a_{4,5}z_{5,3} && 5\in 25\\
  a_{12,56}z_{5,2} & a_{13,56}z_{5,3} & a_{14,56}z_{5,3} & a_{23,56}z_{5,3} & a_{24,56}z_{5,3} & a_{34,56}z_{5,3} & 5\in 56\\
  a_{2,6}z_{6,2} & a_{3,6}z_{6,3} & a_{4,6}z_{6,3} &&&& 6\in 16\\
  -a_{1,6}z_{6,2} &&& a_{3,6}z_{6,3} & a_{4,6}z_{6,3} && 6\in 26\\
  a_{12,56}z_{6,2} & a_{13,56}z_{6,3} & a_{14,56}z_{6,3} & a_{23,56}z_{6,3} & a_{24,56}z_{6,3} & a_{34,56}z_{6,3} & 6\in 56
\end{array} \right).
\]
After factoring out \(z_{5,3}\) from the first three rows and \(z_{6,3}\) from the last three rows, and setting \(y_5:=z_{5,2}/z_{5,3}\) and \(y_6:=z_{6,2}/z_{6,3}\), we obtain
\[
\left( \begin{array}{cccccc|c}
  12' & 13 & 14 & 23 & 24 & 34 & \\
  \hline a_{2,5}y_5 & a_{3,5} & a_{4,5} &&&& 5\in 15\\
  -a_{1,5}y_5 &&& a_{3,5} & a_{4,5} && 5\in 25\\
  a_{12,56}y_5 & a_{13,56} & a_{14,56} & a_{23,56} & a_{24,56} & a_{34,56} & 5\in 56\\
  a_{2,6}y_6 & a_{3,6} & a_{4,6} &&&& 6\in 16\\
  -a_{1,6}y_6 &&& a_{3,6} & a_{4,6} && 6\in 26\\
  a_{12,56}y_6 & a_{13,56} & a_{14,56} & a_{23,56} & a_{24,56} & a_{34,56} & 6\in 56
\end{array} \right).
\]
We claim that this matrix is generically invertible. Adjoin the \(12\)-column and extend the matrix to
\[
\left( \begin{array}{ccccccc|c}
  12' & 12& 13 & 14 & 23 & 24 & 34 & \\
  \hline a_{2,5}y_5 &a_{2,5} & a_{3,5} & a_{4,5} &&&& 5\in 15\\
  -a_{1,5}y_5 &-a_{1,5}&&& a_{3,5} & a_{4,5} && 5\in 25\\
  a_{12,56}y_5 &a_{12,56}& a_{13,56} & a_{14,56} & a_{23,56} & a_{24,56} & a_{34,56} & 5\in 56\\
  a_{2,6}y_6 &a_{2,6}& a_{3,6} & a_{4,6} &&&& 6\in 16\\
  -a_{1,6}y_6 &-a_{1,6} &&& a_{3,6} & a_{4,6} && 6\in 26\\
  a_{12,56}y_6 &a_{12,56}& a_{13,56} & a_{14,56} & a_{23,56} & a_{24,56} & a_{34,56} & 6\in 56\\
  &1&&&&&&12
\end{array} \right).
\]
Note that up to sign, the determinant is the same. If we use \(\{1,2,5,6\}\) as the coordinate indices for \(V/U\) instead of \(\{1,2,3,4\}\), we have the following new matrix representation:
\[
\left( \begin{array}{ccccccc|c}
  12' & 12& 15 & 16 & 25 & 26 & 56 & \\
  \hline a_{2,5}y_5 &&1&&&&& 5\in 15\\
  -a_{1,5}y_5 &&&&1&&& 5\in 25\\
  a_{12,56}y_5 &&&&&&1& 5\in 56\\
  a_{2,6}y_6 &&&1&&&& 6\in 16\\
  -a_{1,6}y_6 &&&&&1&& 6\in 26\\
  a_{12,56}y_6 &&&&&&1& 6\in 56\\
  &1&&&&&&12
\end{array} \right),
\]
whose determinant then reduces to
\[\begin{pmatrix} a_{12,56}y_5 & 1\\ a_{12,56}y_6 & 1 \end{pmatrix}.\]
Hence the coefficient is \((y_5-y_6)z_{5,3}^3z_{6,3}^3\) up to scalar, which is generically nonzero by Lemma~\ref{lem:A_minors}.
\end{example}

\begin{example}\label{ex:rk3_vandermonde}
We consider \(m > 1\). Fix \(1\le \varrho\le 3\) and \(I = \{i_1 < i_2 < \dots < i_{2m}\} \subset\Om_{\mathrm{good}}\) with \(|I|=2m\). Let \(C_d\) be the column \((\varrho, I, d)\) with \(0 \le d \le m-1\). For a row \(r := (\ell, \Lambda)\), set
\[b_{r, I} := \begin{cases} \delta_{\Lambda, I}, & \ell \le g+1,\\ \alpha_{\Lambda, I}, & \ell > g+1. \end{cases}\]
Then Proposition~\ref{prop:modified_entry} gives the following column block indexed by \((\varrho, I, -)\):
\[
\left( \begin{array}{cccc|c}
  z_{i_1,\varrho} & z_{i_1,\varrho}\lbd_{i_1} & \cdots & z_{i_1,\varrho}\lbd_{i_1}^{m-1} & i_1\in I\\
  z_{i_2,\varrho} & z_{i_2,\varrho}\lbd_{i_2} & \cdots & z_{i_2,\varrho}\lbd_{i_2}^{m-1} & i_2\in I\\
  \vdots & \vdots && \vdots & \vdots\\
  b_{r, I}\,z_{\ell,\varrho} & b_{r, I}\,z_{\ell,\varrho}\lbd_\ell & \cdots & b_{r, I}\,z_{\ell,\varrho}\lbd_\ell^{m-1} & \ell\in\Lambda\\
  \vdots & \vdots & \ddots & \vdots & \vdots
\end{array} \right),
\]
where \(\Lambda\) denotes the remaining row indices appearing in this block. Thus, we have a generalized Vandermonde-type matrix.

Perform the following column operations from right to left:
\[C_d\longmapsto C_d-\lbd_{i_1}C_{d-1},\qquad d=m-1,\dots,1.\]
Then the block becomes
\[
\left( \begin{array}{cccc|c}
  z_{i_1,\varrho} & 0 & \cdots & 0 & i_1\in I\\
  b_{r, I}\,z_{\ell,\varrho} & b_{r, I}\,z_{\ell,\varrho}(\lbd_\ell-\lbd_{i_1}) & \cdots & b_{r, I}\,z_{\ell,\varrho}\lbd_\ell^{m-2}(\lbd_\ell-\lbd_{i_1}) & \ell\in\Lambda\\
  \vdots & \vdots & \ddots & \vdots & \vdots
\end{array} \right),
\]
so we may choose \(z_{i_1,\varrho}\) and remove the first row and column. Repeat the procedure; after \(s\) steps, the remaining block has the form {\footnotesize
\[
\left( \begin{array}{cccc|c}
  z_{i_{s+1},\varrho}\displaystyle \prod_{\mu=1}^{s}(\lbd_{i_{s+1}}-\lbd_{i_\mu}) & z_{i_{s+1},\varrho}\lbd_{i_{s+1}}\displaystyle \prod_{\mu=1}^{s}(\lbd_{i_{s+1}}-\lbd_{i_\mu}) & \dots & z_{i_{s+1},\varrho}\lbd_{i_{s+1}}^{m-s-1}\displaystyle \prod_{\mu=1}^{s}(\lbd_{i_{s+1}}-\lbd_{i_\mu})& i_{s+1}\in I\\
  b_{r, I}\,z_{\ell,\varrho}\displaystyle \prod_{\mu=1}^{s}(\lbd_\ell-\lbd_{i_\mu}) & b_{r, I}\,z_{\ell,\varrho}\lbd_\ell\displaystyle \prod_{\mu=1}^{s}(\lbd_\ell-\lbd_{i_\mu}) & \dots & b_{r, I}\,z_{\ell,\varrho}\lbd_\ell^{m-s-1}\displaystyle \prod_{\mu=1}^{s}(\lbd_\ell-\lbd_{i_\mu})& \ell\in\Lambda\\
  \vdots & \vdots & \ddots & \vdots & \vdots
\end{array} \right),
\]
} and we may continue the same right-to-left elimination process to pull out the nonzero scalar multiple \(P_s(\lbd_{i_{s+1}})z_{i_{s+1},\varrho}\).
\end{example}

The preceding examples suggest the first reduction strategy.

\begin{algorithm}\label{alg:detcom}
Apply the following procedure to \(\star\bW_{3,m}\):
\begin{enumerate}[label = (\roman*)]
  \item For each column index
  \[I=\{i_1<\dots<i_{2m}\}\subset\Om_{\mathrm{good}},\]
  select \(z_{i_j, 1}\) for the first half \(i_1,\dots,i_m\). The submatrix containing \(\{z_{i_j, 1}\}_{1 \le j \le m}\) forms a Vandermonde matrix as in Example~\ref{ex:rk3_vandermonde}. Remove all rows indexed by \((i_s,I)\) for \(1\le s\le m\) and all columns indexed by \((1,I,d)\).
  \item Replace by zero all entries having \(z_{\ell, 3}, \ell \in \Om_{\mathrm{head}}\) or \(z_{\ell, 2}, \ell \in \Om_{\mathrm{mid}}\).
  \item Factor out \(z_{\ell, 3}\) from each row and write
  \[y_\ell:=\frac{z_{\ell,2}}{z_{\ell,3}}\]
  if \(z_{\ell, 3} \ne 0\).
\end{enumerate}
\end{algorithm}

\begin{definition}\label{def:tildeW}
Let \(\widetilde{\bW}_{3,m}\) denote the matrix obtained from \(\star\bW_{3,m}\) after Steps~(i)--(iii).
\end{definition}

\begin{proposition}\label{prop:rk3_reduction}
If \(\det \widetilde{\bW}_{3,m}\ne 0\), then \(\det \star\bW_{3,m}\ne 0\).
\end{proposition}

\begin{proof}
The determinant \(\det \star \bW_{3,m}\) is a polynomial in \(\{z_{\ell, 1}, z_{\ell, 2}, z_{\ell, 3}\}\). In Step~(i), we extract the coefficient of the lexicographically maximal possible monomial in \(\{z_{\ell, 1}\}_{\ell \in \Om_{\mathrm{good}}}\).

We formally set \(y_\ell := z_{\ell, 2}/z_{\ell, 3}\). After Step (i), the determinant of the remaining matrix can be written as
\[\left(\prod_{\ell}z_{\ell, 3}^{r_\ell}\right)P\left(\frac{z_{\ell,2}}{z_{\ell,3}}\right) = \left(\prod_{\ell}z_{\ell, 3}^{r_\ell}\right)P\left(y_{\ell}\right)\]
for some polynomial \(P\), where \(r_\ell\) is the number of rows carrying the index \(\ell\). In \(P\), choose the highest-degree term in each \(y_{\ell}\) for \(\ell \in \Om_{\mathrm{head}}\) and the constant term in each \(y_{\ell}\) for \(\ell \in \Om_{\mathrm{mid}}\). Algorithm \ref{alg:detcom} gives the coefficient of this term, proving the claim.
\end{proof}

It remains to show that \(\det \widetilde{\bW}_{3,m}\) is nonzero.

\begin{theorem}\label{thm:rk3_main}
For every odd \(g\ge3\) and every \(1\le m\le(g+1)/2\), the polynomial \(\det \widetilde{\bW}_{3,m}\) is not identically zero. Consequently, \(X^{2g-2}\) admits an Ulrich bundle of rank \(3\cdot2^{g-3}\).
\end{theorem}

%%%%%%%%%%%%%%%%%%%%%%%%%%%%%%%%%%%%%%%%%%%%%%%%%%%%
\section{Proof of Theorem~\ref{thm:rk3_main}}\label{sec:prf_main_thm}

We first describe \(\widetilde{\bW}_{3,m}\) as a block matrix whose blocks are parametrized by index sets \(I\). A change of basis in the columns with \(y\)-variables then reduces the problem to the diagonal blocks described below.

\begin{figure}[!ht]
\[
\left( \begin{array}{cccc|cccc|c}
  (2,I,0) & (2,I,1) & \cdots & (2,I,m-1) & (3,I,0) & (3,I,1) & \cdots & (3,I,m-1) & \\
  \hline y_{i_{m+1}} & y_{i_{m+1}}\lbd_{i_{m+1}} & \cdots & y_{i_{m+1}}\lbd_{i_{m+1}}^{m-1}&&&&& i_{m+1}\in I\\
  \vdots & \vdots & \ddots & \vdots&&&&& \vdots\\
  y_{i_p} & y_{i_p}\lbd_{i_p} & \cdots & y_{i_p}\lbd_{i_p}^{m-1}&&&&& i_p\in I\\
  \hline &&&&1 & \lbd_{i_q} & \cdots & \lbd_{i_q}^{m-1} & i_q\in I\\
  &&&&\vdots & \vdots & \ddots & \vdots & \vdots\\
  &&&&1 & \lbd_{i_{2m}} & \cdots & \lbd_{i_{2m}}^{m-1} & i_{2m}\in I\\
  \hline \alpha_{\Lambda,I}\,y_\ell & \alpha_{\Lambda,I}\,y_\ell\lbd_\ell & \cdots & \alpha_{\Lambda,I}\,y_\ell\lbd_\ell^{m-1} & \alpha_{\Lambda,I} & \alpha_{\Lambda,I}\lbd_\ell & \cdots & \alpha_{\Lambda,I}\lbd_\ell^{m-1} & \ell\in\Lambda\\
  \vdots & \vdots & \ddots & \vdots & \vdots & \vdots & \ddots & \vdots & \vdots
\end{array} \right).
\]
\caption{The \(I\)-th block of \(\widetilde{\bW}_{3,m}\).} \label{fig:Iblock}
\end{figure}

Let \(I=\{i_1<\cdots<i_{2m}\}\subset\Om_{\mathrm{good}}\) be of type \(p\), and set
\[q:=\max\{p,m\}+1.\]
The \(I\)-th column block of \(\widetilde{\bW}_{3,m}\) is displayed in Figure~\ref{fig:Iblock}. The block has the form
\[
\begin{pmatrix}
  \bV_{I}^y & 0\\
  0 & \bV_{I}^1\\
  \bB_{I}^y & \bB_{I}^1
\end{pmatrix}.
\]

The rows in the lower blocks \(\bB_{I}^y\) and \(\bB_{I}^1\) are indexed by bad rows \((\ell,\Lambda)\), hence
\[g+2\le \ell\le3(g+1)/2,\qquad\Lambda\subset\Om_{\mathrm{bad}},\]
and \(\alpha_{\Lambda,I}\) is defined as in Proposition~\ref{prop:modified_entry}.

We now collect the \(I\)-column blocks of the same type \(p\). After sorting the column indices \(I\subset\Om_{\mathrm{good}}\) by type, the matrix \(\widetilde{\bW}_{3,m}\) can be written as in Figure~\ref{fig:wholeW3m}, where
\[
\bV_p^y := \bigoplus_{\mathrm{type}\;I = p}\bV_{I}^y, \;\bV_p^1 := \bigoplus_{\mathrm{type}\;I = p} \bV_{I}^1,\; \bB_p^y := \bigsqcup_{\mathrm{type}\;I = p} \bB_{I}^y,\; \mbox{ and } \bB_p^1 := \bigsqcup_{\mathrm{type}\;I = p} \bB_{I}^1.
\]
Here \(\bA \oplus \bB\) means a block diagonal matrix with two blocks \(\bA\) and \(\bB\), while \(\bA \sqcup \bB\) is the augmentation \([\bA| \bB]\). From Definition \ref{def:type}, \(\max\{0, 2m - (g+1)/2\} \le p \le \min\{(g+1)/2, 2m\}\). Therefore, some of the blocks in Figure~\ref{fig:wholeW3m} may be empty; we omit these blocks from \(\widetilde{\bW}_{3,m}\).

\begin{figure}[!ht]
\[
\begin{pmatrix}
  &&&\bV_{m+1}^{y}\\
  &&&&\ddots\\
  &&&&&\bV_{2m}^{y}\\
  &&&&&&\bV_{0}^{1}\\
  &&&&&&&\ddots\\
  &&&&&&&&\bV_{2m}^{1}\\
  \bB_{0}^{y}&\cdots&\bB_{m}^{y}&\bB_{m+1}^{y}&\cdots&\bB_{2m}^{y}&\bB_{0}^{1}&\cdots&\bB_{2m}^{1}
\end{pmatrix},
\]
\caption{The matrix \(\widetilde{\bW}_{3,m}\)} \label{fig:wholeW3m}
\end{figure}

Apply the basis change from the \(\Om_{\mathrm{good}}\)-basis to the \(\Om_{\mathrm{bad}}\)-basis to the \(y\)-part columns, and write
\[\overline e_I=\sum_{\substack{J\subset\Om_{\mathrm{bad}}\\|J|=2m}}\widetilde{\alpha}_{I,J}\,\overline e_J,\qquad I\subset\Om_{\mathrm{good}}.\]
Then the blocks \(\bV_{m+1}^y,\dots,\bV_{2m}^y\) are converted into an upper triangular matrix whose diagonal blocks are \(\widetilde{\bB}_{m+1}^y,\dots,\widetilde{\bB}_{2m}^y\). Meanwhile, the blocks \(\bB_0^y,\dots,\bB_m^y\) are transformed into Vandermonde matrices \(\widetilde{\bV}_0^y, \dots, \widetilde{\bV}_m^y\). The constant part is unchanged. After reordering rows and columns, \(\widetilde{\bW}_{3,m}\) becomes the matrix in Figure~\ref{fig:afterchangeofbasis}, where \(\bB_{i,j}^1\) denotes the block of \(\bB_i^1\) consisting of rows of type \(j\). A further rearrangement gives the block upper bidiagonal matrix in Figure~\ref{fig:rearrangement}. Its diagonal blocks are
\[
\bT_p := \begin{pmatrix} \widetilde{\bB}_{p}^y\\ \widetilde{\bV}_{p}^y \end{pmatrix}, \quad \bU := \begin{pmatrix}
  &&&\bV^1_m\\
  &&&&\ddots\\
  &&&&&\bV^1_{2m}\\
  \widetilde{\bV}^y_0&&&\bB^1_{m,0}&\cdots&\bB^1_{2m,0}\\
  &\ddots&&\vdots&\ddots&\vdots\\
  &&\widetilde{\bV}^y_m&\bB^1_{m,m}&\cdots&\bB^1_{2m,m}
\end{pmatrix}, \quad\bV^1_{p'},
\]
where \(m+1 \le p \le 2m\) and \(0 \le p' \le m-1\). Each \(\bV^1_{p'}\) is invertible, being a block diagonal matrix whose diagonal blocks are \(m\times m\) Vandermonde matrices. Lemmas \ref{lem:BpVp-invertible} and \ref{lem:big-invertible} show the invertibility of the other blocks and hence complete the proof of Theorem~\ref{thm:rk3_main}.

\begin{figure}[!ht]
{\small
\[
\left(\begin{smallmatrix}
  &&&\widetilde{\bB}_{m+1}^y&*&*\\
  &&&&\ddots&*\\
  &&&&&\widetilde{\bB}_{2m}^y\\
  &&&&&&\bV_0^1\\
  &&&&&&&\ddots\\
  &&&&&&&&\bV_{2m}^1\\
  \widetilde{\bV}_{0}^y&&&&&&\bB^1_{0,0}&\cdots&\bB^1_{2m,0}\\
  &\ddots&&&&&\vdots&\ddots&\vdots\\
  &&\widetilde{\bV}_m^y&&&&\bB^1_{0,m}&\cdots&\bB^1_{2m,m}\\
  &&&\widetilde{\bV}_{m+1}^y&&&\bB^1_{0,m+1}&\cdots&\bB^1_{2m,m+1}\\
  &&&&\ddots&&\vdots&\ddots&\vdots\\
  &&&&&\widetilde{\bV}_{2m}^y&\bB^1_{0,2m}&\cdots&\bB^1_{2m,2m}
\end{smallmatrix}\right)
\]
} \caption{\(\widetilde{\bW}_{3,m}\) after a change of basis.} \label{fig:afterchangeofbasis}
\end{figure}

\begin{figure}[!ht]
{\scriptsize
\[
\left(\begin{array}{cccc|cccccc|ccc}
  \widetilde{\bB}_{m+1}^y&*&*&* &&&&&&&\\
  \widetilde{\bV}_{m+1}^y&&&&&&&\bB_{m,m+1}^1&\cdots&\bB_{2m,m+1}^1&\bB_{0,m+1}^1&\cdots&\bB_{m-1,m+1}^1\\
  &\widetilde{\bB}_{m+2}^y&*&*&&&&&&\\
  &\widetilde{\bV}_{m+2}^y&&&&&&\bB_{m,m+2}^1&\cdots&\bB_{2m,m+2}^1&\bB_{0,m+2}^1&\cdots&\bB_{m-1,m+2}^1\\
  &&\ddots&*&&&&&&\\
  &&\ddots&&&&&\vdots&\ddots&\vdots&\vdots&\ddots&\vdots\\
  &&&\widetilde{\bB}_{2m}^y&&&&&&&\\
  &&&\widetilde{\bV}_{2m}^y&&&&\bB_{m,2m}^1&\cdots&\bB_{2m,2m}^1&\bB_{0,2m}^1&\cdots&\bB_{m-1,2m}^1\\
  \hline &&&&&&&\bV_{m}^1&&&\\
  &&&&&&&&\ddots&&&\\
  &&&&&&&&&\bV_{2m}^1\\
  &&&&\widetilde{\bV}_{0}^y&&&\bB_{m,0}^1&\cdots&\bB_{2m,0}^1&\bB_{0,0}^1&\cdots&\bB_{m-1,0}^1\\
  &&&&&\ddots&&\vdots&\ddots&\vdots&\vdots&\ddots&\vdots\\
  &&&&&&\widetilde{\bV}_{m}^y&\bB_{m,m}^1&\cdots&\bB_{2m,m}^1&\bB_{0,m}^1&\cdots&\bB_{m-1,m}^1\\
  \hline &&&&&&&&&&\bV_0^1\\
  &&&&&&&&&&&\ddots\\
  &&&&&&&&&&&&\bV_{m-1}^1
\end{array}\right),
\]
} \caption{\(\widetilde{\bW}_{3,m}\) after the rearrangement of rows and columns.} \label{fig:rearrangement}
\end{figure}

\begin{lemma}\label{lem:BpVp-invertible}
For each \(m+1\le p\le 2m\), any nonempty \(\bT_p\) is square and generically invertible.
\end{lemma}

\begin{proof}
Fix \(p\). Each \(p\)-type subset on \(\Om_{\mathrm{bad}}\) contributes \(m\) columns and \(2m-p\) rows on \(\widetilde{\bV}_{p}^{y}\), while each \(p\)-type subset on \(\Om_{\mathrm{good}}\) contributes \(p-m\) rows on \(\widetilde{\bB}_{p}^{y}\). Hence \(\bT_p\) is square.

Apply the same column elimination as in Example~\ref{ex:rk3_vandermonde}, using \(\widetilde{\bV}_{p}^{y}\) as the pivot block. Then the part coming from \(\widetilde{\bV}_{p}^{y}\) is eliminated completely, and the remaining block in \(\widetilde{\bB}_{p}^{y}\) is
\[
\begin{pmatrix} \widetilde{\alpha}_{J,I}\,\mathbf\Lambda_{I_0,I_1} \end{pmatrix}_{\substack{I\subset\Om_{\mathrm{bad}}\text{ of type }p\\J\subset\Om_{\mathrm{good}}\text{ of type }p}},\qquad\mathbf\Lambda_{I_0,I_1}:=\left(y_{j_s}\lbd_{j_s}^{\,t}\prod_{\mu=p+1}^{2m}(\lbd_{j_s}-\lbd_{i_\mu})\right)_{m+1\le s\le p,\;0\le t\le p-m-1}.
\]
Here we write
\[I=\{i_1<\cdots<i_{2m}\},\qquad J=\{j_1<\cdots<j_{2m}\},\]
so that
\[I_0:=I\cap\Om_{\mathrm{head}}=\{i_1,\dots,i_p\},\qquad I_1:=I\cap\Om_{\mathrm{tail}}=\{i_{p+1},\dots,i_{2m}\}.\]
In particular, \(\mathbf\Lambda_{I_0,I_1}\) is a Vandermonde matrix up to multiplying each row by the nonzero scalar
\[y_{j_s}\prod_{\mu=p+1}^{2m}(\lambda_{j_s}-\lambda_{i_\mu}),\]
and is therefore generically invertible.

Let \(\widetilde{\bA}\) be the matrix expressing the \(\Om_{\mathrm{mid}}\)-basis in terms of the \(\Om_{\mathrm{bad}}\)-basis, and let \(\widetilde a_{\bullet,\bullet}\) denote its minors, defined analogously to the original ones. By construction, \(\widetilde{\alpha}_{I,J}=0\) unless
\[I\cap\Om_{\mathrm{head}}=J\cap\Om_{\mathrm{head}}.\]
Hence, after sorting rows and columns according to the common head part \(I_0\), the whole matrix becomes block diagonal. For fixed \(I_0\), write
\[I=I_0\sqcup I_1,\qquad I_1\subset\Om_{\mathrm{tail}},\qquad J=I_0\sqcup J_1,\qquad J_1\subset\Om_{\mathrm{mid}},\]
with \(|I_1|=|J_1|=2m-p\). The corresponding diagonal block is
\[
\bigl(\widetilde{a}_{J_1,I_1}\mathbf\Lambda_{I_0,I_1}\bigr)_{J_1,I_1}=\bigl((\widetilde{a}_{J_1,I_1})_{J_1,I_1}\otimes \bI_{p-m}\bigr)\operatorname{diag}\bigl(\mathbf\Lambda_{I_0,I_1}\bigr)_{I_1}.
\]
The first factor is the \((2m-p)\)-th compound matrix of \(\widetilde{\bA}\), hence it is invertible by Lemma~\ref{lem:A_minors} and \cite[\S0.8.1]{HJ13}. The second factor is generically invertible by the previous paragraph. Therefore, every diagonal block is generically invertible, and the proof is complete.
\end{proof}

\begin{lemma}\label{lem:big-invertible}
The matrix \(\bU\) is generically invertible.
\end{lemma}

\begin{proof}
Consider the rank-two extremal balanced parabolic bundle
\[
\cH:= \left(\cO\left((1-g)/2\right)^{\oplus 2},\{L'_\ell\}_{\ell=1}^{2g+1}\right),\quad L'_\ell= \begin{cases}
  k\cdot\begin{pmatrix} 0&1 \end{pmatrix}, & (g+1)/2 <\ell\le g+1,\\[6pt]
  k\cdot\begin{pmatrix} y_\ell&1 \end{pmatrix}, & g+1<\ell\le 3(g+1)/2,\\[6pt]
  0, & \text{otherwise}.
\end{cases}
\]
We compare the following two truncations of the matrix describing \(\pHom(\cH, \cR) \subset \rHom(\underline{\cH}, \underline{\cR})\).
\begin{enumerate}
  \item The matrix \(\bU\) is obtained by indexing the lower rows and left columns by subsets of \(\Om_{\mathrm{bad}}\) of type \(\le m\), and the upper rows and right columns by subsets of \(\Om_{\mathrm{good}}\) of type \(\ge m\).
  \item Retain only the blocks indexed by subsets \(I\subset\Om_{\mathrm{mid}}\sqcup\Om_{\mathrm{tail}}\) satisfying \(|I\cap\Om_{\mathrm{mid}}|\le m\). By Lemma~\ref{lem:rank2_block_decomposition}, this truncation is block diagonal, with blocks \(\bV_{2,I}\) indexed by
  \[I=\{i_1<\ldots<i_{2m}\}\subset\Om_{\mathrm{mid}}\sqcup\Om_{\mathrm{tail}},\qquad|I\cap\Om_{\mathrm{mid}}|=p\le m,\]
  evaluated at
  \[z_{i_1,1}=\dots=z_{i_p,1}=0,\qquad z_{i_{p+1},1}=y_{i_{p+1}},\quad\dots,\quad z_{i_{2m},1}=y_{i_{2m}},\qquad z_{i_j,2}=1\quad(1\le j\le2m).\]
  Each such block is generically invertible, hence so is the whole matrix.
\end{enumerate}

It remains to identify the first truncation with the second. The change of basis from \(\Om_{\mathrm{mid}}\sqcup\Om_{\mathrm{tail}}\) to \(\Om_{\mathrm{bad}}\) is block triangular if the source basis is sorted by \(|I\cap\Om_{\mathrm{mid}}|\) and the target basis by type. The change of basis from \(\Om_{\mathrm{mid}} \sqcup \Om_{\mathrm{tail}}\) to \(\Om_{\mathrm{good}}\) is likewise block triangular if the source basis is sorted by \(|I\cap\Om_{\mathrm{tail}}|\) and the target basis by type. In both cases, the diagonal blocks are invertible, so the induced maps on the corresponding truncated subspaces are also invertible.

Since \(|I\cap\Om_{\mathrm{mid}}|\le m\) is equivalent to \(|I\cap\Om_{\mathrm{tail}}|\ge m\), the two truncated matrices differ only by invertible changes of basis and reindexing. Therefore \(\bU\) is generically invertible.
\end{proof}

%%%%%%%%%%%%%%%%%%%%%%%%%%%%%%%%%%%%%%%%
\footnotesize \printbibliography
\end{document}